\documentclass[12pt]{amsart}
\usepackage[utf8]{inputenc}
\usepackage[shortlabels]{enumitem}
\setlist[enumerate]{itemsep=0.5ex}
\usepackage{url}
\usepackage{hyperref}
\usepackage{pdfpages}
\usepackage[export]{adjustbox}

\usepackage{amsbsy}
\usepackage{amsfonts}
\usepackage{amsgen}
\usepackage{amsmath}
\usepackage{amsopn}
\usepackage{amssymb}
\usepackage{amstext}
\usepackage{amsthm}
\usepackage{amsxtra}
\usepackage{mathtools}
\usepackage{thmtools}
\usepackage{dsfont}
\usepackage{cleveref}
\usepackage{multicol}
\usepackage{longtable}
\usepackage{float}
\usepackage{capt-of}

\theoremstyle{definition}
\newtheorem{definition}{Definition}[section]
\newtheorem{example}[definition]{Example}
\newtheorem{remark}[definition]{Remark}

\theoremstyle{plain}
\newtheorem{proposition}[definition]{Proposition}
\newtheorem{lemma}[definition]{Lemma}
\newtheorem{theorem}[definition]{Theorem}
\newtheorem{corollary}[definition]{Corollary}

\newcommand{\N}{\mathbb{N}}

\newcommand{\C}{\mathbb{C}}

\newcommand{\CC}{\mathcal{C}}
\newcommand{\DD}{\mathcal{D}}
\newcommand{\OO}{\mathcal{O}}
\newcommand{\PP}{\mathcal{P}}
\newcommand{\ii}{{\boldsymbol{i}}}
\newcommand{\jj}{{\boldsymbol{j}}}
\newcommand{\nn}{{\boldsymbol{n}}}

\newcommand{\abs}[1]{\left\lvert#1\right\rvert}

\newcommand{\angles}[1]{\left\langle#1\right\rangle}

\usepackage{tikz}
\newcommand{\circled}[1]{\tikz[baseline=(char.base)]{\node[shape=circle,draw,inner sep=0em] (char) {#1};}}
\DeclareMathOperator{\otop}{\circled{\resizebox{!}{1ex}{$\top$}}}
\DeclareMathOperator{\obot}{\circled{\resizebox{!}{1ex}{$\bot$}}}

\newcommand{\colors}{{{\left\{\circ,\bullet\right\}}^*}}

\DeclareMathOperator{\id}{id}
\DeclareMathOperator{\Ob}{Ob}
\DeclareMathOperator{\Hom}{Hom}
\DeclareMathOperator{\Span}{span}
\DeclareMathOperator{\Flat}{Flat}
\DeclareMathOperator{\Perm}{Perm}

\DeclareMathOperator{\Rep}{Rep}

\usepackage{spatialpart}

\newpartitionstyle{2d}{
    aszim   = 0, 
    elev    = 0,
    aspectX = 0.6
}
\newpartitionstyle{3d}{
    aspectX = 0.8,
    aspectZ = 1.2
}

\def\partId{\partition[2d]{%
    \line{0}{0}{0}{0}{1}{0}
    \point{0}{0}{0}{white}
    \point{0}{1}{0}{white}
}}

\def\partIdB{\partition[2d]{%
    \line{0}{0}{0}{0}{1}{0}
    \point{0}{0}{0}{black}
    \point{0}{1}{0}{black}
}}

\def\partIdBW{\partition[2d]{%
    \line{0}{0}{0}{0}{1}{0}
    \point{0}{0}{0}{white}
    \point{0}{1}{0}{black}
}}

\def\partIdWBW{\partition[2d]{%
    \line{0}{0}{0}{0}{1}{0}
    \line{1}{0}{0}{1}{1}{0}
    \line{2}{0}{0}{2}{1}{0}
    \point{0}{0}{0}{white}
    \point{0}{1}{0}{white}
    \point{1}{0}{0}{black}
    \point{1}{1}{0}{black}
    \point{2}{0}{0}{white}
    \point{2}{1}{0}{white}
}}

\def\partPair{\partition[2d]{%
    \line{0}{0}{0}{0}{0.5}{0}
    \line{0}{0.5}{0}{1}{0.5}{0}
    \line{1}{0.5}{0}{1}{0}{0}
    \point{0}{0}{0}{white}
    \point{1}{0}{0}{white}
}}

\def\partPairWB{\partition[2d]{%
    \line{0}{0}{0}{0}{0.5}{0}
    \line{0}{0.5}{0}{1}{0.5}{0}
    \line{1}{0.5}{0}{1}{0}{0}
    \point{0}{0}{0}{white}
    \point{1}{0}{0}{black}
}}

\def\partPairBW{\partition[2d]{%
    \line{0}{0}{0}{0}{0.5}{0}
    \line{0}{0.5}{0}{1}{0.5}{0}
    \line{1}{0.5}{0}{1}{0}{0}
    \point{0}{0}{0}{black}
    \point{1}{0}{0}{white}
}}

\def\partPairUnflat{\partition[3d]{%
    \line{0}{0}{0}{0}{0.5}{0}
    \line{0}{0.5}{0}{0}{0.5}{1}
    \line{0}{0.5}{1}{0}{0}{1}
    \point{0}{0}{0}{white}
    \point{0}{0}{1}{white}
}}

\def\partPairUnflatIds{\partition[3d]{%
    \line{0}{0}{0}{0}{1}{0}
    \line{0}{1}{1}{1}{0}{1}
    \line{0}{0}{1}{0}{0.5}{1}
    \line{0}{0.5}{1}{1}{0.5}{0}
    \line{1}{0.5}{0}{1}{0}{0}
    \point{0}{0}{0}{white}
    \point{0}{0}{1}{white}
    \point{1}{0}{0}{white}
    \point{1}{0}{1}{white}
    \point{0}{1}{0}{white}
    \point{0}{1}{1}{white}
}}

\def\partSingles{\partition[2d]{%
    \line{0}{0}{0}{0}{0.5}{0}
    \line{1}{0}{0}{1}{0.5}{0}
    \point{0}{0}{0}{white}
    \point{1}{0}{0}{white}
}}

\def\partSinglesUnflat{\partition[3d]{%
    \line{0}{0}{0}{0}{0.5}{0}
    \line{0}{0}{1}{0}{0.5}{1}
    \point{0}{0}{0}{white}
    \point{0}{0}{1}{white}
}}

\def\partSinglesUnflatIds{\partition[3d]{%
    \line{0}{0}{0}{0}{1}{0}
    \line{0}{1}{1}{1}{0}{1}
    \line{0}{0}{1}{0}{0.5}{1}
    \line{1}{0}{0}{1}{0.5}{0}
    \point{0}{0}{0}{white}
    \point{0}{0}{1}{white}
    \point{1}{0}{0}{white}
    \point{1}{0}{1}{white}
    \point{0}{1}{0}{white}
    \point{0}{1}{1}{white}
}}

\def\partCross{\partition[2d]{%
    \line{0}{0}{0}{1}{1}{0}
    \line{0}{1}{0}{1}{0}{0}
    \point{0}{0}{0}{white}
    \point{0}{1}{0}{white}
    \point{1}{0}{0}{white}
    \point{1}{1}{0}{white}
}}

\def\partCrossUnflat{\partition[3d]{%
    \line{0}{0}{0}{0}{1}{1}
    \line{0}{1}{0}{0}{0}{1}
    \point{0}{0}{0}{white}
    \point{0}{1}{0}{white}
    \point{0}{0}{1}{white}
    \point{0}{1}{1}{white}
}}

\def\partCrossUnflatIds{\partition[3d]{%
    \line{1}{0}{0}{0}{1}{1}
    \line{1}{1}{0}{0}{0}{1}
    \line{0}{0}{0}{0}{1}{0}
    \line{1}{0}{1}{1}{1}{1}
    \point{0}{0}{0}{white}
    \point{0}{1}{0}{white}
    \point{0}{0}{1}{white}
    \point{0}{1}{1}{white}
    \point{1}{0}{0}{white}
    \point{1}{1}{0}{white}
    \point{1}{0}{1}{white}
    \point{1}{1}{1}{white}  
}}

\def\partFour{\partition[2d]{%
    \line{0}{0}{0}{0}{0.35}{0}
    \line{1}{0}{0}{1}{0.35}{0}
    \line{0}{1}{0}{0}{0.65}{0}
    \line{1}{1}{0}{1}{0.65}{0}
    \line{0}{0.35}{0}{1}{0.35}{0}
    \line{0}{0.65}{0}{1}{0.65}{0}
    \line{0.5}{0.35}{0}{0.5}{0.65}{0}
    \point{0}{0}{0}{white}
    \point{0}{1}{0}{white}
    \point{1}{0}{0}{white}
    \point{1}{1}{0}{white}
}}

\def\partFourUnflat{\partition[3d]{%
    \line{0}{0}{0}{0}{0.35}{0}
    \line{0}{0}{1}{0}{0.35}{1}
    \line{0}{1}{0}{0}{0.65}{0}
    \line{0}{1}{1}{0}{0.65}{1}
    \line{0}{0.35}{0}{0}{0.35}{1}
    \line{0}{0.65}{0}{0}{0.65}{1}
    \line{0}{0.35}{0.5}{0}{0.65}{0.5}
    \point{0}{0}{0}{white}
    \point{0}{1}{0}{white}
    \point{0}{0}{1}{white}
    \point{0}{1}{1}{white}
}}

\def\partFourUnflatIds{\partition[3d]{%
    \line{1}{0}{0}{1}{0.35}{0}
    \line{0}{0}{1}{0}{0.35}{1}
    \line{1}{1}{0}{1}{0.65}{0}
    \line{0}{1}{1}{0}{0.65}{1}
    \line{1}{0.35}{0}{0}{0.35}{1}
    \line{1}{0.65}{0}{0}{0.65}{1}
    \line{0.5}{0.35}{0.5}{0.5}{0.65}{0.5}
    \line{0}{0}{0}{0}{1}{0}
    \line{1}{0}{1}{1}{1}{1}
    \point{0}{0}{0}{white}
    \point{0}{1}{0}{white}
    \point{0}{0}{1}{white}
    \point{0}{1}{1}{white}
    \point{1}{0}{0}{white}
    \point{1}{1}{0}{white}
    \point{1}{0}{1}{white}
    \point{1}{1}{1}{white}
}}

\def\partCrossSingles{\partition[2d]{%
    \line{0}{1}{0}{1}{0}{0}
    \line{0}{0}{0}{0}{0.35}{0}
    \line{1}{1}{0}{1}{0.65}{0}
    \point{0}{0}{0}{white}
    \point{0}{1}{0}{white}
    \point{1}{0}{0}{white}
    \point{1}{1}{0}{white}
}}

\def\partCrossSinglesUnflat{\partition[3d]{%
    \line{0}{1}{0}{0}{0}{1}
    \line{0}{0}{0}{0}{0.35}{0}
    \line{0}{1}{1}{0}{0.65}{1}
    \point{0}{0}{0}{white}
    \point{0}{1}{0}{white}
    \point{0}{0}{1}{white}
    \point{0}{1}{1}{white}
}}

\def\partCrossSinglesUnflatIds{\partition[3d]{%
    \line{0}{0}{0}{0}{1}{0}
    \line{0}{1}{1}{1}{0}{0}
    \line{0}{0}{1}{0}{0.35}{1}
    \line{1}{1}{0}{1}{0.65}{0}
    \line{1}{0}{1}{1}{1}{1}
    \point{0}{0}{0}{white}
    \point{0}{1}{0}{white}
    \point{0}{0}{1}{white}
    \point{0}{1}{1}{white}
    \point{1}{0}{0}{white}
    \point{1}{1}{0}{white}
    \point{1}{0}{1}{white}
    \point{1}{1}{1}{white}
}}

\def\partTrippleCross{\partition[2d]{%
    \line{1}{0}{0}{1}{1}{0}
    \line{0}{1}{0}{2}{0}{0}
    \line{0}{0}{0}{0}{0.5}{0}
    \line{3}{0}{0}{3}{0.5}{0}
    \line{0}{0.5}{0}{3}{0.5}{0}
    \point{0}{0}{0}{white}
    \point{0}{1}{0}{white}
    \point{1}{0}{0}{white}
    \point{1}{1}{0}{white}
    \point{2}{0}{0}{white}
    \point{3}{0}{0}{white}
}}

\def\partTrippleCrossUnflat{\partition[3d]{%
    \line{0}{0}{1}{0}{1}{1}
    \line{0}{0}{0}{0}{0.5}{0}
    \line{1}{0}{1}{1}{0.5}{1}
    \line{0}{0.5}{0}{1}{0.5}{1}
    \point{0}{0}{1}{white}
    \line{0}{1}{0}{1}{0}{0}
    \point{0}{0}{0}{white}
    \point{0}{1}{0}{white}
    \point{0}{1}{1}{white}
    \point{1}{0}{0}{white}
    \point{1}{0}{1}{white}
}}

\def\partTrippleCrossUnflatIds{\partition[3d]{%
    \line{0}{0}{0}{0}{1}{0}
    \line{0}{0}{1}{0}{0.5}{1}
    \line{0}{1}{1}{1}{0}{1}
    \line{1}{0}{0}{1}{1}{0}
    \line{1}{1}{1}{2}{0}{1}
    \line{2}{0}{0}{2}{0.5}{0}
    \line{0}{0.5}{1}{2}{0.5}{0}
    \point{0}{0}{0}{white}
    \point{0}{1}{0}{white}
    \point{1}{0}{0}{white}
    \point{1}{1}{0}{white}
    \point{2}{0}{0}{white}
    \point{0}{0}{1}{white}
    \point{0}{1}{1}{white}
    \point{1}{0}{1}{white}
    \point{1}{1}{1}{white}
    \point{2}{0}{1}{white}
}}

\def\partChair{\partition[2d]{%
    \line{0}{0}{0}{0}{1}{0}
    \line{1}{0}{0}{1}{0.5}{0}
    \line{1}{0.5}{0}{0}{0.5}{0}
    \point{0}{0}{0}{white}
    \point{0}{1}{0}{white}
    \point{1}{0}{0}{white}
}}

\begin{document}

\title{Projective Versions of Spatial Partition Quantum Groups}
\author{Nicolas Faroß}
\address{Saarland University, Fachbereich Mathematik, Postfach 151150, 66041 Saarbr\"ucken, Germany}
\email{faross@math.uni-sb.de}
\date{\today}
\thanks{The author thanks his supervisor Moritz Weber for many helpful comments and suggestions.
This article is part of the author's PhD thesis.
This work is a contribution to the SFB-TRR 195.}

\begin{abstract}
\noindent 
We generalize categories of spatial partitions in the sense of Cébron-Weber by introducing 
new base partitions. This allows us to construct additional examples of free orthogonal quantum groups
but yields the same class of spatial partition quantum groups as before.
Further, we use these new base partitions to show that the class of spatial partition quantum groups
is closed under taking projective versions and in particular contains the projective version of all easy quantum groups.
As an application, we determine the quantum groups corresponding to the categories of all spatial pair partitions
and give explicit descriptions of the projective versions of easy quantum groups in terms of spatial partitions.
\end{abstract}

\maketitle

\section{Introduction}\label{sec:intro}

Spatial partition quantum groups were first introduced
by Cébron-Weber in~\cite{cebron16} and are examples of compact quantum 
groups in the sense of Woronowicz~\cite{woronowicz87,woronowicz91}.
They generalize easy quantum groups by Banica-Speicher~\cite{banica09}
and are obtained by replacing two-dimensional partitions with three-dimensional spatial partitions.

A spatial partition on $m$ levels $p \in \PP^{(m)}$ consists of $k \cdot m$ upper points and $\ell \cdot m$ 
lower points that are partitioned into disjoint subsets by lines. Further, we allow both 
upper and lower points to be colored uniformly along the levels. For example, we have
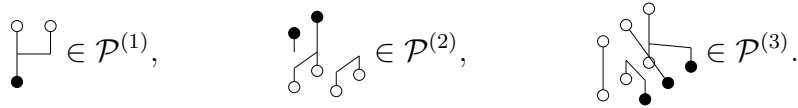
\[
  {\partition[2d]{%
    \line{0}{0}{0}{0}{1}{0}
    \line{1}{1}{1}{1}{0.5}{0}
    \line{1}{0.5}{0}{0}{0.5}{0}
    \point{0}{0}{0}{black}
    \point{0}{1}{0}{white}
    \point{1}{1}{0}{white}
  }} \in \PP^{(1)},
  \qquad\qquad
  {\partition[3d]{%
    \line{0}{0.65}{0}{0}{1}{0}
    \line{0}{0}{0}{0}{0.35}{0}
    \line{0}{0.35}{0}{0}{0.35}{1}
    \line{0}{0}{1}{0}{1}{1}
    \line{1}{0}{1}{1}{0.35}{1}
    \line{1}{0}{0}{1}{0.35}{0}
    \line{1}{0.35}{1}{1}{0.35}{0}
    \point{0}{0}{0}{white}
    \point{0}{0}{1}{white}
    \point{1}{0}{0}{white}
    \point{1}{0}{1}{white}
    \point{0}{1}{0}{black}
    \point{0}{1}{1}{black}}} \in \PP^{(2)},
  \qquad\qquad
  {\partition[3d]{%
  \line{0}{0}{0}{0}{1}{0}
  \line{0}{0}{1}{0}{0.35}{1}
  \line{1}{0}{0}{1}{0.35}{0}
  \line{0}{0.35}{1}{1}{0.35}{0}
  \line{0}{0}{2}{0}{1}{2}
  \line{1}{0}{2}{1}{0.35}{2}
  \line{0}{0.35}{2}{1}{0.35}{2}
  \point{0}{1}{0}{white}
  \point{0}{1}{2}{white}
  \point{0}{0}{0}{white}
  \point{0}{0}{1}{white}
  \point{0}{0}{2}{white}
  \point{1}{0}{0}{black}
  \point{1}{0}{1}{black}
  \point{1}{0}{2}{black}
  \line{0}{1}{1}{1}{0}{1}
  \point{0}{1}{1}{white}}} \in \PP^{(3)}.
\]

Given spatial partitions on the same number of levels, we can construct new spatial partitions by forming their tensor product, involution and composition. 
A category of spatial partitions in the sense of Cébron-Weber is a set of spatial partitions 
that is closed under these operations and contains the base partitions $\partId^{(m)}$, $\partIdB^{(m)}$
and $\partPairWB^{(m)}$, $\partPairBW^{(m)}$. Here, $p^{(m)} \in \PP^{(m)}$ denotes the spatial partition that is obtained by 
placing $m$ copies of the partition $p$ along each level.

By realizing spatial partitions as linear operators, categories of spatial partitions
give rise to concrete $C^*$-tensor categories, which then correspond to spatial partition quantum groups via Woronowicz Tannaka-Krein duality~\cite{woronowicz88}.
In the case of partitions on one level, these quantum groups are exactly unitary easy quantum groups~\cite{banica09, tarrago17}  
and include for example the free orthogonal and free unitary quantum groups~\cite{wang95}, the quantum permutation group~\cite{wang98} 
or the hyperoctahedral quantum group~\cite{banica07}.

Since easy quantum groups are based on partitions, they form a concrete class of quantum groups that can be studied and classified 
using combinatorics~\cite{freslon16a,gromada20b,freslon23, banica10a,weber13,raum14}. In the case of orthogonal easy quantum groups, the classification has been completed in~\cite{raum16}, whereas
the classification in the unitary case is still ongoing. See for 
example~\cite{tarrago18, freslon19, gromada18} or the more recent work by Mang~\cite{mang20,mang21c,mang21a,mang21b}. 

Spatial partition quantum groups have so far been studied in~\cite{cebron16,faross22}. In~\cite{cebron16}, Cébron-Weber
show that these quantum groups are closed under glued products, which implies that the 
class of spatial partition quantum groups is strictly larger than the class of easy quantum groups. Further, they 
provide a partial classification of categories of spatial pair partitions on two levels and 
discuss links to the quantum symmetries of finite-dimensional $C^*$-algebras.

In~\cite{faross22}, the author shows that the category $P^{(2)}_{2}$ of all spatial pair partitions on two levels 
gives rise to the classical projective orthogonal group $PO_n$ yielding a simpler example
of a non-easy spatial partition quantum group. Additionally, an explicit description
of the category of spatial partitions corresponding to the quantum symmetry group 
of $M_n(\C) \otimes \C^m$ is given.

\subsection*{Main results}
In~\cite{cebron16}, Cébron-Weber ask about a 
generalization of the base partitions $\partPairWB^{(m)}$ and $\partPairBW^{(m)}$
that still allows the construction of compact matrix quantum groups from 
categories of spatial partitions.
We answer this question by showing that the previous two base partitions
can be replaced with any pairs of spatial partitions $r$ and $s$ satisfying the conjugate equations
\[
  \left[ r^* \otimes \, {{\partId}^{(m)}} \right] \cdot \left[ \, {{\partId}^{(m)}} \otimes s\right] = {{\partId}^{(m)}},
  \qquad 
  \left[s^* \otimes \, {{\partIdB}^{(m)}} \right] \cdot \left[ \, {{\partIdB}^{(m)}} \otimes r \right] = {{\partIdB}^{(m)}}.
\]
In the case of one level, the partitions $r = \partPairWB$ and $s = \partPairBW$ are the only solutions to these equations.
However, in the case $m \geq 2$, we do not only obtain the previous base partitions $r = \partPairWB^{(m)}$ and $s = \partPairBW^{(m)}$, but also twisted versions like
\begin{align*}
  r = {\partition[3d]{
    \line{0}{0}{0}{0}{0.5}{0}
    \line{0}{0}{1}{0}{0.5}{1}
    \line{1}{0}{0}{1}{0.5}{0}
    \line{1}{0}{1}{1}{0.5}{1}
    \line{0}{0.5}{0}{1}{0.5}{1}
    \line{0}{0.5}{1}{1}{0.5}{0}
    \point{0}{0}{0}{white}
    \point{0}{0}{1}{white}
    \point{1}{0}{0}{black}
    \point{1}{0}{1}{black}
  }}, \quad s = {\partition[3d]{
    \line{0}{0}{0}{0}{0.5}{0}
    \line{0}{0}{1}{0}{0.5}{1}
    \line{1}{0}{0}{1}{0.5}{0}
    \line{1}{0}{1}{1}{0.5}{1}
    \line{0}{0.5}{0}{1}{0.5}{1}
    \line{0}{0.5}{1}{1}{0.5}{0}
    \point{0}{0}{0}{black}
    \point{0}{0}{1}{black}
    \point{1}{0}{0}{white}
    \point{1}{0}{1}{white}
  }}
  &&
  \text{and}
  &&
  r = {\partition[3d]{
    \line{0}{0}{2}{0}{0.5}{2}
    \line{0}{0}{1}{0}{0.5}{1}
    \line{0}{0}{0}{0}{0.5}{0}
    \line{1}{0}{0}{1}{0.5}{0}
    \line{1}{0}{1}{1}{0.5}{1}
    \line{1}{0}{2}{1}{0.5}{2}
    \point{0}{0}{2}{white}
    \line{0}{0.5}{2}{1}{0.5}{1}
    \line{0}{0.5}{1}{1}{0.5}{0}
    \line{0}{0.5}{0}{1}{0.5}{2}
    \point{0}{0}{0}{white}
    \point{0}{0}{1}{white}
    \point{1}{0}{0}{black}
    \point{1}{0}{1}{black}
    \point{1}{0}{2}{black}
  }}, \quad s = {\partition[3d]{
    \line{0}{0}{0}{0}{0.5}{0}
    \line{0}{0}{1}{0}{0.5}{1}
    \line{0}{0}{2}{0}{0.5}{2}
    \line{1}{0}{0}{1}{0.5}{0}
    \line{1}{0}{1}{1}{0.5}{1}
    \line{1}{0}{2}{1}{0.5}{2}
    \point{0}{0}{2}{black}
    \line{0}{0.5}{0}{1}{0.5}{1}
    \line{0}{0.5}{1}{1}{0.5}{2}
    \line{0}{0.5}{2}{1}{0.5}{0}
    \point{0}{0}{0}{black}
    \point{0}{0}{1}{black}
    \point{1}{0}{0}{white}
    \point{1}{0}{1}{white}
    \point{1}{0}{2}{white}
  }}.
\end{align*}
We refer to \Cref{prop:conj-eq-solutions} for a characterization of all possible solutions 
to these conjugate equations in the context of spatial partitions.

Using our new base partitions, we can now construct
additional examples of free orthogonal quantum groups in the 
sense of Van Daele-Wang~\cite{daele96}.
In the notation of~\cite{banica97}, these quantum groups are given by $O^+(F_\sigma)$ and with parameters
\[
  F_\sigma \colon {(\C^n)}^{\otimes m} \to {(\C^n)}^{\otimes m}, 
  \quad
  F_\sigma(e_{i_1} \otimes \dots \otimes e_{i_m})
  = e_{i_{\sigma^{-1}(1)}} \otimes \dots \otimes e_{i_{\sigma^{-1}(m)}}
\]
for all $n \in \N$ and $\sigma \in S_m$. See also \Cref{prop:universal-are-spatial} for more details and the precise statement in full generality. 

Although we can construct new examples of spatial partition quantum groups, it turns out that 
using our new base partitions, we obtain the same class of quantum groups as defined by Cébron-Weber.

\begingroup
\def\thetheorem{1}
\begin{theorem}[\Cref{corr:assume-base-partitions}]
Let $G$ be a spatial partition quantum group defined by any pair of spatial base partitions
satisfying the conjugate equations. Then $G$ is equivalent to a spatial
partition quantum group in the sense of Cébron-Weber defined by the base partitions
${\partPairWB}^{(m)}$ and ${\partPairBW}^{(m)}$.
\end{theorem}
\endgroup

More generally, we show in \Cref{sec:permuting-levels} that permuting the points of any category of spatial 
partitions along the levels leaves the corresponding quantum group invariant, which makes it possible 
to reduce any base partitions to the case of ${\partPairWB}^{(m)}$ and ${\partPairBW}^{(m)}$.
Still, our new base partitions
are useful on a combinatorial level and allow us to show 
that the class of spatial partition quantum groups is closed under taking 
projective versions.

Consider a compact matrix quantum group $G$ with fundamental representation $u$ and assume $\overline{u}$ is unitary.
Then its projective version $PG$ is the compact matrix quantum group defined by the representation $u \otop \overline{u}$.
If $G$ is a classical group, then $PG$ corresponds exactly to the quotient
\[
    PG =  G / (G \cap \{ \lambda I \mid \lambda \in \C \}).
\]
Further, in the case quantum groups, projective versions have for example been studied in~\cite{banica99,banica10a,banica10b,gromada22a}. 
In this context, our main result can be formulated as follows.

\begingroup
\def\thetheorem{2}
\begin{theorem}[\Cref{corr:proj-spatial-closed}]
Let $G$ be a spatial partition quantum group. 
Then $PG$ is again a spatial partition quantum group. Moreover, its category of spatial partitions
is given by $\Flat_{m,{\circ}{\bullet}}^{-1}(\CC)$, where $\CC \subseteq \PP^{(m)}$ is the category of spatial partitions corresponding to $G$ and $\Flat_{m,{\circ}{\bullet}}$ is the functor defined in \Cref{sec:spatial-functor}.
\end{theorem}
\endgroup

Note that the previous result does not only apply to projective versions $PG$ but to any 
compact matrix quantum group defined by a $\otop$-product of the fundamental representation $u$ and a unitary conjugate representation $u^\bullet$. 
See \Cref{sec:spatial-tensor-powers} for further details.

As an application, we consider the categories $\PP^{(m)}_2$ of all spatial pair partitions
on $m$ levels. In~\cite{banica09, faross22}, it is shown that $\PP^{(1)}_2$ corresponds to the classical orthogonal group $O_n$
and that $\PP^{(2)}_2$ corresponds to its projective version $PO_n$.
Using the previous theorem, we are now able to generalize these results in \Cref{sec:P2} to all $m \in \N$ and obtain
\[
  \PP_2^{(m)} \quad \longleftrightarrow \quad 
  \begin{cases}
    PO_n & \text{if $m$ is even}, \\
    O_n & \text{if $m$ is odd}.
  \end{cases}
\]

Finally, we consider projective versions of easy quantum groups in \Cref{sec:proj-easy-qg}. Since easy quantum groups are a subclass 
of spatial partition quantum groups, our main result implies that their projective versions 
are again spatial partition quantum groups. Using a result of Gromada~\cite{gromada22a}, we 
then derive sets of spatial partitions generating the categories of the projective versions of 
orthogonal easy quantum groups with degree of reflection two. This allows us to
describe these quantum groups explicitly as universal $C^*$-algebras defined by finite sets of relations.

\subsection*{Overview}

We begin in \Cref{sec:prelim} with some preliminaries
about compact matrix quantum groups, their representation categories and Woronowicz Tannaka-Krein duality. 
Then, we consider the combinatorics of spatial partitions in \Cref{sec:comb}
and introduce our new base partitions. In particular, we characterize all possible base partitions and introduce two functors between categories
of spatial partitions that are used to prove our main results.

In \Cref{sec:spatial-partition-QGs}, we show how to construct compact matrix quantum 
groups from categories of spatial partitions containing our new base partitions and give a description of 
these quantum groups in terms of generators and relations. This allows us
to construct the free orthogonal quantum groups $O^+(F_\sigma)$, before we show that our new base partitions yield the 
same class of spatial partition quantum groups as defined by Cébron-Weber.

Finally, we prove in \Cref{sec:proj-version} that the class of spatial partition quantum groups is closed under 
taking projective versions. Further, we determine the quantum groups corresponding to the categories $P_2^{(m)}$
of spatial pair partitions and give explicit descriptions of the 
projective versions of some orthogonal easy quantum groups in terms of spatial partitions.

\section{Preliminaries}\label{sec:prelim}

\subsection{Notation}\label{sec:notation}

We begin by introducing some notations and conventions
that will be used throughout the rest of the paper.

Let $\circ$ and $\bullet$ be two colors. Then we denote with
\[
  \colors := \{ 1, \circ, \bullet, {\circ}{\circ}, {\circ}{\bullet}, {\bullet}{\circ}, {\bullet}{\bullet}, {\circ}{\circ}{\circ}, \ldots \}
\]
the free monoid generated by $\{\circ, \bullet\}$. It is given
by the set of all finite words over $\{\circ, \bullet\}$ with concatenation as multiplication and the 
empty word  $1$ as the identity element. 
If $x \in \colors$, then we denote with $\abs{x}$ the length of the word $x$ and with 
$x_i \in \{\circ, \bullet\}$ the $i$-th color.
Further, let $\overline{\,\cdot\,} \colon \colors \to \colors$ be the anti-homomorphism 
defined by 
\[
  \overline{\circ} = \bullet, \qquad
  \overline{\bullet} = \circ, \qquad
  \overline{x \cdot y} = \overline{y} \cdot \overline{x}
  \quad 
  \forall x, y \in \colors.
\]

Next, consider a finite-dimensional Hilbert space $V$.
Then we denote with $\overline{V}$ its conjugate Hilbert space
where the conjugate of a vector $v \in V$ is denoted by $\overline{v} \in \overline{V}$.
Using colors, we additionally define $V^\circ := V$, $V^\bullet := \overline{V}$
and $v^\circ := v$, $v^\bullet := \overline{v}$ for all $v \in V$. Further, we extend this notation to tensor products by defining 
$V^{\otimes x} := V^{x_1} \otimes \dots \otimes V^{x_k}$ for all $x \in \colors$ with $k := \abs{x}$.

Let $n \in \N$. Then we define $[n] := \{1, \ldots, n\}$
and more generally $[\nn] := [n_1] \times \cdots \times [n_m]$
for all $\nn := (n_1, \dots, n_m) \in \N^m$.
Additionally, we introduce the Hilbert spaces $\C^\nn := \C^{n_1} \otimes \dots \otimes \C^{n_m}$
with canonical bases given by $e_\ii := e_{i_1} \otimes \dots \otimes e_{i_m}$
for all $\ii := (i_1, \dots, i_m) \in [\nn]$. 

Next, consider two finite-dimensional Hilbert spaces $V$ and $W$. Then $B(V, W)$ denotes 
the vector space of all linear operators $T \colon V \to W$ and we define $B(V) := B(V, V)$.
Given an orthonormal basis of $V$ indexed by $J$ and an orthonormal basis of $W$ indexed by $I$, 
we identify elements of $B(V, W)$ with matrices $(T^i_j)_{i \in I, j \in J}$.
Moreover, if $T \colon V \to W$ is a linear operator, then its conjugate
$\overline{T} \colon \overline{V} \to \overline{W}$ is given by 
$\overline{T}(\overline{v}) = \overline{T(v)}$ for all $v \in V$.
Using colors, we additionally write $T^\circ := T, T^\bullet := \overline{T}$
and $T^{\otimes x} := T^{x_1} \otimes \dots \otimes T^{x_{k}}$ for all $x \in \colors$ with $k := \abs{x}$.

Throughout the rest of the paper, we will use basic facts about $C^*$-algebras, their tensor products and universal $C^*$-algebras. For more information on 
these topics, we refer to~\cite{blackadar06}. In particular, we will denote with $\otimes$ the minimal tensor product of $C^*$-algebras.

Let $A$ be a unital $C^*$-algebra. Then we identify elements $u \in B(V, W) \otimes A$
with $A$-valued matrices $(u^i_j)_{i\in I,j\in J}$ with respect to orthonormal bases of $V$ and $W$ index by $J$ and $I$.
Moreover, we embed $B(V, W)$ into $B(V, W) \otimes A$ via $T \mapsto T \otimes 1$
and define the anti-linear map $\overline{\, \cdot \,} \colon B(V, W) \otimes A \to B(\overline{V}, \overline{W}) \otimes A$ by 
\[
  \overline{(T \otimes a)} := \overline{T} \otimes a^*
  \quad 
  \forall T \in B(V, W), \, a \in A.
\]
With respect to the conjugate bases of $V$ and $W$, the matrix coefficients of $\overline{u}$ are then given by $\overline{u}^i_j = {(u^i_j)}^*$.

Finally, we introduce the Woronowicz tensor products $\obot$ and $\otop$ from~\cite{woronowicz87}.
These are the bilinear operators defined by
\begin{alignat*}{3}
  (T \otimes a) \obot (S \otimes b) &:= TS \otimes a \otimes b   &\quad \in B(V) \otimes A \otimes A, \\
  (T \otimes a) \otop (S \otimes b) &:= (T \otimes S) \otimes ab &\quad \in B(V \otimes W) \otimes A
\end{alignat*}
for all $T \otimes a \in B(V) \otimes A$ and $S \otimes b \in B(V) \otimes A$ or $S \otimes b \in B(W) \otimes A$ respectively.
The matrix coefficients of both operators are given by
\[
  {(u \obot v)}^i_j = \sum_k u^i_k \otimes v^k_j,
  \qquad
  {(u \otop v)}^{i_1 i_2}_{j_1 j_2} = u^{i_1}_{j_1} \cdot v^{i_2}_{j_2}
\]
with respect to orthonormal bases of $V$, $W$ and the induced basis of $V \otimes W$. 

\subsection{Compact matrix quantum groups}

Compact quantum groups were first introduced by Woronowicz~\cite{woronowicz87,woronowicz91}
as a generalization of classical compact groups and provide the general framework for 
easy quantum groups and spatial partition quantum groups. 
In the following, we will restrict ourselves to compact matrix quantum groups, which 
are a class of compact quantum groups defined by a single unitary representation.
We refer to~\cite{weber17a, banica23} for a detailed introduction to compact matrix quantum groups 
and to~\cite{timmermann08, neshveyev13} for the general case of compact quantum groups.

\begin{definition}
Let $V$ be a finite-dimensional Hilbert space, $A$ be a unital $C^*$-algebra and $u \in B(V) \otimes A$.
Then $G := (A, u)$ is called a \textit{compact matrix quantum group} if 
\begin{enumerate}
\item $A$ is generated as $C^*$-algebra by the matrix coefficients
\[
  \left\{ (\varphi \otimes \id_A)(u) \mid \varphi \in B(V, \C) \right\},
\]
\item $u$ is unitary and $\overline{u}$ is invertible,
\item there exists a $*$-homomorphism $\Delta \colon A \to A \otimes A$, such that 
\[
  (\id_{B(V)} \otimes \Delta)(u) = u \obot u.
\]
\end{enumerate}
The element $u$ is called the \textit{fundamental representation} of $G$. 
Further, we denote the $C^*$-algebra $A$ with $C(G)$
and the dense $*$-algebra generated by 
the matrix coefficients with $\OO(G)$.
\end{definition}

Note that we have chosen a basis-independent definition of compact matrix quantum groups.
This allows us to keep track of the tensor product structure of the underlying Hilbert space $V$
and does not force us to identify $\C^n \otimes \C^n$ with $\C^{n^2}$ in an arbitrary way.

Examples of compact matrix quantum groups are Wang's free orthogonal and free unitary quantum groups~\cite{wang95}
or its deformations in the sense of Van Daele-Wang~\cite{daele96}.
Using the notation of Banica~\cite{banica97}, these quantum groups are defined as follows.

\begin{definition}\label{def:universal-qgs}
Let $\nn \in \N^m$, $F \in B(\C^{\nn})$ be invertible and denote with $\iota \colon \overline{\C^{\nn}} \to \C^{\nn}$ the linear isomorphism 
defined by $\iota(\overline{e_\ii}) = e_\ii$ for all $\ii \in [\nn]$. Define the universal unital $C^*$-algebras 
\begin{align*}
  A_o(F) &:= C^*( u_{\ii\jj} \mid \text{$u$ is unitary and $u = (F \iota) \, \overline{u} \, {(F \iota)}^{-1}$}), \\
  A_u(F) &:= C^*( u_{\ii\jj} \mid \text{$u$ and $(F \iota) \, \overline{u} \, (F \iota)^{-1}$ are unitary})
\end{align*} 
generated by the coefficients of a matrix $u := {(u^\ii_\jj)}_{\ii,\jj \in [\nn]}$. Then 
\[
  O^+(F) := (A_o(F), u), \qquad U^+(F) := (A_u(F), u)
\]
are the \textit{free orthogonal} and the \textit{free unitary} quantum groups
with parameter $F$. Moreover, we define $O^+_{\nn} := O^+(\id_{\C^{\nn}})$ and $U^+_{\nn} := U^+(\id_{\C^{\nn}})$.

\end{definition}

Additional examples of compact matrix quantum groups are the quantum permutation group $S_n^+$~\cite{wang98},
the hyperoctahedral quantum groups $H_n^+$~\cite{banica07} or more generally easy quantum groups~\cite{banica09} and spatial partition quantum groups~\cite{cebron16}. 
We introduce spatial partition quantum groups in \Cref{sec:spatial-partition-QGs}, which includes easy quantum groups 
as a special case. For a more detailed introduction to easy quantum groups, we refer to~\cite{weber16,weber17a}.

As in the case of classical groups, one can generalize the notion of 
subgroup and isomorphic groups to the setting of compact matrix quantum groups.

\begin{definition}
Let $G$ and $H$ be compact matrix quantum groups
with fundamental representations $u$ on $V$ and $v$ on $W$. Then
\begin{enumerate}
\item $H$ is a \textit{subgroup} of $G$ and we write $H \subseteq G$ if there exists a unitary $Q\colon W \to V$ and 
a unital $*$-homomorphism $\varphi \colon C(G) \to C(H)$ with 
\[
  \varphi(u) = Q v Q^{-1},
\]
\item $G$ and $H$ are \textit{isomorphic} and we write $G = H$ if $H \subseteq G$ and the previous $*$-homomorphism
$\varphi$ is a $*$-isomorphism,
\item $G$ and $H$ are \textit{isomorphic as compact quantum groups} if there exists a 
unital $*$-isomorphism $\varphi \colon C(G) \to C(H)$ with
\[
  (\varphi \otimes \varphi) \circ \Delta_G = \Delta_H \circ \varphi.
\]
\end{enumerate}
\end{definition}

Note that an isomorphism between compact quantum groups is more general
and does not respect the fundamental representation of a compact matrix quantum group.
In particular, it allows us to compare quantum groups with fundamental representations of different dimensions.

Finally, we introduce representations of compact
matrix quantum groups, which will be discussed in more detail in the following.

\begin{definition}
Let $G$ be a compact matrix quantum group. A \textit{representation} of $G$ on a finite-dimensional Hilbert space $V$ 
is an invertible element $v \in B(V) \otimes C(G)$ that satisfies
\[
  (\id_{B(V)} \otimes \Delta)(v) = v \obot v.
\]
\end{definition}

Consider a compact matrix quantum group  $G$ on a Hilbert space $V$. Then
its fundamental representation $u \in B(V) \otimes C(G)$ and the trivial representation $1 \in B(\C) \otimes C(G)$
are always representations.
Further, if $v$ and $w$ are representations, then both $\overline{v}$ and $v \otop w$ are 
again representations.

\subsection{Representation categories and rigidity}\label{sec:repr-categories}

Next, we recall that representations of a compact matrix quantum group have the 
structure of a tensor category, which allows us to formulate Woronowicz Tannaka-Krein duality in the next section.
We refer to~\cite{woronowicz88,timmermann08, neshveyev13, gromada20a} for more on the representation theory 
of compact matrix quantum groups and to~\cite{etingof16} for tensor categories in general.

We begin by introducing intertwiners between representations.

\begin{definition}
Let $G$ be a compact matrix quantum group with representations $v$ on $V$ and $w$ on $W$.
Then the space of \textit{intertwiners} between $v$ and $w$
is given by
\[
  \Hom(v, w) := \left\{ T \in B(V, W) \mid T v = w T \right\}.
\]
\end{definition}

Consider a compact matrix quantum group $G$. 
Then we can construct a category $\Rep(G)$ with objects given by finite-dimensional unitary representations
of $G$ and morphisms $\Hom(v, w)$ given by intertwiners between $v$ and $w$.
Additionally, we can use the $\otop$-operator and the adjoint $T^*$ of intertwiners to 
define a monoidal and a $*$-structure on this category. 
The following proposition summarizes its most important properties.

\begin{proposition}\label{prop:reprs-monoidal-cat}
The finite-dimensional unitary representations of a compact matrix quantum group 
are a concrete monoidal $*$-category in following the sense:
\begin{enumerate}
  \item $\id_V \in \Hom(v, v)$ for every representation $v$ on $V$.
  \item $\Hom(v, w)$ is a linear subspace of $B(V, W)$ for all representations $v$ on $V$ and $w$ on $W$.
  \item If $S \in \Hom(v, w)$ and $T \in \Hom(w, x)$, then $ST \in \Hom(v, x)$.
  \item If $T \in \Hom(v, w)$, then $T^* \in \Hom(w, v)$.
  \item If $S \in \Hom(v, w)$ and $T \in \Hom(x, y)$, then $S \otimes T \in \Hom(v \otop x, w \otop y)$.
\end{enumerate}
\end{proposition}

In addition to the previous properties, the representation category of a compact matrix quantum group is rigid
in the sense that every unitary representation has a conjugate representation.

\begin{definition}\label{def:conjugate-repr}
Let $G$ be a compact matrix quantum group with unitary representations $u^\circ$ on $V$ and $u^\bullet$
on $W$.
Then $u^\bullet$ is \textit{conjugate} to $u^\circ$ if there exist
intertwiners
$R \in \Hom(1, u^\circ \otop u^\bullet)$ and $S \in \Hom(1, u^\bullet \otop u^\circ)$ satisfying the conjugate equations
\[
  (R^* \otimes \id_{V}) \cdot (\id_{V} \otimes S) = \id_{V},
  \qquad
  (S^* \otimes \id_{W}) \cdot (\id_{W} \otimes R) = \id_{W}. 
\]
\end{definition}

Note that the representation $u^\bullet$ is also called a dual of $u^\circ$ and 
the intertwiners $R$ and $S$ are called duality morphisms. We refer to~\cite{etingof16} and~\cite{neshveyev13} for the definition of rigidity
in a more general context.

In~\cite{woronowicz88}, it is shown that the representation category
of a compact matrix quantum groups is generated by tensor powers of its fundamental representation $u$ and its 
conjugate representation. 
Thus, it will be enough to consider only the fundamental representation $u^\circ:= u$
with conjugate $u^\bullet$ in the following. 
Further, we introduce the notation
$
  u^x := u^{x_1} \otop \dots \otop u^{x_k}
$
for all $x \in \colors$ with $k := \abs{x} > 0$. In the case $x = 1$, we 
define $u^1 := 1$ as the trivial representation.

Note that the conjugate of a representation is not unique 
but is only determined up to unitary equivalence.
Furthermore, the representation $\overline{u^\circ}$ is in general not conjugate to $u^\circ$ 
since it is not necessarily unitary. However, we show in the following that any conjugate representation 
$u^\bullet$ is equivalent to $\overline{u^\circ}$ and this equivalence is 
defined by the solutions $R$ and $S$ of the conjugate equations.

\begin{proposition}\label{prop:u-bullet-impl-conj}
Consider a compact matrix quantum group with fundamental representation $u^\circ$ on $V$ and let $F \in B(\overline{V})$ be invertible.
If $u^\bullet := F \overline{u^\circ} F^{-1}$ is unitary, 
then $u^\circ$ and $u^\bullet$ are conjugate with duality morphisms $R$ and $S$ given by
\[
  R^{ij} = F^j_i, 
  \qquad 
  S^{ij} = {(\overline{F}^{-1})}^j_i
\]
with respect to any basis of $V$ and its conjugate basis.
\end{proposition}
\begin{proof}
A proof can be found in~\cite{gromada20a}. One checks directly that 
\begin{align*}
  u^\circ {(u^\circ)}^* &= 1 \ \Longleftrightarrow \ R \in \Hom(1, u^\circ \otop u^\bullet), \\
  u^\bullet {(u^\bullet)}^* &= 1 \ \Longleftrightarrow \ S \in \Hom(1, u^\bullet \otop u^\circ).
\end{align*}
Thus, $R$ and $S$ are intertwiners. Further, they satisfy the conjugate equations since
\[
  (R^* \otimes \id_{V}) \cdot (\id_{V} \otimes S) = \overline{F}^{-1} \overline{F},
  \qquad 
  (S^* \otimes \id_{\overline{V}}) \cdot (\id_{\overline{V}} \otimes R) = F F^{-1}.
\]
\end{proof}

\begin{proposition}\label{prop:u-black-formula}
Let $G$ be a compact matrix quantum group with unitary representation $u^\circ$ on $V$.
Assume $u^\bullet$ is a unitary representation on $\overline{V}$ that 
is conjugate to $u^\circ$ with duality morphisms $R$ and $S$.
Then $u^\bullet = F \overline{u^\circ} F^{-1}$, where $F \in B(\overline{V})$ is given by $F^i_j = R^{ji}$
with respect to any basis of $V$ and its conjugate basis.
\end{proposition}
\begin{proof}
The representation $\overline{u^\circ}$ is equivalent to 
a unitary representation that is conjugate to $u^\circ$, see \cite[Example 2.2.3]{neshveyev13}. 
Further, \cite[Proposition 2.2.4]{neshveyev13} 
shows that conjugates are unique up to equivalence. Hence, there exists an invertible $\widetilde{F} \in B(\overline{V})$
such that $u^\bullet = \widetilde{F} \overline{u^\circ} \widetilde{F}^{-1}$.
The previous proposition then implies that 
\[
  \widetilde{R} \in \Hom(1, u^\circ \otop u^\bullet), \qquad \widetilde{S} \in \Hom(1, u^\bullet \otop u^\circ),
\]
where $\widetilde{R}$ and $\widetilde{S}$ are the solutions to the conjugate equations defined by $\widetilde{F}$.
As in~\cite[Theorem 3.4.6]{gromada18}, %
we compute 
\[
  (\widetilde{S}^* \otimes \id_{\overline{V}}) \cdot (\id_{\overline{V}} \otimes R)
  = F \widetilde{F}^{-1} \in \Hom(u^\bullet, u^\bullet),
\]
which yields
\[
  u^\bullet = F \widetilde{F}^{-1} u^\bullet \widetilde{F} F^{-1} = F \overline{u^\circ} F^{-1}.
\]
\end{proof}

\subsection{Woronowicz Tannaka-Krein duality}\label{sec:tannaka-krein}

In the case of classical groups, Tannaka-Krein duality~\cite{doplicher98} allows the reconstruction of a compact group 
from an abstract category of representations. A similar result for compact matrix quantum groups 
was first proven by Woronowicz in~\cite{woronowicz88}.

There exist multiple proofs of Woronowicz Tannaka-Krein duality at various levels of abstraction, see for example~\cite{woronowicz88,neshveyev13,malacarne16,gromada20a}.
In the following, we choose the approach of Gromada~\cite{gromada20a},
which is most suitable for the setting of categories defined by partitions.
Thus, we begin by introducing the notion of an abstract two-colored representation category that captures the properties of the 
representation category of a compact matrix quantum group discussed in the previous section.

\begin{definition}\label{def:rep-category}
Let $V$ be a finite-dimensional Hilbert space. 
A \textit{two-colored representation category} $\CC$ is a collection of linear subspaces
$\CC(x, y) \subseteq B(V^{\otimes x}, V^{\otimes y})$ for all $x, y \in \colors$ satisfying the following
properties:
\begin{enumerate}
\item $\id_{V^{\otimes x}} \in \CC(x, x)$ for all $x \in \colors$.
\item If $S \in \CC(x, y)$ and $T \in \CC(y, z)$, then $ST \in \CC(x, z)$.
\item If $T \in \CC(x, y)$, then $T^* \in \CC(y, x)$. 
\item If $S \in \CC(w, x)$ and $T \in \CC(y, z)$, then $S \otimes T \in \CC(wy, xz)$.
\item There exist $R \in \CC(1, \circ\bullet)$ and $S \in \CC(1, \bullet\circ)$ satisfying the conjugate equations
\[
  (R^* \otimes \id_{V}) \cdot (\id_{V} \otimes S) = \id_{V},
  \qquad
  (S^* \otimes \id_{\overline{V}}) \cdot (\id_{\overline{V}} \otimes R) = \id_{\overline{V}}. 
\]
\end{enumerate}
\end{definition}

Woronowicz Tannaka-Krein duality now states that given any two-colored representation category $\CC$,
there exists a unique compact matrix quantum group whose representation category $\Rep(G)$ is given 
by $\CC$.

\begin{theorem}[Woronowicz Tannaka-Krein duality]\label{thm:tannaka-krein-duality}
Let $\CC$ be a two-colored representation category on a Hilbert space $V$.
Then there exists a unique compact matrix quantum group $G$ with 
fundamental representation $u^\circ$ on $V$
and unitary representation $u^\bullet$ on $\overline{V}$, such that
\[
  \Hom(u^{x}, u^{y}) = \CC(x, y) \quad \forall x, y \in \colors.
\]
\end{theorem}

First, let us comment on the uniqueness of the unitary representation $u^\bullet$ in the previous theorem.
Since $\CC$ contains a pair of morphisms $R$ and $S$ satisfying the conjugate equations, the representations $u^\circ$ and $u^\bullet$ are conjugate.
\Cref{prop:u-black-formula} then implies that $u^\bullet = F \overline{u^\circ} F^{-1}$ with $F^i_j = R^{ji}$. Thus, the representation $u^\bullet$ is 
uniquely determined by the category $\CC$.

Next, consider the uniqueness of the quantum group $G$. The proof 
of Woronowicz Tannaka-Krein duality in~\cite{gromada20a} shows that all relations
of the dense $*$-algebra $\OO(G)$ are spanned by the intertwiner relations
\[
  T u^{x} = u^{y} T  \qquad \forall x, y \in \colors, \ T \in \CC(x, y).
\]  
Hence, $\OO(G)$ is uniquely determined by the category $\CC$ up to $*$-isomorphism. 
In contrast, the $C^*$-algebra $C(G)$ is not unique but it is always possible to choose $C(G)$ as the maximal $C^*$-completion of $\OO(G)$.
In this case, $G$ is uniquely defined by the following universal property.

\begin{proposition}
Let $G$ be the compact matrix quantum group in \Cref{thm:tannaka-krein-duality}
and $H$ be a compact matrix quantum group with fundamental representation $w^\circ$ on $W$
and unitary representation $w^\bullet$ on $\overline{W}$. If
there exists a unitary $Q\colon V \to W$, such that 
\[
  \Hom(u^{x}, u^{y}) \subseteq {(Q^{-1})}^{\otimes y} \cdot \Hom(w^x, w^y) \cdot Q^{\otimes x}
  \quad \forall x, y\in \colors,
\]
then $H$ is a subgroup $G$ via $u^\circ \mapsto Q^{-1}w^\circ Q$.
\end{proposition}

Motivated by Woronowicz Tannaka-Krein duality and the previous proposition, we introduce
the following notion of equivalence between compact matrix quantum groups.

\begin{definition}\label{def:equivalence}
Let $G$ be a compact matrix quantum group 
with fundamental representation $u^\circ$ and conjugate representation $u^\bullet$.
A compact matrix quantum group $H$ with fundamental representation $w^\circ$ on $W$
and unitary representation $w^\bullet$ on $\overline{W}$ is called \textit{equivalent}
to $G$ if there exists a unitary $Q\colon V \to W$ such that
\[
  \Hom(u^{x}, u^{y}) = {(Q^{-1})}^{\otimes y} \cdot \Hom(w^x, w^y) \cdot Q^{\otimes x}
  \quad \forall x, y\in \colors.
\] 
\end{definition}

One can check that the previous definition is indeed
an equivalence relation between compact matrix quantum groups. Further, Woronowicz Tannaka-Krein duality
and the previous discussion show that two compact matrix quantum groups 
$G$ and $H$ are equivalent if and only if $\OO(G)$ and $\OO(H)$ are $*$-isomorphic.
Moreover, this isomorphism is given by 
$u = Q^{-1} w Q$, where $u$ and $w$ are the fundamental representations
of $G$ and $H$ respectively.

\section{Combinatorics of spatial partitions}\label{sec:comb}

In the following, we introduce spatial partitions, their categories and our new base partitions as purely combinatorial objects.
These will then be used again in \Cref{sec:spatial-partition-QGs} to define spatial partition quantum groups.
Further, we construct two functors between categories of spatial partitions that will be used to 
prove our main results in \Cref{sec:spatial-partition-QGs} and \Cref{sec:proj-version}.

\subsection{Spatial partitions}

We begin with spatial partitions, which are the main combinatorial objects to 
define categories of spatial partitions.
In contrast to~\cite{cebron16}, we will not focus on spatial partitions with only white points, 
but we will consider colored partitions as described in~\cite[Remark 2.7]{cebron16}.

\begin{definition}
Let $m \in \N$. A \textit{spatial partition} on 
$m$ \textit{levels}
is a tuple $(x, y, \{ B_i \})$, where $x, y \in \colors$ and 
$\{B_i\}$ is a decomposition of the points
\[
  \{1, \ldots, \abs{x} + \abs{y}\} \times \{1, \ldots, m\}
\]
into disjoint subsets called \textit{blocks}.
We call $x$ the \textit{upper colors} and $y$ the \textit{lower colors}
of the partition.
\end{definition}
Given a spatial partition, we can visualize it as a colored string-diagram as follows.
First, we draw an upper layer consisting of the points $(i, j)$ with $1 \leq i \leq \abs{x}$ 
and a lower layer consisting of the points $(i, j)$ with $\abs{x} < i \leq \abs{y}$, where 
in both layers $i$ increases from the left to the right and $j$ from the front to the back.
Then we connect all points in the same blocks by lines and assign each upper point 
$(i, j)$ the color $x_{i}$ and each lower point $(\abs{x} + i, j)$ the color $y_i$.

\begin{example}
  Consider the spatial partition on two levels with upper colors $x = {\circ}{\bullet}$, lower colors $y = {\circ}{\circ}{\bullet}$
  and blocks
  \[
    \{(1, 1), (3, 2)\}, \{(1, 2), (3, 1)\}, \{(2, 1), (4, 1)\}, \{(2, 2), (4, 2), (5, 2)\}, \{(5, 1)\}.
  \]
  Visualizing it as string-diagram yields
  \[%
  \partition[3d]{%
      %
      %
      \label{0}{1}{0}{{\Huge(1,1)}}
      \label{0}{1}{1}{{\Huge(1,2)}}
      \label{1}{1}{0}{{\Huge(2,1)}}
      \label{1}{1}{1}{{\Huge(2,2)}}
      \label{0}{0}{0}{{\Huge(3,1)}}
      \label{0}{0}{1}{{\Huge(3,2)}}
      \label{1}{0}{0}{{\Huge(4,1)}}
      \label{1}{0}{1}{{\Huge(4,2)}}
      \label{2}{0}{0}{{\Huge(5,1)}}
      \label{2}{0}{1}{{\Huge(5,2)}}
  }%
  \quad
  \longrightarrow
  \quad
  \partition[3d]{%
      %
      \line{0}{0}{0}{0}{1}{1}
      \line{0}{1}{0}{0}{0}{1}
      \line{1}{0}{0}{1}{1}{0}
      \line{1}{0}{1}{1}{1}{1}
      \line{2}{0}{0}{2}{0.35}{0}
      \line{2}{0}{1}{2}{0.35}{1}
      \line{1}{0.35}{1}{2}{0.35}{1}
      \label{0}{1}{0}{{\Huge(1,1)}}
      \label{0}{1}{1}{{\Huge(1,2)}}
      \label{1}{1}{0}{{\Huge(2,1)}}
      \label{1}{1}{1}{{\Huge(2,2)}}
      \label{0}{0}{0}{{\Huge(3,1)}}
      \label{0}{0}{1}{{\Huge(3,2)}}
      \label{1}{0}{0}{{\Huge(4,1)}}
      \label{1}{0}{1}{{\Huge(4,2)}}
      \label{2}{0}{0}{{\Huge(5,1)}}
      \label{2}{0}{1}{{\Huge(5,2)}}
  }%
  \quad
  \longrightarrow
  \quad
  \partition[3d]{%
      %
      \line{0}{0}{0}{0}{1}{1}
      \line{0}{1}{0}{0}{0}{1}
      \line{1}{0}{0}{1}{1}{0}
      \line{1}{0}{1}{1}{1}{1}
      \line{2}{0}{0}{2}{0.35}{0}
      \line{2}{0}{1}{2}{0.35}{1}
      \line{1}{0.35}{1}{2}{0.35}{1}
      \point{0}{1}{0}{white}
      \point{0}{1}{1}{white}
      \point{1}{1}{0}{black}
      \point{1}{1}{1}{black}
      \point{0}{0}{0}{white}
      \point{0}{0}{1}{white}
      \point{1}{0}{0}{white}
      \point{1}{0}{1}{white}
      \point{2}{0}{0}{black}
      \point{2}{0}{1}{black}
  }.%
  \]
\end{example}

Note that spatial partitions on one level with only white points are exactly
the partitions appearing for example in the definition of the Temperly-Liebe algebra~\cite{temperly71},
the Brauer algebra~\cite{brauer37} or more generally orthogonal easy quantum groups~\cite{banica07}.

In the following, we denote with $\PP^{(m)}$ the set
of all spatial partitions on $m$ levels and 
with $\PP^{(m)}(x, y)$ the set of all spatial partitions on $m$ levels
with upper colors $x$ and lower colors $y$.
Further, we introduce the following spatial partitions that will be used throughout
the rest of the paper.

\begin{definition}\label{def:special-partitions}
\leavevmode
\begin{enumerate}
\item Let $x \in \colors$. Then we denote with $\id_x \in \PP^{(1)}(x, x)$ the \textit{identity partition} on $x$.
It is the spatial partition on one level with upper and lower colors $x$, where 
each upper point is directly connected to the corresponding lower point, e.g.
\[
 \id_\circ = \partId, \qquad \id_\bullet = \partIdB, \qquad \id_{{\circ}{\bullet}{\circ}} = \partIdWBW.
\]
\item Let $p \in \PP^{(1)}(x, y)$ be a spatial partition. Then we denote with
$p^{(k)} \in \PP^{(k)}(x, y)$ the \textit{amplification} of $p$. It is obtained by placing $k$ copies of $p$
behind each other, e.g.
\[
  \id_\circ^{(2)} = \partition[3d]{%
  \line{0}{0}{0}{0}{1}{0}
  \line{0}{0}{1}{0}{1}{1}
  \point{0}{0}{0}{white}
  \point{0}{0}{1}{white}
  \point{0}{1}{0}{white}
  \point{0}{1}{1}{white}
}, \qquad 
  {\partPairWB}^{(3)} = 
  \partition[3d]{%
  \line{0}{0}{0}{0}{0.5}{0}
  \line{0}{0.5}{0}{1}{0.5}{0}
  \line{1}{0.5}{0}{1}{0}{0}
  \line{0}{0}{1}{0}{0.5}{1}
  \line{0}{0.5}{1}{1}{0.5}{1}
  \line{1}{0.5}{1}{1}{0}{1}
  \line{0}{0}{2}{0}{0.5}{2}
  \line{0}{0.5}{2}{1}{0.5}{2}
  \line{1}{0.5}{2}{1}{0}{2}
  \point{0}{0}{0}{white}
  \point{1}{0}{0}{black}
  \point{0}{0}{1}{white}
  \point{1}{0}{1}{black}
  \point{0}{0}{2}{white}
  \point{1}{0}{2}{black}
}.
\]
\item Let $x, y \in \{\circ, \bullet\}$ and $\sigma \in S_m$ be a permutation on $\{1, \ldots, m\}$. 
Then we denote with $\sigma_{xy} \in \PP^{(m)}(1, xy)$ the spatial partition with 
lower points $xy$ and blocks 
given by $\{ (1, i), (2, \sigma(i)) \}$ for all $1 \leq i \leq m$.
For example, we have
\[
  {(1)(2)}_{\circ\bullet} = \partition[3d]{
    \line{0}{0}{0}{0}{0.5}{0}
    \line{0}{0}{1}{0}{0.5}{1}
    \line{1}{0}{0}{1}{0.5}{0}
    \line{1}{0}{1}{1}{0.5}{1}
    \line{0}{0.5}{0}{1}{0.5}{0}
    \line{0}{0.5}{1}{1}{0.5}{1}
    \point{0}{0}{0}{white}
    \point{0}{0}{1}{white}
    \point{1}{0}{0}{black}
    \point{1}{0}{1}{black}
  }, \qquad 
  {(1 2)}_{\circ\bullet} = \partition[3d]{
    \line{0}{0}{0}{0}{0.5}{0}
    \line{0}{0}{1}{0}{0.5}{1}
    \line{1}{0}{0}{1}{0.5}{0}
    \line{1}{0}{1}{1}{0.5}{1}
    \line{0}{0.5}{0}{1}{0.5}{1}
    \line{0}{0.5}{1}{1}{0.5}{0}
    \point{0}{0}{0}{white}
    \point{0}{0}{1}{white}
    \point{1}{0}{0}{black}
    \point{1}{0}{1}{black}
  }, \qquad 
  {(1 3 2)}_{\circ\circ} = \partition[3d]{
    \line{0}{0}{2}{0}{0.5}{2}
    \line{0}{0}{1}{0}{0.5}{1}
    \line{0}{0}{0}{0}{0.5}{0}
    \line{1}{0}{0}{1}{0.5}{0}
    \line{1}{0}{1}{1}{0.5}{1}
    \line{1}{0}{2}{1}{0.5}{2}
    \point{0}{0}{2}{white}
    \line{0}{0.5}{2}{1}{0.5}{1}
    \line{0}{0.5}{1}{1}{0.5}{0}
    \line{0}{0.5}{0}{1}{0.5}{2}
    \point{0}{0}{0}{white}
    \point{0}{0}{1}{white}
    \point{1}{0}{0}{white}
    \point{1}{0}{1}{white}
    \point{1}{0}{2}{white}.
  },
\]
where the corresponding permutations are written in cycle notation.
Similarly, we denote with $\sigma^x_y \in \PP^{(m)}(x, y)$ the
rotated version of $\sigma_{xy}$ with upper color $x$ and lower color $y$, e.g.\ 
\[
    {(1)(2)}^\circ_\bullet = \partition[3d]{
      \line{0}{0}{0}{0}{1}{0}
      \line{0}{0}{1}{0}{1}{1}
      \point{0}{0}{0}{black}
      \point{0}{0}{1}{black}
      \point{0}{1}{0}{white}
      \point{0}{1}{1}{white}
    }, \qquad 
    {(1 2)}^\circ_\bullet = \partition[3d]{
      \line{0}{0}{0}{0}{1}{1}
      \line{0}{0}{1}{0}{1}{0}
      \point{0}{0}{0}{black}
      \point{0}{0}{1}{black}
      \point{0}{1}{0}{white}
      \point{0}{1}{1}{white}
    }, \qquad 
    {(1 3 2)}^\circ_\circ = \partition[3d]{
      \line{0}{1}{2}{0}{0}{1}
      \line{0}{1}{1}{0}{0}{0}
      \line{0}{1}{0}{0}{0}{2}
      \point{0}{0}{0}{white}
      \point{0}{0}{1}{white}
      \point{0}{0}{2}{white}
      \point{0}{1}{0}{white}
      \point{0}{1}{1}{white}
      \point{0}{1}{2}{white}.
    }.
\]
\end{enumerate}
\end{definition}

Given spatial partitions on a fixed number of levels $m$, we can use the following operations to construct 
new spatial partitions.

\begin{definition}
\leavevmode
\begin{enumerate}
\item Let $p \in \PP^{(m)}(x, y)$ and $q \in \PP^{(m)}(w, z)$. Then 
the \textit{tensor product} $p\otimes q \in \PP^{(m)}(xw, yz)$ is obtained by placing $p$ and $q$
next to each other, e.g.
\[
  \partition[3d]{%
  \line{0}{0}{0}{0}{1}{0}
  \line{0}{0}{1}{0}{1}{1}
  \point{0}{1}{0}{white}
  \point{0}{1}{1}{white}
  \point{0}{0}{0}{white}
  \point{0}{0}{1}{white}
  }
  \quad \otimes \quad 
  \partition[3d]{%
  \line{0}{0}{0}{0}{0.35}{0}
  \line{0}{0}{1}{0}{0.35}{1}
  \line{0}{0.35}{0}{0}{0.35}{1}
  \point{0}{0}{0}{black}
  \point{0}{0}{1}{black}
  } 
  \quad = \quad 
  \partition[3d]{%
  \line{0}{0}{0}{0}{1}{0}
  \line{0}{0}{1}{0}{1}{1}
  \line{1}{0}{0}{1}{0.35}{0}
  \line{1}{0}{1}{1}{0.35}{1}
  \line{1}{0.35}{0}{1}{0.35}{1}
  \point{0}{1}{0}{white}
  \point{0}{1}{1}{white}
  \point{0}{0}{0}{white}
  \point{0}{0}{1}{white}
  \point{1}{0}{0}{black}
  \point{1}{0}{1}{black}
  }.
\]
\item Let $p \in \PP^{(m)}(x, y)$. Then the \textit{involution}
$p^* \in \PP^{(m)}(y, x)$ is obtained by swapping the upper and lower points, e.g.
\[
  {\left(\partition[3d]{%
    \line{0}{0.65}{0}{0}{1}{0}
    \line{0}{0}{0}{0}{0.35}{0}
    \line{0}{0.35}{0}{0}{0.35}{1}
    \line{0}{0}{1}{0}{1}{1}
    \line{1}{0}{1}{1}{0.35}{1}
    \line{1}{0}{0}{1}{0.35}{0}
    \line{1}{0.35}{1}{1}{0.35}{0}
    \point{0}{0}{0}{white}
    \point{0}{0}{1}{white}
    \point{1}{0}{0}{black}
    \point{1}{0}{1}{black}
    \point{0}{1}{0}{white}
    \point{0}{1}{1}{white}}\right)}^*
    \quad = \quad
    \partition[3d]{%
    \line{0}{0.35}{0}{0}{0}{0}
    \line{0}{1}{0}{0}{0.65}{0}
    \line{0}{0.65}{0}{0}{0.65}{1}
    \line{0}{1}{1}{0}{0}{1}
    \line{1}{1}{1}{1}{0.65}{1}
    \line{1}{1}{0}{1}{0.65}{0}
    \line{1}{0.65}{1}{1}{0.65}{0}
    \point{0}{1}{0}{white}
    \point{0}{1}{1}{white}
    \point{1}{1}{0}{black}
    \point{1}{1}{1}{black}
    \point{0}{0}{0}{white}
    \point{0}{0}{1}{white}}.
\]
\item Let $p \in \PP^{(m)}(y, z)$ and $q \in \PP^{(m)}(x, y)$. Then the 
\textit{composition} $pq \in \PP^{(m)}(x, z)$ is obtained by first placing $q$ on top 
of $p$ and identifying the lower points of $q$ with the upper points of $p$.
Then these points are removed and the resulting blocks are simplified, e.g.
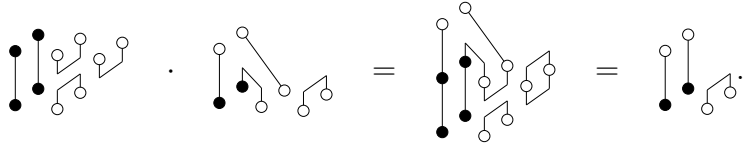
\[
  \partition[3d]{%
  \line{0}{0}{0}{0}{1}{0}
  \line{0}{0}{1}{0}{1}{1}
  \line{1}{0}{0}{1}{0.35}{0}
  \line{1}{0}{1}{1}{0.35}{1}
  \line{1}{0.35}{0}{1}{0.35}{1}
  \line{1}{1}{0}{1}{0.65}{0}
  \line{1}{1}{1}{1}{0.65}{1}
  \line{1}{0.65}{0}{1}{0.65}{1}
  \line{2}{1}{0}{2}{0.65}{0}
  \line{2}{1}{1}{2}{0.65}{1}
  \line{2}{0.65}{0}{2}{0.65}{1}
  \point{0}{1}{0}{black}
  \point{0}{1}{1}{black}
  \point{1}{1}{0}{white}
  \point{1}{1}{1}{white}
  \point{2}{1}{0}{white}
  \point{2}{1}{1}{white}
  \point{0}{0}{0}{black}
  \point{0}{0}{1}{black}
  \point{1}{0}{0}{white}
  \point{1}{0}{1}{white}}
  \quad \cdot \quad
  \partition[3d]{%
  \line{0}{0}{0}{0}{1}{0}
  \line{0}{1}{1}{1}{0}{1}
  \line{0}{0}{1}{0}{0.35}{1}
  \line{1}{0}{0}{1}{0.35}{0}
  \line{0}{0.35}{1}{1}{0.35}{0}
  \line{2}{0}{0}{2}{0.35}{0}
  \line{2}{0}{1}{2}{0.35}{1}
  \line{2}{0.35}{0}{2}{0.35}{1}
  \point{0}{1}{0}{white}
  \point{0}{1}{1}{white}
  \point{0}{0}{0}{black}
  \point{0}{0}{1}{black}
  \point{1}{0}{0}{white}
  \point{1}{0}{1}{white}
  \point{2}{0}{0}{white}
  \point{2}{0}{1}{white}}
  \quad = \quad 
  \partition[3d]{%
  \line{0}{0}{0}{0}{1}{0}
  \line{0}{0}{1}{0}{1}{1}
  \line{1}{0}{0}{1}{0.35}{0}
  \line{1}{0}{1}{1}{0.35}{1}
  \line{1}{0.35}{0}{1}{0.35}{1}
  \line{1}{1}{0}{1}{0.65}{0}
  \line{1}{1}{1}{1}{0.65}{1}
  \line{1}{0.65}{0}{1}{0.65}{1}
  \line{2}{1}{0}{2}{0.65}{0}
  \line{2}{1}{1}{2}{0.65}{1}
  \line{2}{0.65}{0}{2}{0.65}{1}
  \line{0}{1}{0}{0}{2}{0}
  \line{0}{2}{1}{1}{1}{1}
  \line{0}{1}{1}{0}{1.35}{1}
  \line{1}{1}{0}{1}{1.35}{0}
  \line{0}{1.35}{1}{1}{1.35}{0}
  \line{2}{1}{0}{2}{1.35}{0}
  \line{2}{1}{1}{2}{1.35}{1}
  \line{2}{1.35}{0}{2}{1.35}{1}
  \point{0}{1}{0}{black}
  \point{0}{1}{1}{black}
  \point{1}{1}{0}{white}
  \point{1}{1}{1}{white}
  \point{2}{1}{0}{white}
  \point{2}{1}{1}{white}
  \point{0}{0}{0}{black}
  \point{0}{0}{1}{black}
  \point{1}{0}{0}{white}
  \point{1}{0}{1}{white}
  \point{0}{2}{0}{white}
  \point{0}{2}{1}{white}
  }
  \quad = \quad 
  \partition[3d]{%
  \line{0}{0}{0}{0}{1}{0}
  \line{0}{0}{1}{0}{1}{1}
  \line{1}{0}{0}{1}{0.35}{0}
  \line{1}{0}{1}{1}{0.35}{1}
  \line{1}{0.35}{0}{1}{0.35}{1}
  \point{0}{1}{0}{white}
  \point{0}{1}{1}{white}
  \point{0}{0}{0}{black}
  \point{0}{0}{1}{black}
  \point{1}{0}{0}{white}
  \point{1}{0}{1}{white}
  }.
\]
\end{enumerate}
\end{definition}

Finally, we introduce the following properties
of spatial partitions that will be used later.

\begin{definition}
Let $p \in \PP^{(m)}(x, y)$ be a spatial partition. Then $p$ is called
\begin{enumerate}
\item \textit{invertible} if there exists a spatial partition $p^{-1} \in P^{(m)}(y, x)$
such that $p^{-1}p = \id_x$ and $pp^{-1} = \id_y$.
\item \textit{$\nn$-graded} for some $\nn := (n_1, \dots, n_m) \in \N^m$ 
if $n_k = n_\ell$ for all points $(i, k)$ and $(j, \ell)$
that are in the same block in $p$.
\item \textit{pair partition} if every block of $p$ has size two.
\end{enumerate}
\end{definition}

Note that the inverse of a spatial partition $p \in P^{(m)}(x, y)$ is always 
given by $p^{-1} = p^*$ and in this case $p$ and $p^*$ are both pair partitions whose blocks 
form a bijection between the corresponding upper and lower points, e.g.
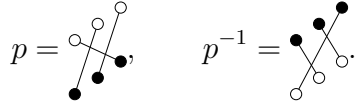
\[
    p = \partition[3d]{
      \line{0}{0}{0}{0}{1}{1}
      \line{0}{0}{1}{0}{1}{2}
      \line{0}{0}{2}{0}{1}{0}
      \point{0}{0}{0}{black}
      \point{0}{0}{1}{black}
      \point{0}{0}{2}{black}
      \point{0}{1}{0}{white}
      \point{0}{1}{1}{white}
      \point{0}{1}{2}{white}
    }, \qquad 
    p^{-1} = \partition[3d]{
      \line{0}{0}{0}{0}{1}{2}
      \line{0}{0}{1}{0}{1}{0}
      \line{0}{0}{2}{0}{1}{1}
      \point{0}{0}{0}{white}
      \point{0}{0}{1}{white}
      \point{0}{0}{2}{white}
      \point{0}{1}{0}{black}
      \point{0}{1}{1}{black}
      \point{0}{1}{2}{black}
    }.
\]
In particular, if $p$ is invertible, then $\abs{x} = \abs{y}$.
In the following, we present a proof of this statement for the case $\abs{x} = \abs{y} = 1$,
which will be relevant later.

\begin{lemma}\label{lem:inv-perm}
Let $p \in \PP^{(m)}(x, y)$ be invertible with $\abs{x} = \abs{y} = 1$.
Then $p = \sigma^x_y$ and $p^{-1} = {(\sigma^{-1})}^y_x$ for 
a permutation $\sigma \in S_m$.
\end{lemma}
\begin{proof}
As in~\cite{freslon16a}, we denote with $t(p)$ the number of through-blocks of $p$, i.e.\ the number of blocks of $p$ that contain both an upper and a lower point.
Then $t(p) \leq m$ and equality holds if and only if $p = \sigma^x_y$ for a permutation $\sigma \in S_m$. 
Further, $t(q p) \leq \min(t(q), t(p))$ for any $q \in \PP^{(m)}(y, x)$, see~\cite[Remark 2.6]{freslon16a}.
Thus, if $p$ is invertible, then
\[
  m = t(\id_x^{(m)}) = t(p^{-1} \cdot p) \leq \min(t(p^{-1}), t(p^{})) \leq t(p) \leq m.
\]
This shows $t(p) = m$ and similarly one obtains $t(p^{-1}) = m$. 
Thus, both $p$ and $p^{-1}$
are of the form $\sigma^x_y$ and ${(\sigma^{-1})}^y_x$ for a permutation $\sigma \in S_m$.
\end{proof}

\subsection{Categories of spatial partitions}

Next, we introduce categories of spatial partitions
as purely combinatorial objects. These will be used again in \Cref{sec:spatial-partition-QGs},
where we interpret categories of spatial partitions as representation
categories of compact matrix quantum groups.

\begin{definition}
A \textit{category of spatial partitions} on $m$ levels is 
a subset $\CC \subseteq \PP^{(m)}$ that
\begin{enumerate}
  \item is closed under compositions, tensor products and involutions,
  \item contains $\id_\circ$ and $\id_\bullet$.
\end{enumerate}
\end{definition}
If $\CC \subseteq \PP^{(m)}$ is a category of spatial partitions, then 
we denote with $\CC(x, y)$ the set of all spatial partitions in $\CC$
with upper colors $x$ and lower colors $y$.

Note that in contrast to~\cite{cebron16}, we do not include
the base partitions ${\partPairWB}^{(m)}$ and ${\partPairBW}^{(m)}$ in the definition. 
Thus, our definition is more general and includes all categories of spatial partitions
in the sense of Cébron-Weber as special case. However, we will come back to the role of 
these base partitions later. 
Examples of categories of spatial partitions include
\begin{enumerate}
\item the set $\PP^{(m)}$ of all spatial partitions on $m$ levels,
\item the set $\PP^{(m)}_2$ of all spatial pair partitions on $m$ levels,
\item the set $\mathcal{NC} \subseteq \PP^{(1)}$ of all non-crossing partitions, i.e.\ partitions that can be drawn without crossing lines in two dimensions.
\end{enumerate}
Additional examples of categories of spatial partitions can for example be found in~\cite{raum16,tarrago18,cebron16}. 
As in the previous examples, categories of spatial partitions are often defined
by specifying a property of spatial partitions. However, it is also possible to 
define categories of spatial partition by a generating set. 

\begin{definition}
  Let $\CC_0$ be a set of spatial partitions. Then we denote with $\CC := \angles{\CC_0}$ the category
  \textit{generated} by $\CC_0$. It is the smallest category of spatial partitions containing $\CC_0$
  and consists of all finite combinations of elements in $\CC_0$ and 
  the base partitions $\id_\circ$ and $\id_\bullet$.
\end{definition}

Consider again the categories of spatial partitions from the previous example.
Using generators, these 
categories can be written as
\[
    \PP^{(1)} = \big\langle{ \partPair, \partCross, \partChair}\big\rangle, \qquad
    \PP^{(1)}_2 = \big\langle{ \partPair, \partCross}\big\rangle, \qquad
    \mathcal{NC} = \big\langle{ \partPair, \partChair}\big\rangle,
\]
see~\cite{raum16} for further details. Moreover, generators for general $\PP^{(m)}$ and $\PP^{(m)}_2$ can be found in~\cite{cebron16}.

In the original definition of spatial partition quantum groups in~\cite{cebron16}, Cébron-Weber 
include the additional base partitions ${\partPairWB}^{(m)}$ and ${\partPairBW}^{(m)}$. These are used 
in the construction of spatial partitions quantum groups and guarantee that the representation $\overline{u}$
is unitary.
However, as discussed in \Cref{sec:repr-categories} and \Cref{sec:tannaka-krein},
it is only required that there exists a conjugate representation $u^\bullet$ that is not necessarily
given by $\overline{u}$. This is equivalent to the existence of solutions to the conjugate
equations in the sense of \Cref{def:conjugate-repr}. Thus, translating the conjugate equations into the setting of spatial 
partitions yields the following definition. 

\begin{definition}\label{def:rigid}
Let $\CC \subseteq \PP^{(m)}$ be a category of spatial partitions. Then 
$\CC$ is called \textit{rigid} if it contains a pair of duality partitions
$r \in \CC(1, \circ\bullet)$ and $s \in \CC(1, \bullet\circ)$ satisfying
the conjugate equations
\[
  (r^* \otimes \id_\circ) \cdot (\id_\circ \otimes s) = \id_\circ,
  \qquad 
  (s^* \otimes \id_\bullet) \cdot (\id_\bullet \otimes r) = \id_\bullet.
\]
\end{definition}

In the case of spatial partitions on one level, the conjugate
equations can be visualized diagrammatically as follows:
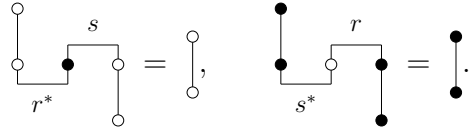
\def\partSnakeW{\partition[2d]{
  \line{3}{0}{0}{3}{1}{0}
  \line{0}{2}{0}{0}{1}{0}
  \line{0}{1}{0}{0}{0.65}{0}
  \line{1.5}{1}{0}{1.5}{0.65}{0}
  \line{0}{0.65}{0}{1.5}{0.65}{0}
  \line{3}{1}{0}{3}{1.35}{0}
  \line{1.5}{1}{0}{1.5}{1.35}{0}
  \line{3}{1.35}{0}{1.5}{1.35}{0}
  \point{0}{2}{0}{white}
  \point{0}{1}{0}{white}
  \point{1.5}{1}{0}{black}
  \point{3}{1}{0}{white}
  \point{3}{0}{0}{white}
  \label{0.75}{0.3}{0}{\scalebox{3}{$r^*$}}
  \label{2.25}{1.7}{0}{\scalebox{3}{$s$}}}}
\def\partSnakeB{\partition[2d]{
  \line{3}{0}{0}{3}{1}{0}
  \line{0}{2}{0}{0}{1}{0}
  \line{0}{1}{0}{0}{0.65}{0}
  \line{1.5}{1}{0}{1.5}{0.65}{0}
  \line{0}{0.65}{0}{1.5}{0.65}{0}
  \line{3}{1}{0}{3}{1.35}{0}
  \line{1.5}{1}{0}{1.5}{1.35}{0}
  \line{3}{1.35}{0}{1.5}{1.35}{0}
  \point{0}{2}{0}{black}
  \point{0}{1}{0}{black}
  \point{1.5}{1}{0}{white}
  \point{3}{1}{0}{black}
  \point{3}{0}{0}{black}
  \label{0.75}{0.3}{0}{\scalebox{3}{$s^*$}}
  \label{2.25}{1.7}{0}{\scalebox{3}{$r$}}}}
\[
  {\partSnakeW}
  \ = \
  {\partId},
  \qquad 
  {\partSnakeB}
    \ = \
  {\partIdB}.
\]
In particular, this shows that $r = \partPairWB$ and $s = \partPairBW$ 
are the only solutions in the case $m = 1$.
Further, one checks that the duality partitions $r = \partPairWB^{(m)}$ and $s = \partPairBW^{(m)}$
of Cébron-Weber satisfy the conjugate equations, which implies that 
categories of spatial partitions in the sense of~\cite{cebron16} are always rigid.
However, for $m \geq 2$, there exist additional solutions given by spatial partitions like
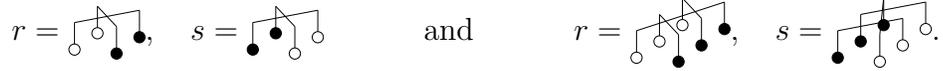
\begin{align*}
  r = {\partition[3d]{
    \line{0}{0}{0}{0}{0.5}{0}
    \line{0}{0}{1}{0}{0.5}{1}
    \line{1}{0}{0}{1}{0.5}{0}
    \line{1}{0}{1}{1}{0.5}{1}
    \line{0}{0.5}{0}{1}{0.5}{1}
    \line{0}{0.5}{1}{1}{0.5}{0}
    \point{0}{0}{0}{white}
    \point{0}{0}{1}{white}
    \point{1}{0}{0}{black}
    \point{1}{0}{1}{black}
  }}, \quad s = {\partition[3d]{
    \line{0}{0}{0}{0}{0.5}{0}
    \line{0}{0}{1}{0}{0.5}{1}
    \line{1}{0}{0}{1}{0.5}{0}
    \line{1}{0}{1}{1}{0.5}{1}
    \line{0}{0.5}{0}{1}{0.5}{1}
    \line{0}{0.5}{1}{1}{0.5}{0}
    \point{0}{0}{0}{black}
    \point{0}{0}{1}{black}
    \point{1}{0}{0}{white}
    \point{1}{0}{1}{white}
  }}
  &&
  \text{and}
  &&
  r = {\partition[3d]{
    \line{0}{0}{2}{0}{0.5}{2}
    \line{0}{0}{1}{0}{0.5}{1}
    \line{0}{0}{0}{0}{0.5}{0}
    \line{1}{0}{0}{1}{0.5}{0}
    \line{1}{0}{1}{1}{0.5}{1}
    \line{1}{0}{2}{1}{0.5}{2}
    \point{0}{0}{2}{white}
    \line{0}{0.5}{2}{1}{0.5}{1}
    \line{0}{0.5}{1}{1}{0.5}{0}
    \line{0}{0.5}{0}{1}{0.5}{2}
    \point{0}{0}{0}{white}
    \point{0}{0}{1}{white}
    \point{1}{0}{0}{black}
    \point{1}{0}{1}{black}
    \point{1}{0}{2}{black}
  }}, \quad s = {\partition[3d]{
    \line{0}{0}{0}{0}{0.5}{0}
    \line{0}{0}{1}{0}{0.5}{1}
    \line{0}{0}{2}{0}{0.5}{2}
    \line{1}{0}{0}{1}{0.5}{0}
    \line{1}{0}{1}{1}{0.5}{1}
    \line{1}{0}{2}{1}{0.5}{2}
    \point{0}{0}{2}{black}
    \line{0}{0.5}{0}{1}{0.5}{1}
    \line{0}{0.5}{1}{1}{0.5}{2}
    \line{0}{0.5}{2}{1}{0.5}{0}
    \point{0}{0}{0}{black}
    \point{0}{0}{1}{black}
    \point{1}{0}{0}{white}
    \point{1}{0}{1}{white}
    \point{1}{0}{2}{white}
  }}.
\end{align*}

The following proposition characterizes all possible solutions
of the conjugate equations in the context of spatial partitions.

\begin{proposition}\label{prop:conj-eq-solutions}
Duality partitions in the sense of \Cref{def:rigid} are exactly
of the form $r = \sigma_{\circ\bullet}$ and $s = \sigma^{-1}_{\bullet\circ}$ for a permutation $\sigma \in S_m$.
\end{proposition}
\begin{proof}
Fix a number of levels $m$ and assume $r \in \PP^{(m)}(1, \circ\bullet)$ and $s \in \PP^{(m)}(1, \bullet\circ)$ 
are a pair of duality partitions.
Define the new partitions $r' \in \PP^{(m)}(\circ, \bullet)$ and $s' \in P^{(m)}(\bullet, \circ)$ by moving 
the lower points $(1, i)$ for $1 \leq i \leq m$ to the upper layer. 
Then the conjugate equations (visualized only for one level) yield
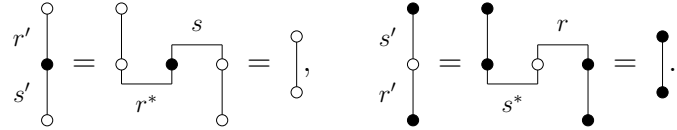
\[
  \partition[2d]{
    \line{0}{0}{0}{0}{1}{0}
    \line{0}{1}{0}{0}{2}{0}
    \point{0}{0}{0}{white}
    \point{0}{1}{0}{black}
    \point{0}{2}{0}{white}
    \label{-0.75}{0.5}{0}{\scalebox{3}{$s'$}}
    \label{-0.75}{1.5}{0}{\scalebox{3}{$r'$}}
  }
  \ = \
  {\partSnakeW}
  \ = \
  {\partId},
  \qquad 
  \partition[2d]{
    \line{0}{0}{0}{0}{1}{0}
    \line{0}{1}{0}{0}{2}{0}
    \point{0}{0}{0}{black}
    \point{0}{1}{0}{white}
    \point{0}{2}{0}{black}
    \label{-0.75}{0.5}{0}{\scalebox{3}{$r'$}}
    \label{-0.75}{1.5}{0}{\scalebox{3}{$s'$}}.
  }
  \ = \
  {\partSnakeB}
    \ = \
  {\partIdB}.
\]
Thus, $r'$ and $s'$ are a pair of inverse partitions and \Cref{lem:inv-perm} implies that 
they are the form $r' = \sigma^\circ_\bullet$ and $s' = {(\sigma^{-1})}^\bullet_\circ$ 
for a permutation $\sigma \in S_m$.
Hence, we obtain $r = \sigma_{\circ\bullet}$ and $s = \sigma^{-1}_{\bullet\circ}$ 
by moving the upper points $(1, i)$ for $1 \leq i \leq m$ back to the lower layer.

Conversely, let $\sigma \in S_m$ and define the partition
$r = \sigma_{\circ\bullet}$ and $s = \sigma^{-1}_{\bullet\circ}$. By tracing the blocks 
in the first conjugate equation, one checks that each point $(1, i)$ in $r$ 
is connected to $(2, \sigma(i))$, which is again connected to 
$(2, \sigma^{-1}(\sigma(i))) = (2, i)$ in $s$. Thus, the first conjugate equation
\[
  (r^* \otimes \id_\circ) \cdot (\id_\circ \otimes s) = \id_\circ
\]
is satisfied. Similarly, one checks that the second conjugate equation is satisfied, which shows that 
$r$ and $s$ are duality partitions.
\end{proof}

So far, we have only considered categories of spatial partition $\CC$ as purely combinatorial
objects. However, we can also give them the structure of a category in the usual sense 
by defining the set of objects $\Ob(\CC) := \colors$ and the set of morphisms
\[
  \Hom(x, y) := \CC(x, y) = \CC \cap \PP^{(m)}(x, y)
  \qquad
  \forall x,y \in \colors.
\]
Moreover, using the language of categories, the tensor product and involution give categories of spatial partitions 
the structure of strict monoidal $\dagger$-categories. See~\cite{etingof16} for precise definitions of these terms.

Using this framework, we can now show that the duality of the objects
$\circ$ and $\bullet$ in the sense of \Cref{def:rigid} implies the duality 
of $x$ and $\overline{x}$ for all objects $x \in \colors$.
Thus, our notation of rigidity agrees with the general notation
of rigidity as for example defined in~\cite{etingof16} and~\cite{neshveyev13}.

\begin{proposition}\label{prop:existance-duals}
  Let $\CC \subseteq \PP^{(m)}$ be a rigid category of spatial partitions and $x \in \colors$. 
  Then $\CC$ contains a pair of duality partitions
  $r \in \CC(1, x\overline{x})$ and $s \in \CC(1, \overline{x}x)$ satisfying
  the conjugate equations for $x$ and $\overline{x}$, i.e.\
  \[
    (r^* \otimes \id_x) \cdot (\id_x \otimes s) = \id_x,
    \qquad 
    (s^* \otimes \id_{\overline{x}}) \cdot (\id_{\overline{x}} \otimes r) = \id_{\overline{x}}.
  \]
\end{proposition}
\begin{proof}
The statement follows directly by inductively applying~\cite[Proposition 1.10.7]{etingof16} in the context 
of spatial partitions. Alternatively, the duality partitions $r$ and $s$ can be constructed explicitly
by nesting the duality partitions for $\circ$ and $\bullet$.
\end{proof}

\subsection{Spatial partition functors}\label{sec:spatial-functor}

Next, we introduce functors between categories of spatial partitions.
These allow us to transform spatial partitions while preserving 
their basic operations and will be used in the proofs our main results 
in \Cref{sec:spatial-partition-QGs} and \Cref{sec:proj-version}.

\begin{definition}\label{def:functor}
Let $\CC \subseteq \PP^{(n)}$ and $\DD \subseteq \PP^{(m)}$ be categories of spatial partitions. 
A \textit{spatial partition functor} $F \colon \CC \to \DD$ is 
given by a function $F \colon \colors \to \colors$ and functions
$F \colon \CC(x, y) \to \DD(F(x), F(y))$ for all $x, y \in \colors$ 
satisfying the following properties:
\begin{enumerate}
\item $F(1) = 1$ and $F(xy) = F(x)F(y)$ for all $x, y \in \colors$
\item $F$ preserves the category operators, i.e.\ 
\[
  F(p \cdot q) = F(p) \cdot F(q), \quad F(p^*) = F(p)^*, \quad F(p \otimes q) = F(p) \otimes F(q)
\]
for all (composable) $p, q \in \CC$,
\item $F(\id_{x}) = \id_{F(x)}$ for all $x \in \colors$.
\end{enumerate}
\end{definition}

Note that a functor does not necessarily preserve the number of levels 
of spatial partitions. Using the language of categories, 
we also introduce the notion of fully faithful functors.

\begin{definition}
Let $F \colon \CC \to \DD$ be a spatial partition functor. Then 
$F$ is called \textit{fully faithful} if the functions 
$F \colon \CC(x, y) \to \DD(F(x), F(y))$ are bijections
for all $x, y \in \colors$.
\end{definition}

A simple example of a spatial partition functor is obtained by decoloring spatial 
partitions. Formally, we define $F \colon \PP^{(m)} \to \PP^{(m)}$ by
\begin{alignat*}{3}
  F(x) &:= \circ^{\abs{x}} &\qquad& \forall x \in \colors, \\
  F(p) &:= (\circ^{\abs{x}}, \circ^{\abs{y}}, \{B_i\}) &\qquad& \forall p := (x, y, \{B_i\}) \in \PP^{(m)}.
\end{alignat*}

Since $F$ preserves the block structure of spatial partitions, one can check directly 
that it satisfies all properties in \Cref{def:functor} and that $F$ is fully faithful.
However, $F$ is not surjective in the sense that its image does not contain ${\partIdB}^{(m)}$.
In particular, this shows that the image of a spatial partition functor is not necessarily again a category 
of spatial partitions.

Next, we present two additional examples of spatial partition functors that  
will be used again in the context of spatial partition quantum groups in \Cref{sec:spatial-partition-QGs} and \Cref{sec:proj-version}.

\subsubsection{Permuting levels}\label{sec:functor-perm}

Let $\sigma, \tau \in S_m$ be permutations. Then we obtain a 
functor $\Perm_{\sigma,\tau} \colon \PP^{(m)} \to \PP^{(m)}$ by permuting the
levels of white points using $\sigma$ and the levels of black points using $\tau$. 
Consider for example $\sigma = (12)$ and $\tau = (1)(2)$. Then
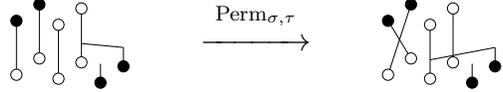
\[
  \partition[3d]{%
      \line{0}{0}{0}{0}{1}{0}
      \line{0}{0}{1}{0}{1}{1}
      \line{1}{0}{0}{1}{1}{0}
      \line{1}{0}{1}{1}{1}{1}
      \line{2}{0}{0}{2}{0.35}{0}
      \line{2}{0}{1}{2}{0.35}{1}
      \line{1}{0.35}{1}{2}{0.35}{1}
      \point{0}{1}{0}{black}
      \point{0}{1}{1}{black}
      \point{1}{1}{0}{white}
      \point{1}{1}{1}{white}
      \point{0}{0}{0}{white}
      \point{0}{0}{1}{white}
      \point{1}{0}{0}{white}
      \point{1}{0}{1}{white}
      \point{2}{0}{0}{black}
      \point{2}{0}{1}{black}
  }%
  \qquad \overset{\Perm_{\sigma,\tau}}{\xrightarrow{\hspace{3em}}} \qquad
  \partition[3d]{%
      \line{0}{0}{0}{0}{1}{1}
      \line{0}{0}{1}{0}{1}{0}
      \line{1}{0}{0}{1}{1}{0}
      \line{1}{0}{1}{1}{1}{1}
      \line{2}{0}{0}{2}{0.35}{0}
      \line{2}{0}{1}{2}{0.35}{1}
      \point{0}{1}{0}{black}
      \point{0}{1}{1}{black}
      \point{1}{1}{0}{white}
      \point{1}{1}{1}{white}
      \point{0}{0}{0}{white}
      \point{0}{0}{1}{white}
      \point{1}{0}{0}{white}
      \point{1}{0}{1}{white}
      \point{2}{0}{0}{black}
      \point{2}{0}{1}{black}
      \line{1}{0.35}{0}{2}{0.35}{1}
  }.%
\]
Using the permutation partitions introduced in \Cref{def:special-partitions}, we can define $\Perm_{\sigma,\tau}$
as follows.

\begin{definition}\label{def:functor-perm}
Let $\sigma, \tau \in S_m$ and define
\[
  q_{\sigma,\tau}^\circ   := \sigma^\circ_\circ, \qquad
  q_{\sigma,\tau}^\bullet := \tau^\bullet_\bullet, \qquad
  q_{\sigma,\tau}^{xy}    := q_{\sigma,\tau}^{x} \otimes q_{\sigma,\tau}^{y} \quad \forall x, y \in \colors, \, \abs{xy} > 1.
\]
Then the \textit{permutation functor}
$\Perm_{\sigma,\tau} \colon \PP^{(m)} \to \PP^{(m)}$
is given by 
\begin{alignat*}{3}
  \Perm_{\sigma,\tau}(x) &:= x &\qquad& \forall x \in \colors, \\
  \Perm_{\sigma,\tau}(p) &:= q_{\sigma,\tau}^y \cdot p \cdot {(q_{\sigma,\tau}^x)}^{-1} &\qquad& \forall p \in \PP^{(m)}(x, y).
\end{alignat*}
\end{definition}

\begin{example}
Consider again the spatial partition
from the initial example with $\sigma := (12)$ and $\tau = (1)(2)$. Then 
\[
  q_{\sigma,\tau}^{{\bullet}{\circ}} = 
  {\partition[3d]{
    \line{0}{0}{0}{0}{1}{0}
    \line{0}{0}{1}{0}{1}{1}
    \line{1}{0}{0}{1}{1}{1}
    \line{1}{0}{1}{1}{1}{0}
    \point{0}{0}{0}{black}
    \point{0}{0}{1}{black}
    \point{1}{0}{0}{white}
    \point{1}{0}{1}{white}
    \point{0}{1}{0}{black}
    \point{0}{1}{1}{black}
    \point{1}{1}{0}{white}
    \point{1}{1}{1}{white}
  }},%
  \qquad 
  q_{\sigma,\tau}^{{\circ}{\circ}{\bullet}} = 
  {\partition[3d]{
    \line{0}{0}{0}{0}{1}{1}
    \line{0}{0}{1}{0}{1}{0}
    \line{1}{0}{0}{1}{1}{1}
    \line{1}{0}{1}{1}{1}{0}
    \line{2}{0}{0}{2}{1}{0}
    \line{2}{0}{1}{2}{1}{1}
    \point{0}{0}{0}{white}
    \point{0}{0}{1}{white}
    \point{1}{0}{0}{white}
    \point{1}{0}{1}{white}
    \point{2}{0}{0}{black}
    \point{2}{0}{1}{black}
    \point{0}{1}{0}{white}
    \point{0}{1}{1}{white}
    \point{1}{1}{0}{white}
    \point{1}{1}{1}{white}
    \point{2}{1}{0}{black}
    \point{2}{1}{1}{black}
  }},%
\]
which yields 
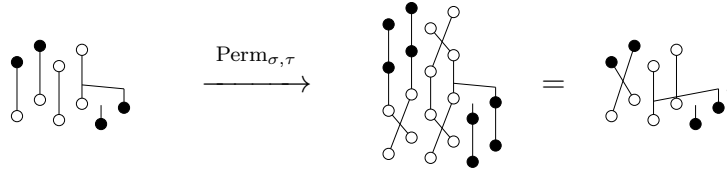
\[
  {\partition[3d]{%
      \line{0}{0}{0}{0}{1}{0}
      \line{0}{0}{1}{0}{1}{1}
      \line{1}{0}{0}{1}{1}{0}
      \line{1}{0}{1}{1}{1}{1}
      \line{2}{0}{0}{2}{0.35}{0}
      \line{2}{0}{1}{2}{0.35}{1}
      \line{1}{0.35}{1}{2}{0.35}{1}
      \point{0}{1}{0}{black}
      \point{0}{1}{1}{black}
      \point{1}{1}{0}{white}
      \point{1}{1}{1}{white}
      \point{0}{0}{0}{white}
      \point{0}{0}{1}{white}
      \point{1}{0}{0}{white}
      \point{1}{0}{1}{white}
      \point{2}{0}{0}{black}
      \point{2}{0}{1}{black}
  }}%
  \qquad \overset{\Perm_{\sigma,\tau}}{\xrightarrow{\hspace{3em}}} \qquad
  {\partition[3d,aspectY=0.8]{%
      \line{0}{0}{0}{0}{1}{0}
      \line{0}{0}{1}{0}{1}{1}
      \line{1}{0}{0}{1}{1}{0}
      \line{1}{0}{1}{1}{1}{1}
      \line{2}{0}{0}{2}{0.35}{0}
      \line{2}{0}{1}{2}{0.35}{1}
      \line{1}{0.35}{1}{2}{0.35}{1}
      \line{0}{1}{0}{0}{2}{0}
      \line{0}{1}{1}{0}{2}{1}
      \line{1}{1}{0}{1}{2}{1}
      \line{1}{1}{1}{1}{2}{0}
      \line{0}{-1}{0}{0}{0}{1}
      \line{0}{-1}{1}{0}{0}{0}
      \line{1}{-1}{0}{1}{0}{1}
      \line{1}{-1}{1}{1}{0}{0}
      \line{2}{-1}{0}{2}{0}{0}
      \line{2}{-1}{1}{2}{0}{1}
      \point{0}{1}{0}{black}
      \point{0}{1}{1}{black}
      \point{1}{1}{0}{white}
      \point{1}{1}{1}{white}
      \point{0}{0}{0}{white}
      \point{0}{0}{1}{white}
      \point{1}{0}{0}{white}
      \point{1}{0}{1}{white}
      \point{2}{0}{0}{black}
      \point{2}{0}{1}{black}
      \point{0}{2}{0}{black}
      \point{0}{2}{1}{black}
      \point{1}{2}{0}{white}
      \point{1}{2}{1}{white}
      \point{0}{-1}{0}{white}
      \point{0}{-1}{1}{white}
      \point{1}{-1}{0}{white}
      \point{1}{-1}{1}{white}
      \point{2}{-1}{0}{black}
      \point{2}{-1}{1}{black}
  }}%
  \quad = \quad
  {\partition[3d]{%
      \line{0}{0}{0}{0}{1}{1}
      \line{0}{0}{1}{0}{1}{0}
      \line{1}{0}{0}{1}{1}{0}
      \line{1}{0}{1}{1}{1}{1}
      \line{2}{0}{0}{2}{0.35}{0}
      \line{2}{0}{1}{2}{0.35}{1}
      \point{0}{1}{0}{black}
      \point{0}{1}{1}{black}
      \point{1}{1}{0}{white}
      \point{1}{1}{1}{white}
      \point{0}{0}{0}{white}
      \point{0}{0}{1}{white}
      \point{1}{0}{0}{white}
      \point{1}{0}{1}{white}
      \point{2}{0}{0}{black}
      \point{2}{0}{1}{black}
      \line{1}{0.35}{0}{2}{0.35}{1}
  }}.%
\]
\end{example}

Next, we verify that the permutation functor $\Perm_{\sigma,\tau}$ is indeed a spatial partition 
functor in the sense of \Cref{def:functor} that additionally is fully faithful.

\begin{proposition}
$\Perm_{\sigma,\tau}$ is a fully faithful spatial partition functor.
\end{proposition}
\begin{proof}
Since $\Perm_{\sigma,\tau}$ preserves colors, it follows immediately that 
it respects the concatenation of colors. 
Further, \Cref{lem:inv-perm} implies $q_{\sigma,\tau}^* = q_{\sigma,\tau}^{-1}$, which yields
\[
  \Perm_{\sigma,\tau}(p)^* = {({(q_{\sigma,\tau}^x)}^{-1})}^* \cdot p^* \cdot {(q_{\sigma,\tau}^y)}^* = \Perm_{\sigma,\tau}(p^*) \qquad \forall p \in \PP^{(m)}(x, y).
\]
Similarly, one verifies 
directly that
\begin{align*}
  \Perm_{\sigma,\tau}(p_1 \cdot p_2) 
  &= \Perm_{\sigma,\tau}(p_1) \cdot \Perm_{\sigma,\tau}(p_2), \\
  \Perm_{\sigma,\tau}(p_1 \otimes p_2) 
  &= \Perm_{\sigma,\tau}(p_1) \otimes \Perm_{\sigma,\tau}(p_2)
\end{align*}
for all (composable) $p_1, p_2 \in \PP^{(m)}$.
Hence, $\Perm_{\sigma,\tau}$ respects the category operations.
Additionally, we have 
\[
  \Perm_{\sigma,\tau}(\id_x) 
  = q_{\sigma,\tau}^{x} \cdot \id_x \cdot {(q_{\sigma,\tau}^{x})}^{-1} 
  = \id_x 
  \quad 
  \forall x \in \colors,
\]
which shows that $\Perm_{\sigma,\tau}$ is spatial partition functor.
Finally, one checks that 
\[
  \Perm_{\sigma,\tau}(\Perm_{\sigma^{-1},\tau^{-1}}(p)) = \Perm_{\sigma^{-1},\tau^{-1}}(\Perm_{\sigma,\tau}(p)) = p
  \qquad 
  \forall p \in \PP^{(m)},
\]
which implies that $\Perm_{\sigma,\tau}$ is bijective on the sets $\PP^{(m)}(x, y)$.
Thus, it is fully faithful.
\end{proof}

Finally, we introduce $\nn$-graded permutations
and show that $\Perm_{\sigma,\tau}$ maps categories
of spatial partitions to categories of spatial partitions while 
preserving rigidity.

\begin{definition}
Let $\nn := (n_1, \dots, n_m) \in \N^m$. A permutation $\sigma \in S_m$ is called \textit{$\nn$-graded} if $n_i = n_{\sigma(i)}$ for all 
$1 \leq i \leq m$.
\end{definition}

\begin{proposition}
Let $\sigma, \tau \in S_m$ be $\nn$-graded permutations and 
$\CC \subseteq \PP^{(m)}$ be a $\nn$-graded rigid category of spatial partitions.
Then $\Perm_{\sigma,\tau}(\CC)$ is again a $\nn$-graded rigid category of spatial partitions. 
\end{proposition}
\begin{proof}
Define $\DD := \Perm_{\sigma,\tau}(\CC)$. Since $\Perm_{\sigma,\tau}$ respects the category
operations, it follows directly that $\DD$ is closed under compositions, tensor products and involutions.
Further, $\Perm_{\sigma,\tau}(\id_x) = \id_x \in \DD$ for all $x \in \colors$, which shows that $\DD$ is 
a category of spatial partitions. Since $\CC$ rigid, there exists a pair 
of duality partitions $r \in \CC(1, \circ\bullet)$ and $s \in \CC(1, \bullet\circ)$
satisfying the conjugate equations. Then $\Perm_{\sigma,\tau}(r) \in \DD(1, \circ\bullet)$
and $\Perm_{\sigma,\tau}(s) \in \DD(1, \bullet\circ)$ also satisfy the conjugate equations, which implies 
that $\DD$ is again rigid. Finally, observe that permuting the levels of a $\nn$-graded
spatial partition by a $\nn$-graded permutation yields again a $\nn$-graded spatial partition.
Thus, $\DD$ is also a $\nn$-graded category of spatial partitions.
\end{proof}

\subsubsection{Flattening partitions}\label{sec:functor-flat}

Next, we introduce the functor $\Flat_{m,z}$ that flattens spatial partitions along the levels and is used 
in \Cref{sec:proj-version} in the context of projective versions of spatial partition quantum groups.
Its definition is motivated by \cite[Remark 2.4\ \& 2.8]{cebron16},
where Cébron-Weber consider the bijection
\[
   \{ 1, \ldots k + \ell \} \times \{1, \ldots m\} \cong \{ 1, \ldots, m \cdot (k + \ell) \}
\]
that identifies a point $(i, j)$ with the point $m \cdot (i-1) + j$. This bijection induces a bijection
between the sets of spatial partitions $\PP^{(m)}(\circ^k, \circ^l)$ and $\PP^{(1)}(\circ^{mk}, \circ^{ml})$
by rearranging the points accordingly, e.g.
\[
    {\partition[3d]{
      \line{0}{0}{0}{0}{1}{0}
      \line{0}{0}{1}{0}{1}{2}
      \line{0}{0}{2}{0}{1}{1}
      \line{1}{0}{0}{1}{0.35}{0}
      \line{1}{0}{1}{1}{0.35}{1}
      \line{1}{0.35}{0}{1}{0.35}{1}
      \line{1}{0}{2}{1}{0.35}{2}
      \point{0}{0}{0}{white}
      \point{0}{0}{1}{white}
      \point{0}{0}{2}{white}
      \point{1}{0}{0}{white}
      \point{1}{0}{1}{white}
      \point{1}{0}{2}{white}
      \point{0}{1}{0}{white}
      \point{0}{1}{1}{white}
      \point{0}{1}{2}{white}
    }} 
    \, \in \PP^{(3)}(\circ^1, \circ^2) 
    \quad \cong \quad 
    {\partition[2d]{
      \line{0}{0}{0}{0}{1}{0}
      \line{1}{0}{0}{2}{1}{0}
      \line{2}{0}{0}{1}{1}{0}
      \line{3}{0}{0}{3}{0.35}{0}
      \line{4}{0}{1}{4}{0.35}{0}
      \line{3}{0.35}{0}{4}{0.35}{0}
      \line{5}{0}{0}{5}{0.35}{0}
      \point{0}{0}{0}{white}
      \point{1}{0}{0}{white}
      \point{2}{0}{0}{white}
      \point{3}{0}{0}{white}
      \point{4}{0}{0}{white}
      \point{5}{0}{0}{white}
      \point{0}{1}{0}{white}
      \point{1}{1}{0}{white}
      \point{2}{1}{0}{white}
    }} 
    \, \in \PP^{(1)}(\circ^3, \circ^6). 
\] 
It respects the composition, tensor product and involution of spatial partitions, but it fails to preserve the base partitions
as defined in~\cite{cebron16} when applied to categories of spatial partitions.

Next, we generalize the previous bijection to the case of $m$ levels and colored spatial partitions. Note that 
we reverse the order of black points in the following definition since in the context of quantum groups 
the representation $\overline{u \otop w}$ is equivalent to $\overline{w} \otop \overline{u}$. See \Cref{sec:proj-version} for further information.

\begin{definition}\label{def:bij-varphi}
Let $m, d \geq 1$
and $x, y \in \colors$ with $k := \abs{x}$ and $\ell := \abs{y}$. 
Define the bijections
\[
  \varphi_{m,d}^{x,y} \colon \{1, \dots, k + \ell \} \times \{1, \dots, m \cdot d\}
  \to \{1, \dots, (k + \ell) \cdot d \} \times \{1, \dots, m\}
\]
given by
\[
  (i, j + k \cdot m) \mapsto \begin{cases}
    (i \cdot d - d + k + 1, j) & \text{if ${(xy)}_{i} = \circ$}, \\
    (i \cdot d - k, j) & \text{if ${(xy)}_{i} = \bullet$},
  \end{cases}
\]
for all $1 \leq i \leq k + \ell$, $1 \leq j \leq m$ and $0 \leq k < d$.  
\end{definition}

In the special case $x = \circ^k$, $y = \circ^\ell$ and $m = 1$, we obtain 
exactly the bijection of Cébron-Weber from the beginning of the section.
A more general case on four levels is visualized in the following example.

\begin{example}
Consider four levels of points with upper colors $x = \circ$ and 
lower colors $y = \circ\bullet$. Then applying $\varphi := \varphi^{x,y}_{2,2}$ to these points yields
\[
  \partition[3d]{%
    \label{0}{1}{0}{{\Huge(1,1)}}
    \label{0}{1}{1}{{\Huge(1,2)}}
    \label{0}{1}{2}{{\Huge(1,3)}}
    \label{0}{1}{3}{{\Huge(1,4)}}
    \label{0}{0}{0}{{\Huge(2,1)}}
    \label{0}{0}{1}{{\Huge(2,2)}}
    \label{0}{0}{2}{{\Huge(2,3)}}
    \label{0}{0}{3}{{\Huge(2,4)}}
    \label{1}{0}{0}{{\Huge(3,1)}}
    \label{1}{0}{1}{{\Huge(3,2)}}
    \label{1}{0}{2}{{\Huge(3,3)}}
    \label{1}{0}{3}{{\Huge(3,4)}}
}%
\qquad \overset{\varphi}{\xrightarrow{\hspace{2em}}} \qquad
\partition[3d]{%
    \label{0}{1}{0}{{\Huge(1,1)}}
    \label{0}{1}{1}{{\Huge(1,2)}}
    \label{0}{1}{2}{{\Huge(2,1)}}
    \label{0}{1}{3}{{\Huge(2,2)}}
    \label{0}{0}{0}{{\Huge(3,1)}}
    \label{0}{0}{1}{{\Huge(3,2)}}
    \label{0}{0}{2}{{\Huge(4,1)}}
    \label{0}{0}{3}{{\Huge(4,2)}}
    \label{1}{0}{0}{{\Huge(6,1)}}
    \label{1}{0}{1}{{\Huge(6,2)}}
    \label{1}{0}{2}{{\Huge(5,1)}}
    \label{1}{0}{3}{{\Huge(5,2)}}
}%
\quad \cong \quad
\partition[3d]{%
    \label{0}{1}{0}{{\Huge(1,1)}}
    \label{0}{1}{1}{{\Huge(1,2)}}
    \label{1}{1}{0}{{\Huge(2,1)}}
    \label{1}{1}{1}{{\Huge(2,2)}}
    \label{0}{0}{0}{{\Huge(3,1)}}
    \label{0}{0}{1}{{\Huge(3,2)}}
    \label{1}{0}{0}{{\Huge(4,1)}}
    \label{1}{0}{1}{{\Huge(4,2)}}
    \label{2}{0}{0}{{\Huge(5,1)}}
    \label{2}{0}{1}{{\Huge(5,2)}}
    \label{3}{0}{0}{{\Huge(6,1)}}
    \label{3}{0}{1}{{\Huge(6,2)}},
} 
\]
where the relabeled points are again rearranged according to our usual convention.
\end{example}

By extending this bijection to the blocks of spatial partitions, we obtain for example
\[
  \partition[3d]{%
    \line{0}{0}{0}{0}{1}{0}
    \line{0}{0}{1}{0}{1}{1}
    \line{0}{0}{2}{0}{1}{3}
    \line{0}{0}{3}{0}{1}{2}
    \line{1}{0}{0}{1}{0.35}{0}
    \line{1}{0}{1}{1}{0.35}{1}
    \line{1}{0.35}{0}{1}{0.35}{1}
    \line{1}{0}{2}{1}{0.35}{2}
    \line{1}{0}{3}{1}{0.35}{3}
    \point{0}{1}{0}{white}
    \point{0}{1}{1}{white}
    \point{0}{1}{2}{white}
    \point{0}{1}{3}{white}
    \point{0}{0}{0}{white}
    \point{0}{0}{1}{white}
    \point{0}{0}{2}{white}
    \point{0}{0}{3}{white}
    \point{1}{0}{0}{black}
    \point{1}{0}{1}{black}
    \point{1}{0}{2}{black}
    \point{1}{0}{3}{black}
}%
\qquad \overset{\varphi}{\xrightarrow{\hspace{2em}}} \qquad
\partition[3d]{%
    \line{0}{0}{0}{0}{1}{0}
    \line{0}{0}{1}{0}{1}{1}
    \line{1}{0}{0}{1}{1}{1}
    \line{1}{0}{1}{1}{1}{0}
    \line{3}{0}{0}{3}{0.35}{0}
    \line{3}{0}{1}{3}{0.35}{1}
    \line{3}{0.35}{0}{3}{0.35}{1}
    \line{2}{0}{0}{2}{0.35}{0}
    \line{2}{0}{1}{2}{0.35}{1}
    \point{0}{1}{0}{white}
    \point{0}{1}{1}{white}
    \point{1}{1}{0}{white}
    \point{1}{1}{1}{white}
    \point{0}{0}{0}{white}
    \point{0}{0}{1}{white}
    \point{1}{0}{0}{white}
    \point{1}{0}{1}{white}
    \point{2}{0}{0}{black}
    \point{2}{0}{1}{black}
    \point{3}{0}{0}{black}
    \point{3}{0}{1}{black}
}. 
\]
In particular, we observe that black points are flattened in reverse order.
Using these bijections $\varphi^{x,y}_{m,d}$ from \Cref{def:bij-varphi}, we can now define the 
functor $\Flat_{m,z}$.

\begin{definition}
Let $m \in \N$ and $z \in \colors$ with $d := \abs{z} > 1$.
Define the spatial partition functor $\Flat_{m,z} \colon \PP^{(m \cdot d)} \to \PP^{(m)}$ by 
\[
  \Flat_{m,z}(\circ) := z, \quad 
  \Flat_{m,z}(\bullet) := \overline{z} = \overline{z_d} \cdots \overline{z_1}
\]
and 
\[
\Flat_{m,z}(p) := \big(\Flat_{m,z}(x), \Flat_{m,z}(y), \{\varphi^{x,y}_{m,d}(B_i)\} \big)
\]
for all spatial partitions $p := (x, y, \{B_i\}) \in \PP^{(m \cdot d)}$.
\end{definition}

If $z = \circ \cdots \circ$, then $\Flat_{m,z}$ preserves colors and only applies the permutations
$\varphi^{x,y}_{m,\abs{z}}$ as in the previous example. However, in the general case, we obtain for example
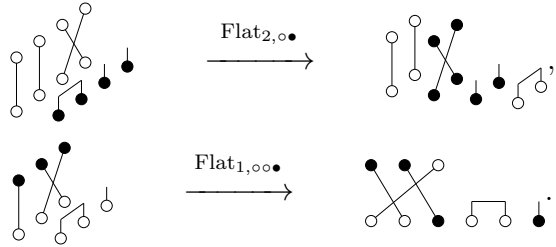
\[
  \partition[3d]{%
    \line{0}{0}{0}{0}{1}{0}
    \line{0}{0}{1}{0}{1}{1}
    \line{0}{0}{2}{0}{1}{3}
    \line{0}{0}{3}{0}{1}{2}
    \line{1}{0}{0}{1}{0.35}{0}
    \line{1}{0}{1}{1}{0.35}{1}
    \line{1}{0.35}{0}{1}{0.35}{1}
    \line{1}{0}{2}{1}{0.35}{2}
    \line{1}{0}{3}{1}{0.35}{3}
    \point{0}{1}{0}{white}
    \point{0}{1}{1}{white}
    \point{0}{1}{2}{white}
    \point{0}{1}{3}{white}
    \point{0}{0}{0}{white}
    \point{0}{0}{1}{white}
    \point{0}{0}{2}{white}
    \point{0}{0}{3}{white}
    \point{1}{0}{0}{black}
    \point{1}{0}{1}{black}
    \point{1}{0}{2}{black}
    \point{1}{0}{3}{black}
}%
\qquad \overset{\Flat_{2,\circ\bullet}}{\xrightarrow{\hspace{3em}}} \qquad
\partition[3d]{%
    \line{0}{0}{0}{0}{1}{0}
    \line{0}{0}{1}{0}{1}{1}
    \line{1}{0}{0}{1}{1}{1}
    \line{1}{0}{1}{1}{1}{0}
    \line{3}{0}{0}{3}{0.35}{0}
    \line{3}{0}{1}{3}{0.35}{1}
    \line{3}{0.35}{0}{3}{0.35}{1}
    \line{2}{0}{0}{2}{0.35}{0}
    \line{2}{0}{1}{2}{0.35}{1}
    \point{0}{1}{0}{white}
    \point{0}{1}{1}{white}
    \point{1}{1}{0}{black}
    \point{1}{1}{1}{black}
    \point{0}{0}{0}{white}
    \point{0}{0}{1}{white}
    \point{1}{0}{0}{black}
    \point{1}{0}{1}{black}
    \point{2}{0}{0}{black}
    \point{2}{0}{1}{black}
    \point{3}{0}{0}{white}
    \point{3}{0}{1}{white},
},
\]

\[
  {\partition[3d]{
      \line{0}{0}{0}{0}{1}{0}
      \line{0}{0}{1}{0}{1}{2}
      \line{0}{0}{2}{0}{1}{1}
      \line{1}{0}{0}{1}{0.35}{0}
      \line{1}{0}{1}{1}{0.35}{1}
      \line{1}{0.35}{0}{1}{0.35}{1}
      \line{1}{0}{2}{1}{0.35}{2}
      \point{0}{0}{0}{white}
      \point{0}{0}{1}{white}
      \point{0}{0}{2}{white}
      \point{1}{0}{0}{white}
      \point{1}{0}{1}{white}
      \point{1}{0}{2}{white}
      \point{0}{1}{0}{black}
      \point{0}{1}{1}{black}
      \point{0}{1}{2}{black}
    }} 
\qquad
\overset{\Flat_{1,{\circ}{\circ}{\bullet}}}{\xrightarrow{\hspace{3em}}}
\qquad
{\partition[2d]{
  \line{0}{0}{0}{2}{1}{0}
  \line{1}{0}{0}{0}{1}{0}
  \line{2}{0}{0}{1}{1}{0}
  \line{3}{0}{0}{3}{0.35}{0}
  \line{4}{0}{1}{4}{0.35}{0}
  \line{3}{0.35}{0}{4}{0.35}{0}
  \line{5}{0}{0}{5}{0.35}{0}
  \point{0}{0}{0}{white}
  \point{1}{0}{0}{white}
  \point{2}{0}{0}{black}
  \point{3}{0}{0}{white}
  \point{4}{0}{0}{white}
  \point{5}{0}{0}{black}
  \point{0}{1}{0}{black}
  \point{1}{1}{0}{black}
  \point{2}{1}{0}{white}
}}. 
\]
It remains to verify that $\Flat_{m,z}$ is indeed a spatial partition functor.

\begin{proposition}
$\Flat_{m,z}$ is a fully faithful spatial partition functor.
\end{proposition}
\begin{proof}
By definition, $\Flat_{m,z}$ respects the concatenation of colors.
Further, $\Flat_{m,z}$ respects compositions and involutions since it permutes the points of spatial 
only depending on the upper colors $x$ and lower colors $y$.
Moreover, $\Flat_{m,z}$ moves points at back levels in consecutive groups of size $m$ to 
the front, which also implies that it respects tensor products.
Finally, 
\[
  \Flat_{m,z}(\id_\circ^{(m \cdot \abs{z})}) = \id_{z}^{(m)},
  \qquad 
  \Flat_{m,z}(\id_\bullet^{(m \cdot \abs{z})}) = \id_{\overline{z}}^{(m)}
\]
yields that $\Flat_{m,z}$ maps identity partitions to identity partitions.
Thus, $\Flat_{m,z}$ is a spatial partition functor. Additionally, it is fully faithful since 
the maps $\varphi^{x,y}_{m,d}$ are bijections. 
\end{proof}

The following proposition now shows that the pre-image under $\Flat_{m,z}$
preserves categories of spatial partitions, rigidity and gradings. This 
will be used again in \Cref{sec:proj-version}, where it allows us to relate the category of 
a spatial partition quantum group to the category of its projective versions. 
Note that the proof relies on our generalization of base partitions 
and the statement would not hold in the setting of spatial partition quantum groups defined by Cébron-Weber.

\begin{proposition}\label{prop:flat-category}
Let $\nn := (n_1, \dots, n_m) \in \N^m$ and $z \in \colors$ with $d := \abs{z} \geq 1$.
If $\CC \subseteq \PP^{(m)}$ is a
$\nn$-graded rigid category of spatial partitions, then 
\[
  \Flat_{m,z}^{-1}(\CC) := \{ p \in \PP^{(m \cdot d)} \ | \ \Flat_{m,z}(p) \in \CC \}
\]
is a $(\nn \cdots \nn)$-graded rigid category of spatial partitions, where 
\[
  (\nn \cdots \nn) := (n_1, \dots, n_m, \dots, n_1, \dots, n_m) \in \N^{m \cdot d}.
\]
\end{proposition}
\begin{proof}
Define $\DD := \Flat_{m,z}^{-1}(\CC)$. Since $\Flat_{m,z}$ respects all operations on spatial partitions,
it follows that $\DD$ is closed under compositions, involutions and tensor products.
Similarly, $\Flat_{m,z}$ maps identity partitions to identity partitions, which implies
$\id_x \in \DD$ for all $x \in \colors$. Thus, $\DD$ is a category of spatial partitions.
Now, assume $\CC$ is rigid. The \Cref{prop:existance-duals} shows that there 
exist spatial partitions $r \in \CC(1, z\overline{z})$ and $s \in \CC(1, \overline{z}z)$
satisfying the conjugate equations for $z$ and $\overline{z}$. Since $\Flat_{m,z}$ is fully faithful,
there exist $\widetilde{r} \in \DD(1, \circ\bullet)$ and $\widetilde{s} \in \DD(1, \bullet\circ)$
with $\Flat_{m,z}(\widetilde{r}) = r$ and $\Flat_{m,z}(\widetilde{s}) = s$. 
These satisfy the conjugate equations for $\circ$ and $\bullet$ such that $\DD$ is also rigid. 
Finally, assume $\CC$ is $\nn$-graded. Since $\Flat_{m,z}$ only moves levels in consecutive groups of size $m$ to 
the front, it follows immediately that $\DD$ is $(\nn\cdots\nn)$-graded.
\end{proof}

\section{From spatial partitions to quantum groups}\label{sec:spatial-partition-QGs}

In the following, we use spatial partitions from the previous section to define spatial partition quantum groups.
We then describe these quantum groups in terms of universal $C^*$-algebras and show that the free orthogonal quantum groups $O^+(F_\sigma)$ 
can be constructed from our new base partitions. Finally, we show that all resulting quantum groups are invariant
under permuting levels, which implies that our new base partitions yield the same class 
of quantum groups as in~\cite{cebron16}.

Our definition of spatial partition quantum groups follows 
the one in~\cite{cebron16}, which generalizes the definition of easy quantum groups in~\cite{banica09,tarrago17}.
However, we discuss additional technicalities that arise when combining colors with our new base partitions.

\subsection{Spatial partition quantum groups}

In order to define spatial partition quantum groups, we need to associate a linear operator $T^{(\nn)}_p$ to each spatial 
partition $p \in \PP^{(m)}$. However, we first have to introduce some 
additional notation.

\begin{definition}
Let $p \in \PP^{(m)}(x, y)$ be a spatial partition
and $\ii_1, \dots, \ii_{\abs{x}}, \jj_{1}, \dots, \jj_{\abs{y}} \in \N^m$ be multi-indices of the form 
$\ii_k = (i_{k,1}, \dots, i_{k,m})$ for $1 \leq k \leq \abs{x}$
and $\jj_k = (j_{k,1}, \dots, j_{k,m})$ for $1 \leq k \leq \abs{y}$.
Label each upper point $(k, \ell)$ of $p$ with the index $i_{k,\ell}$ and 
each lower point $(\abs{x} + k, \ell)$ with the index $j_{k,\ell}$.
Then define 
\[
  {(\delta_p)}^{\ii_1, \dots, \ii_{\abs{x}}}_{\jj_{1}, \dots, \jj_{\abs{y}}}
  := \begin{cases}
    1 & \text{if all points in each block have equal labels}, \\
    0 & \text{otherwise}.
  \end{cases}
\]
\end{definition}

\begin{example}
Consider the spatial partition
$
  p \ := \ \partition[3d]{
    \line{0}{0}{0}{0}{1}{1}
    \line{0}{1}{0}{0}{0}{1}
    \line{1}{0}{0}{1}{0.35}{0}
    \line{1}{0}{1}{1}{0.35}{1}
    \line{1}{0.35}{0}{1}{0.35}{1}
    \point{0}{0}{0}{white}
    \point{0}{0}{1}{white}
    \point{1}{0}{0}{black}
    \point{1}{0}{1}{black}
    \point{0}{1}{0}{white}
    \point{0}{1}{1}{white}
  }
  \ \in \ \PP^{(2)}(\circ, {\circ}{\bullet})
$
and the indices $\ii_1 := (i_{1,1}, i_{1,2}), \jj_1 := (j_{1,1}, j_{1,2}), \jj_2 := (j_{2,1}, j_{2,2}) \in \N^2$. Then
\[
  {(\delta_p)}^{\ii_1}_{\jj_1,\jj_2}
  \ = \
  \partition[3d]{
    \line{0}{0}{0}{0}{1}{1}
    \line{0}{1}{0}{0}{0}{1}
    \line{1}{0}{0}{1}{0.35}{0}
    \line{1}{0}{1}{1}{0.35}{1}
    \line{1}{0.35}{0}{1}{0.35}{1}
    \label{0}{-0.2}{0}{\scalebox{3}{$j_{1,1}$}}
    \label{0}{-0.2}{1}{\scalebox{3}{$j_{1,2}$}}
    \label{1}{-0.2}{0}{\scalebox{3}{$j_{2,1}$}}
    \label{1}{-0.2}{1}{\scalebox{3}{$j_{2,2}$}}
    \label{0}{1.25}{0}{\scalebox{3}{$i_{1,1}$}}
    \label{0}{1.25}{1}{\scalebox{3}{$i_{1,2}$}}
  }
  \ = \
  \delta_{i_{1,1} j_{1,2}} \delta_{i_{1,2} j_{1,1}} \delta_{j_{2,1} j_{2,2}},
\]
where $\delta_{k\ell}$ denotes the usual Kronecker delta.
\end{example}

Now, we can assign linear maps to spatial partitions as follows.

\begin{definition}
Let $p \in \PP^{(m)}(x, y)$ be a $\nn$-graded spatial partition.
Then define the linear map 
\[
  T_p^{(\nn)} \colon {(\C^{\nn})}^{\otimes x} \to {(\C^{\nn})}^{\otimes y},
  \quad
  {\big(T_p^{(\nn)}\big)}^{\ii_1, \dots, \ii_\ell}_{\jj_1, \dots, \jj_k} := {\big(\delta_p\big)}^{\jj_1, \dots, \jj_k}_{\ii_1, \dots, \ii_\ell},
\]
where $k := \abs{x}$, $\ell := \abs{y}$ and the coordinates are with respect to the canonical bases 
\[
  {(e_{\jj_1}^{x_1} \otimes \dots \otimes e_{\jj_k}^{x_k})}_{\jj_1,\dots,\jj_k\in[\nn]}, \quad 
  {(e_{\ii_1}^{y_1} \otimes \dots \otimes e_{\ii_\ell}^{y_\ell})}_{\ii_1,\dots,\ii_\ell\in[\nn]}
\]
of ${(\C^{\nn})}^{\otimes x}$ and ${(\C^{\nn})}^{\otimes y}$ respectively.
\end{definition}

A spatial partition quantum group is a compact matrix quantum group where
the intertwiner spaces between tensor powers of its fundamental representation $u^\circ$ and its conjugate representation $u^\bullet$ are spanned 
by linear maps associated with a category of spatial partitions. 

In contrast to~\cite{cebron16}, the conjugate representation
$u^\bullet$ will no longer be given by $\overline{u^\circ}$ but depend on the corresponding category of spatial partitions. Further, we include
a possible change of basis in the definition such that the notion of spatial partition quantum group is 
compatible with isomorphisms of compact matrix quantum groups.

\begin{definition}\label{def:spatial-partition-qg}
Let $G$ be a compact matrix quantum group 
with fundamental representation $u^\circ$ on $V$
and unitary representation $u^\bullet$ on $\overline{V}$.
Then $G$ is a \textit{spatial partition quantum group} if there 
exists a $\nn$-graded rigid category of spatial partitions $\CC \subseteq \PP^{(m)}$
and a unitary $Q\colon \C^{\nn} \to V$, such that
\begin{align*}
  \Hom(u^x, u^y) = \Span\big\{ Q^{\otimes y} \cdot T^{(\nn)}_p \cdot {(Q^{-1})}^{\otimes x} \mid p \in \CC(x, y) \big\}
  \quad 
  \forall x, y \in \colors.
\end{align*}
\end{definition}

Since the category $\CC$ contains a pair of spatial partitions
solving the conjugate equations, it follows that $u^\bullet$ is conjugate to $u^\circ$.
Before we discuss how this conjugate representation $u^\bullet$ depends on $\CC$, we first show that for every rigid 
category of spatial partitions, there exists a corresponding spatial partition quantum group. 
This requires the following proposition, which states that the mapping $p \mapsto T^{(\nn)}_p$ is almost functorial.

\begin{proposition}\label{prop:Tp-ops}
Let $\nn := (n_1, \dots, n_m) \in \N^m$ and $p, q \in \PP^{(m)}$ be $\nn$-graded spatial partitions. Then
\begin{enumerate}
\item $T^{(\nn)}_p \otimes T^{(\nn)}_q = T^{(\nn)}_{p \otimes q}$
\item ${\big(T^{(\nn)}_p\big)}^* = T^{(\nn)}_{p^*}$
\item $T^{(\nn)}_p \cdot T^{(\nn)}_q = N^{\alpha} \cdot T^{(\nn)}_{pq}$, where $N := n_1 \cdots n_m$ and $\alpha \in \N$ denotes the number of removed loops when composing $p$ and $q$ (if the composition is defined).
\end{enumerate}
\end{proposition}
\begin{proof}
See~\cite[Proposition 3.7]{cebron16}.
\end{proof}

Using the previous proposition, we can now show the existence of spatial partition quantum groups using 
Woronowicz Tannaka-Krein duality.

\begin{theorem}\label{thm:spatial-tannaka-krein}
Let $\CC \subseteq \PP^{(m)}$ be a $\nn$-graded rigid category of spatial partitions.
Then there exists a unique compact matrix quantum group $G_{\nn}(\CC)$ 
with fundamental representation $u^\circ$ on $V := \C^{\nn}$ and a unitary representation $u^\bullet$ on 
$\overline{V}$, such that 
\[
  \Hom(u^{x}, u^{y}) = \Span\big\{ T^{(\nn)}_p \mid p \in \CC(x, y) \big\}
  \quad 
  \forall x, y \in \colors.
\]
\end{theorem}
\begin{proof}
Define the linear subspaces
\[
  \widehat{\CC}(x, y) := \Span\big\{ T_p^{(\nn)} \mid p \in \CC(x, y) \big\}
  \subseteq B(V^{\otimes x}, V^{\otimes y})
\]
for all $x, y \in \colors$. Then \Cref{prop:Tp-ops} implies that 
$\widehat\CC$ is closed under composition, involution and tensor products 
as in \Cref{def:rep-category}.
Further, one checks that  
$T_{\id_x}^{(\nn)} = \id_{{V}^{\otimes x}} \in \widehat{\CC}(x, x)$
for all $x \in \colors$.
Since $\CC$ is rigid, there exists a pair of spatial partitions
$r \in \CC(1, \circ\bullet)$ and $s \in \CC(1, \bullet\circ)$ satisfying the conjugate 
equations. By \Cref{prop:Tp-ops}, the operators $R := T^{(\nn)}_r$ and 
$S := T^{(\nn)}_s$ satisfy again the conjugate equations for $u^\circ$ and $u^\bullet$ because 
no closed loops are removed when composing $r$ and $s$.
Thus, the subspaces $\widehat{\CC}(x, y)$ form 
a two-colored representation category in the sense of \Cref{def:rep-category}
and the statement follows by applying \Cref{thm:tannaka-krein-duality}.
\end{proof}

\begin{remark}
As discussed in \Cref{sec:tannaka-krein}, the quantum group $G_{\nn}(\CC)$
from the previous theorem is uniquely determined by the following universal property. 
Let $H$ be a compact matrix quantum group with fundamental representation $w^\circ$ on $V$ and 
unitary representation $w^\bullet$ on $\overline{V}$. If there exists
a unitary $Q\colon \C^{\nn} \to V$, such that 
\[
  \Hom(u^{x}, u^{y}) \subseteq {(Q^{-1})}^{\otimes y} \cdot \Hom(w^x, w^y) \cdot {Q}^{\otimes x}
  \quad 
  \forall x, y \in \colors, 
\]
then $H$ is a subgroup of $G$ via the map $u^\circ \mapsto Q^{-1} w^\circ Q$.
\end{remark}

\begin{remark}\label{rem:spatial-equiv-G}
It follows directly from the definition that a compact matrix quantum group is a 
spatial partition quantum group if and only if it is equivalent to 
$G_{\nn}(\CC)$ for a $\nn$-graded rigid category of spatial partitions $\CC$.
\end{remark}

Next, we come back to the conjugate representations $u^\bullet$ of 
spatial partition quantum groups and show that $u^\bullet$
can be expressed in terms of $u^\circ$ and unitaries $F^{(\nn)}_\sigma$.

\begin{definition}
Let $\sigma \in S_m$ be a $\nn$-graded permutation. 
Then define the linear map $F_\sigma^{(\nn)} \colon \C^\nn \to \C^\nn$ by
\[
  F^{(\nn)}_\sigma(e_{i_1} \otimes \dots \otimes e_{i_m}) = e_{i_{\sigma^{-1}(1)}} \otimes \dots \otimes e_{i_{\sigma^{-1}(m)}}
  \quad 
  \forall (i_1, \dots, i_m) \in [\nn].
\]
\end{definition}

Note that the mapping $\sigma \mapsto F^{(\nn)}_\sigma$
defines a unitary representation of $\nn$-graded permutations on $\C^{\nn}$, i.e.\ we have 
$F^{(\nn)}_{\sigma} \cdot F^{(\nn)}_{\tau} = F^{(\nn)}_{\sigma\tau}$
and ${\big(F^{(\nn)}_{\sigma}\big)}^* = {\big(F^{(\nn)}_{\sigma}\big)}^{-1} = F^{(\nn)}_{\sigma^{-1}}$
for all $\nn$-graded $\sigma, \tau \in S_m$. Additionally, we write
$\overline{F}^{(\nn)}_{\sigma}$ for the corresponding conjugate operator.

Using the linear maps $F^{(\nn)}_\sigma$, we can now describe the conjugate representations
$u^\bullet$ of spatial partition quantum groups of the form $G_{\nn}(\CC)$.

\begin{proposition}\label{prop:spatial-sub-U}
Let $\CC \subseteq \PP^{(m)}$ be a $\nn$-graded rigid category of spatial partitions 
containing duality partitions $\sigma_{\circ\bullet}$ and $\sigma^{-1}_{\bullet\circ}$. 
Then the conjugate representation $u^\bullet$ of $G_{\nn}(\CC)$ is given by 
$u^\bullet = \overline{F}^{(\nn)}_\sigma \, \overline{u^\circ} \, {\big(\overline{F}^{(\nn)}_\sigma \big)}^{-1}$.
\end{proposition}
\begin{proof}
Denote with $u^\circ$ the fundamental representation of $G_{\nn}(\CC)$ and 
define $r := \sigma_{\circ\bullet}$ and $s := \sigma_{\bullet\circ}^{-1}$.
Since $r$ and $s$ are a pair of duality partitions, \Cref{prop:Tp-ops} implies that $R := T^{(\nn)}_{r}$ and $S := T^{(\nn)}_{s}$ satisfy the conjugate equations for $u^\circ$ and $u^\bullet$.
Hence, \Cref{prop:u-black-formula} yields $u^\bullet = F \overline{u^\circ} F^{-1}$ with 
$F$ defined by $F^\ii_\jj = R^{\jj,\ii}$ for all $\ii,\jj \in [\nn]$. 
Next, we compute
\[
  F^\ii_\jj = {\big(T^{(\nn)}_{r}\big)}^{\jj,\ii} = {(\delta_r)}_{\jj,\ii} 
  = \delta_{j_1 i_{\sigma(1)}} \cdots \delta_{j_m i_{\sigma(m)}} = {\big(F_{\sigma}^{(\nn)}\big)}^\ii_\jj
\]
for all $\ii := (i_1,\dots,i_m), \jj := (i_1,\dots,j_m) \in [\nn]$. Thus, $F = F_{\sigma}^{(\nn)}$. 
\end{proof}

The previous proposition shows that the conjugate representation $u^\bullet$ of $G_{\nn}(\CC)$ 
is uniquely determined by the category $\CC$. Since any spatial partition quantum group 
is equivalent to a $G_{\nn}(\CC)$, it follows
that the conjugate representation $u^\bullet$ of a general spatial partition quantum group is uniquely determined 
by the corresponding category $\CC$ and the change of basis $Q$.

Since the linear maps $\overline{F}^{(\nn)}_\sigma$ are unitary, the 
previous proposition additionally implies that 
the representation $\overline{u^\circ} = {\big(\overline{F}^{(\nn)}_\sigma \big)}^{-1} u^\bullet \overline{F}^{(\nn)}_\sigma$
is again unitary.
This shows that $\overline{u^\circ}$ is a conjugate representation of $u^\circ$ and that
every spatial partition quantum group is a subgroup of $U^+_{\nn}$. 
We will come back to this fact in
\Cref{sec:permuting-levels}, where we show the intertwiners spaces of a spatial partition
quantum groups are still spanned by spatial partition when we replace $u^\bullet$ by $\overline{u^\circ}$.

\subsection{Presentations of spatial partitions quantum groups}

Next, we describe the $C^*$-algebras $C(G_\nn(\CC))$ from the previous section as universal $C^*$-algebras.
If the category of spatial partitions $\CC$ is generated by a finite set of spatial partitions, then we 
obtain a universal $C^*$-algebras defined by a finite set of relations. In particular, this
allows us to show that the quantum groups $O^+(F^{(\nn)}_\sigma)$ are spatial partition quantum groups.

We begin by extending the notation $G_{\nn}(\CC)$ to arbitrary sets of spatial partitions that generate a rigid category
of spatial partitions.

\begin{definition}\label{def:generator-presentation}
Let $\CC_0 \subseteq \PP^{(m)}$ be a set of $\nn$-graded spatial partitions
with duality partitions $\sigma_{\circ\bullet}, \sigma^{-1}_{\bullet\circ} \in \langle\CC_0\rangle$.
Denote with $A$ the universal unital $C^*$-algebra generated by the entries of a 
matrix $u := {(u^\ii_\jj)}_{\ii, \jj \in [\nn]}$ and relations
\begin{enumerate}
\item $u^\circ := u$ and $u^\bullet := \overline{F}^{(\nn)}_{\sigma} \overline{u} {\big(\overline{F}^{(\nn)}_{\sigma}\big)}^{-1}$ are unitary,
\item 
$T^{(\nn)}_p \, u^{x} = u^{y} \, T^{(\nn)}_p$
for all $x, y \in \colors$ and $p \in \CC_0 \cap \PP^{(m)}(x, y)$.
\end{enumerate}
Then $G_{\nn}(\CC_0) := (A, u)$ is the \textit{spatial partitions quantum groups defined by $\CC_0$}.
\end{definition}

First, note that the compact matrix quantum group $G_{\nn}(\CC_0)$ from the previous definition
is well-defined. Indeed, $A$ is generated by the elements $u^\ii_\jj$ and by the first relations,
$u$ is unitary and $\overline{u}$ is invertible. Further, one checks
that both relations are compatible with the comultiplication, 
see also \Cref{def:universal-qgs} and \cite{woronowicz88}.

Next, we show that the $C^*$-algebra $C(G_\nn(\CC_0))$ and 
thus the quantum group $G_\nn(\CC_0)$ does not depend on the particular choice 
of duality partitions $\sigma_{\circ\bullet}, \sigma^{-1}_{\bullet\circ} \in \langle\CC_0\rangle$.
However, we first need the following lemma.

\begin{lemma}\label{lemma:Tp-sub-intertwiners}
Let $\CC_0 \subseteq \PP^{(m)}$ be a set of 
$\nn$-graded spatial partitions generating a rigid category of spatial partitions $\CC := \angles{\CC_0}$ and
consider the quantum group $G_{\nn}(\CC_0)$. Then 
\[
    \Span\big\{ T^{(\nn)}_p \mid p \in \CC(x, y) \big\}
    \subseteq 
    \Hom(u^{x}, u^{y})
    \quad 
    \forall x, y \in \colors.
\]
\end{lemma}
\begin{proof}
\Cref{prop:reprs-monoidal-cat} and \Cref{prop:Tp-ops} imply that 
$T^{(\nn)}_{p \otimes q}$, $T^{(\nn)}_{p^*}$ and $T^{(\nn)}_{pq}$
are also intertwiners of $G_{\nn}(\CC_0)$ for all (composable) $p, q \in \CC_0$. Thus, we obtain inductively that
$T^{(\nn)}_p$ is an intertwiner for all $p \in \CC$.
The statement follows by applying the linearity of intertwiner spaces from \Cref{prop:reprs-monoidal-cat}.
\end{proof}

Now, consider the quantum group $G_{\nn}(\CC_0)$ and assume $\CC := \langle \CC_0\rangle$
contains a different pair of duality partitions $r := \tau_{\circ\bullet}$ and $s := \tau^{-1}_{\bullet\circ}$ for $\tau \in S_m$.
Then the linear maps $R := T^{(\nn)}_r$ and $S := T^{(\nn)}_s$ are intertwiners by the previous proposition
and the proof of \Cref{prop:spatial-sub-U} shows that $u^\bullet = \overline{F}^{(\nn)}_{\tau} \overline{u} {\big(\overline{F}^{(\nn)}_{\tau}\big)}^{-1}$.
Therefore, $u^\bullet$ and the $C^*$-algebra $C(G_{\nn}(\CC_0))$ do not depend on the particular choice of duality partitions
in \Cref{def:generator-presentation}.

If $\CC_0 = \CC$ is already a category of spatial partitions, then the notation $G_\nn(\CC_0)$ of the previous definition
and $G_\nn(\CC)$ of \Cref{thm:spatial-tannaka-krein} overlap.
However, the following proposition shows that both quantum groups coincide in this case and that 
more generally $G_{\nn}(\langle \CC_0 \rangle)$ and $G_{\nn}(\CC_0)$ agree.

\begin{proposition}\label{prop:C-C0-iso}
Let $\CC_0 \subseteq \PP^{(m)}$ be a set of 
$\nn$-graded spatial partitions generating a rigid category of spatial partitions $\CC := \angles{\CC_0}$.
Then the compact matrix quantum groups 
$G_{\nn}(\CC)$ and $G_{\nn}(\CC_0)$
from \Cref{thm:spatial-tannaka-krein} and \Cref{def:generator-presentation} are isomorphic.
\end{proposition}
\begin{proof}
Denote with $u$ and $w$ the fundamental representations of $G_\nn(\CC_0)$ and 
$G_\nn(\CC)$ respectively.
Then \Cref{lemma:Tp-sub-intertwiners}
and the universal property of $G_{\nn}(\CC)$ yield the inclusion $G_{\nn}(\CC_0) \subseteq G_\nn(\CC)$ via $w \mapsto u$.
Conversely, $\CC_0 \subseteq \CC$ and \Cref{prop:spatial-sub-U} show that 
$w$ satisfies all defining relations of $u$.
Thus, the universal property of $C(G_{\nn}(\CC_0))$ yields the 
inverse inclusion $G_{\nn}(\CC) \subseteq G_{\nn}(\CC_0)$ via $u \mapsto w$. 
Therefore, both quantum groups are isomorphic.
\end{proof}

Using the description of spatial partition quantum groups as universal $C^*$-algebras, we can 
now construct the free orthogonal quantum groups $O^+(F^{(\nn)}_\sigma)$ from spatial partitions.

\begin{proposition}\label{prop:universal-are-spatial}
Let $\sigma \in S_m$ be a $\nn$-graded permutation
and $\CC := \big\langle\sigma_{\circ\bullet}, \sigma^{-1}_{\bullet\circ}, {\partIdBW}^{(m)}\big\rangle$.
Then $G_{\nn}(\CC)$ and $O^+(F^{(\nn)}_\sigma)$ are isomorphic.
\end{proposition}
\begin{proof}
Denote with $u =: u^\circ$ the fundamental
representation of $G_{\nn}(\CC)$ and with $w =: w^\circ$ the 
fundamental representation of $O^+(F^{(\nn)}_\sigma)$.
Further, denote with $\iota \colon \overline{\C^{\nn}} \to \C^{\nn}$ the linear isomorphism 
given by $\iota(\overline{e_i}) = e_i$ for all $i \in [\nn]$.
Then $u$ is a unitary
and \Cref{prop:spatial-sub-U} yields that its conjugate representation $u^\bullet$ 
is given by $u^\bullet = \overline{F}_\sigma^{(\nn)} \, \overline{u} \, {\big(\overline{F}_\sigma^{(\nn)}\big)}^{-1}$.
Since $\iota = T^{(\nn)}_{{\partIdBW}^{(m)}} \in \Hom(u^\bullet, u^\circ)$, we have $\iota u^\bullet = u^\circ \iota$, which implies 
\[
  u = u^\circ = \iota u^\bullet \iota^{-1} 
  = {\big(\iota \overline{F}_\sigma^{(\nn)}\big)} \, \overline{u} \, {\big(\iota \overline{F}_\sigma^{(\nn)}\big)}^{-1}
  = {\big({F}_\sigma^{(\nn)} \iota\big)} \, \overline{u} \, {\big({F}_\sigma^{(\nn)} \iota\big)}^{-1}.
\]
Thus, the universal property of $C(O^+(F^{(\nn)}_\sigma))$ yields
the inclusion $G_{\nn}(\CC) \subseteq O^+(F^{(\nn)}_\sigma)$ via $w \mapsto u$.

Conversely, define $w^\bullet := \iota^{-1} w \iota$, which is unitary since both
$w$ and $\iota$ are unitary. Further, we have
$\iota w^\bullet = \iota \iota^{-1} w \iota = w^\circ \iota$,
which shows $T^{(\nn)}_{{\partIdBW}^{(m)}} = \iota \in \Hom(w^\bullet, w^\circ)$.
By the definition of $O^+(F_\sigma^{(\nn)})$ and the argument before, we also have
\[
  w^\bullet = \iota^{-1} w \iota = \big( \iota^{-1} {F}_\sigma^{(\nn)} \iota \big) \, \overline{w} \, {\big(\iota^{-1} {F}_\sigma^{(\nn)} \iota \big)}^{-1} 
  = \overline{F}_\sigma^{(\nn)} \, \overline{w^\circ} \, {\big( \overline{F}_\sigma^{(\nn)} \big)}^{-1}. 
\]
Hence, \Cref{prop:u-bullet-impl-conj} shows that $w^\circ$ and $w^\bullet$ are conjugate via intertwiners
$R \in \Hom(1, w^{\circ\bullet})$ and $S \in \Hom(1, w^{\bullet\circ})$ given by 
\[
  R^{\ii,\jj} = {(\overline{F}_\sigma^{(\nn)})}^\jj_\ii, \quad 
  S^{\ii,\jj} = {({F}_{\sigma^{-1}}^{(\nn)})}^\jj_\ii \qquad
  \forall \ii, \jj \in [\nn]. 
\]
As in proof of \Cref{prop:spatial-sub-U}, this is equivalent to $R = T^{(\nn)}_{\sigma_{\circ\bullet}}$
and $S = T^{(\nn)}_{\sigma_{\bullet\circ}^{-1}}$. Therefore,
$w$ satisfies all the defining relations of $C(G_\nn(\CC_0))$ with $\CC_0 := \{\sigma_{\circ\bullet}, \sigma^{-1}_{\bullet\circ}, {\partIdBW}^{(m)}\}$.
Since $G_\nn(\CC_0) = G_\nn(\CC)$ by \Cref{prop:C-C0-iso},
the universal property of $C(G_{\nn}(\CC_0))$ yields 
the inverse inclusion $O^+(F^{(\nn)}_\sigma) \subseteq G_{\nn}(\CC)$ via $u \mapsto w$.
\end{proof}

Note that $O^+(F)$ and $O^+(QFQ^T)$ are isomorphic for 
any unitary $Q$, see for example~\cite{timmermann08}. 
Since $F_\tau^{(\nn)} F_\sigma^{(\nn)} {\big(F_\tau^{(\nn)}\big)}^T = F_{\tau \sigma \tau^{-1}}^{(\nn)}$, 
this implies that $O^+(F^{(\nn)}_{\sigma})$ and $O^+(F^{(\nn)}_{\tau\sigma\tau^{-1}})$
are isomorphic for all $\sigma, \tau \in S_m$. 
Thus, $O^+(F^{(\nn)}_{\sigma})$ depends only on the conjugacy class of 
$\sigma$.

\subsection{Permuting levels}\label{sec:permuting-levels}

In \Cref{sec:functor-perm}, we introduced the functor 
$\Perm_{\sigma,\tau}$ that permutes the levels of a spatial partition
depending on the color of the points.
In the following, we show that $\Perm_{\sigma,\tau}$ leave the quantum groups
$G_{\nn}(\CC)$ invariant. This implies that our 
new duality partitions yield the same class of spatial partition quantum groups as defined by Cébron-Weber in~\cite{cebron16}. 

The main work to prove this result will be done in the following
technical lemma about the linear maps 
$T_{\Perm_{\sigma,\tau}(p)}^{(\nn)}$.

\begin{lemma}\label{lem:perm-proof}
Let $\sigma, \tau \in S_m$ be $\nn$-graded permutations.
Consider a compact matrix quantum group with fundamental
representation $u^\circ$ on $\C^{\nn}$ and unitary representation $u^\bullet$ on $\overline{\C^{\nn}}$.
Define the unitary representations
$\hat{u}^\circ := u^\circ$ and $\hat{u}^\bullet := {\big(\overline{F}^{(\nn)}_{\tau\sigma^{-1}}\big)}^{-1} \, u^\bullet \, \overline{F}^{(\nn)}_{\tau\sigma^{-1}}$.
Then
\[
  T_{\Perm_{\sigma,\tau}(p)}^{(\nn)} \in \Hom(u^x, u^y)
  \ \Longleftrightarrow \ 
  {\big(F^{(\nn)}_\sigma\big)}^{\otimes y} \, T_p^{(\nn)} \, {\big({\big(F^{(\nn)}_\sigma\big)}^{-1}\big)}^{\otimes x} \in \Hom(\hat{u}^x, \hat{u}^y)
  \qquad 
\]
for all $p \in \PP^{(m)}(x, y)$.
\end{lemma}
\begin{proof}
Let $p \in \PP^{(m)}(x, y)$ and recall from \Cref{def:functor-perm}
that 
\[
  \Perm_{\tau,\rho}(p) := q^y_{\sigma,\tau} \cdot p \cdot {(q^x_{\sigma,\tau})}^{-1}.
\]
Define the linear operators $Q^z_{\sigma,\tau} := T^{(\nn)}_{q^z_{\sigma,\tau}}$ for all $z \in \colors$. Then 
\[
  {(Q^\circ_{\sigma,\tau})}^\ii_\jj = {(\delta_{q_{\sigma,\tau}^\circ})}^\jj_\ii = \delta_{j_1 i_{\sigma(1)}} \dots \delta_{j_m i_{\sigma(m)}} = {(F_\sigma^{(\nn)})}^\ii_\jj
\]
for all $\ii := (i_1, \dots, i_m), \jj := (j_1, \dots, j_m) \in [\nn]$. Hence, $Q^\circ_{\sigma,\tau} = F_\sigma^{(\nn)}$ and similarly one shows 
$Q^\bullet_{\sigma,\tau} = \overline{F}_\tau^{(\nn)}$. Since $F_\sigma^{(\nn)}$ and $\overline{F}_\tau^{(\nn)}$ come from a unitary representation, it follows
inductively that $Q^z_{\sigma, \tau}$ defines a unitary representation of $\nn$-graded permutations in $S_m \times S_m$.
Thus, we can write
\[
  T^{(\nn)}_{\Perm_{\tau,\rho}(p)} 
  = Q^y_{\sigma,\tau} \cdot T^{(\nn)}_{p} \cdot {(Q^x_{\sigma,\tau})}^{-1}
  = {(Q^y_{\id,\tau\sigma^{-1}} \cdot Q^y_{\sigma,\sigma})} \cdot T^{(\nn)}_{p} \cdot {(Q^x_{\id,\tau\sigma^{-1}} \cdot Q^x_{\sigma,\sigma})}^{-1}.
\]
Hence, $T^{(\nn)}_{\Perm_{\tau,\rho}(p)} \in \Hom(u^x, u^y)$ if and only if
\begin{multline*}
  (Q^y_{\id,\tau\sigma^{-1}} \cdot Q^y_{\sigma,\sigma}) \cdot T^{(\nn)}_{p} \cdot {(Q^x_{\id,\tau\sigma^{-1}} \cdot Q^x_{\sigma,\sigma})}^{-1} \cdot u^x \\ 
  = u^y \cdot {(Q^y_{\id,\tau\sigma^{-1}} \cdot Q^y_{\sigma,\sigma})} \cdot T^{(\nn)}_{p} \, {(Q^x_{\id,\tau\sigma^{-1}} \cdot Q^x_{\sigma,\sigma})}^{-1},
\end{multline*}
which is again equivalent to 
\begin{multline*}
  Q^y_{\sigma,\sigma} \cdot T^{(\nn)}_{p}  \cdot {(Q^x_{\sigma,\sigma})}^{-1}  \cdot {(Q^x_{\id,\tau\sigma^{-1}})}^{-1}  \cdot u^x  \cdot Q^x_{\id,\tau\sigma^{-1}} \\ 
  =  {(Q^y_{\id,\tau\sigma^{-1}})}^{-1} \cdot u^y \cdot Q^y_{\id,\tau\sigma^{-1}} \cdot {Q^y_{\sigma,\sigma}} \cdot T^{(\nn)}_{p} \cdot {(Q^x_{\sigma,\sigma})}^{-1}.
\end{multline*}
Since 
\[
  {(F_{\sigma}^{(\nn)})}^{\otimes z} = Q^z_{\sigma,\sigma}, \quad
  \hat{u}^z = {(Q^z_{\id,\tau\sigma^{-1}})}^{-1} u^z Q^z_{\id,\tau\sigma^{-1}}
  \qquad \forall z \in \colors,
\]
we conclude that
\[
  T_{\Perm_{\sigma,\tau}(p)}^{(\nn)} \in \Hom(u^x, u^y)
  \ \Longleftrightarrow \ 
  {\big(F^{(\nn)}_\sigma\big)}^{\otimes y} \, T_p^{(\nn)} \, {\big({\big(F^{(\nn)}_\sigma\big)}^{-1}\big)}^{\otimes x} \in \Hom(\hat{u}^x, \hat{u}^y).
\]
\end{proof}

Using the previous lemma, we can now show that the spatial partition functor
$\Perm_{\sigma,\tau}$ leaves the quantum groups $G_{\nn}(\CC)$ invariant.

\begin{theorem}
Let $\CC \subseteq \PP^{(m)}$ be a $\nn$-graded rigid category of spatial partitions
and $\sigma, \tau \in S_m$ be $\nn$-graded permutations. Then $G_{\nn}(\CC)$ and $G_{\nn}(\Perm_{\sigma,\tau}(\CC))$ are isomorphic.
\end{theorem}
\begin{proof}
Define $\DD := \Perm_{\sigma,\tau}(\CC)$ and denote with $u^\circ$, $u^\bullet$ and $w^\circ$, $w^\bullet$
the fundamental representation and conjugate representation of $G_\nn(\CC)$
and $G_\nn(\DD)$ respectively. Since the functor $\Perm_{\sigma,\tau}$ is fully faithful and does not change colors, we have
\[
  \Hom(w^x, w^y) = \Span\big\{ T_p^{(\nn)} \mid p \in \DD(x, y) \big\}
                 = \Span\big\{ T_{\Perm_{\sigma,\tau}(p)}^{(\nn)} \mid p \in \CC(x, y) \big\}
\]
for all $x, y \in \colors$. Define $\hat{w}^\circ := w^\circ$ and $\hat{w}^\bullet := {\big(\overline{F}^{(\nn)}_{\tau\sigma^{-1}}\big)}^{-1} \, w^\bullet \, \overline{F}^{(\nn)}_{\tau\sigma^{-1}}$. Then \Cref{lem:perm-proof} yields
\[
   \Span\big\{ {\big(F^{(\nn)}_\sigma\big)}^{\otimes y} \, T_p^{(\nn)} \, {\big({\big(F^{(\nn)}_\sigma\big)}^{-1}\big)}^{\otimes x} \mid p \in \CC(x, y) \big\} \subseteq \Hom(\hat{w}^x, \hat{w}^y),
\]
which is equivalent to
\[
  \Hom(u^x, u^y)
  \subseteq 
  {\big({\big(F^{(\nn)}_\sigma\big)}^{-1}\big)}^{\otimes y} \cdot \Hom(\hat{w}^x, \hat{w}^y) \cdot {\big(F^{(\nn)}_\sigma\big)}^{\otimes x}.
\]
Thus, the universal property of $G_{\nn}(\CC)$ and $\widehat{w}^\circ = w^\circ$ yield the inclusion
\[
  G_{\nn}(\DD) \subseteq G_{\nn}(\CC), \quad u^\circ \mapsto {\big(F^{(\nn)}_\sigma\big)}^{-1} w^\circ F^{(\nn)}_\sigma.
\]
Since $\CC = \Perm_{\sigma^{-1},\tau^{-1}}(\DD)$, the same argument also yields the inverse inclusion
\[
  G_\nn(\CC) \subseteq G_\nn(\DD), \quad w^\circ \mapsto {\big(F^{(\nn)}_{\sigma^{-1}}\big)}^{-1} u^\circ F^{(\nn)}_{\sigma^{-1}}.
\]
Therefore, $G_\nn(\CC)$ and $G_\nn(\DD)$ are isomorphic.
\end{proof}

The previous theorem might be generally useful for determining the quantum groups 
associated with concrete categories of spatial partitions.
However, as an immediate consequence, we obtain that the class of spatial partition quantum groups 
defined with our new duality partitions agrees with the class of spatial partition quantum groups defined
by Cébron-Weber in~\cite{cebron16}.

\begin{corollary}\label{corr:assume-base-partitions}
Let $G$ be a spatial partition quantum group. Then $G$ is equivalent
to $G_{\nn}(\CC)$ for a $\nn$-graded rigid category of spatial partitions $\CC \subseteq \PP^{(m)}$
containing the duality partitions $\partPairWB^{(m)}$ and $\partPairBW^{(m)}$.
\end{corollary}
\begin{proof}
As discussed in \Cref{rem:spatial-equiv-G}, any spatial partition quantum group $G$ 
is equivalent to $G_\nn(\DD)$ for a $\nn$-graded rigid category of spatial partitions $\DD \subseteq \PP^{(m)}$
containing some duality partitions $\sigma_{\circ\bullet}$ and $\sigma_{\bullet\circ}^{-1}$. By the previous theorem,
this quantum group is again equivalent to $G_\nn(\CC)$, where $\CC := \Perm_{\id,\sigma^{-1}}(\DD)$ contains 
the duality partitions
\[
  \Perm_{\id,\sigma^{-1}}(\sigma_{\circ\bullet}) = \partPairWB^{(m)}, \qquad
  \Perm_{\id,\sigma^{-1}}(\sigma_{\bullet\circ}^{-1}) = \partPairBW^{(m)}.
\]
\end{proof}

In \Cref{sec:spatial-partition-QGs}, we have shown that for every spatial partition quantum group,
the representation $\overline{u}$ is unitary and conjugate to $u$. However, 
it remained open if we can choose $u^\bullet = \overline{u}$ and obtain again a spatial partition quantum group. 
The previous corollary now answers this positively.
Since $\partPairWB^{(m)} = (\id_{S_m})_{\circ\bullet}$, the previous corollary and
\Cref{prop:spatial-sub-U} show that any spatial partition quantum group is equivalent to 
a spatial partition quantum group with 
\[
    u^\bullet = \overline{F}_{\id_{S_m}}^{(\nn)} \cdot \overline{u} \cdot {\big( \overline{F}_{\id_{S_m}}^{(\nn)} \big)}^{-1} = \overline{u}.
\]
Thus, we can always choose $u^\bullet = \overline{u}$ without loss of generality in the setting of spatial partition quantum groups.

\section{Projective spatial partition quantum groups}\label{sec:proj-version}

In the following, we use the functor $\Flat_{m,z}$
from \Cref{sec:functor-flat} to show that 
the class of spatial partition quantum groups is closed under taking projective versions 
or more generally taking tensor powers of $u^\circ$ and $u^\bullet$.
We then use this result  
to compute the spatial partition quantum groups corresponding
to the categories $P_2^{(m)}$ of spatial pair partition on $m$ levels.
Further, we show that all projective versions of easy quantum groups are again spatial partition quantum groups.
A result of Gromada~\cite{gromada22a} then allows us to describe these projective versions explicitly in terms of generators and relations,
if the underlying quantum group has a degree of reflection two.

\subsection{Tensor powers of spatial partition quantum groups}\label{sec:spatial-tensor-powers}

If $w$ is a unitary representation of a compact matrix quantum group, then 
its matrix coefficients can be used to define a new compact matrix quantum group.
Now, assume $u$ is the fundamental representation of a quantum group $G$ and 
$\overline{u}$ is also unitary. Then $w := u \otop \overline{u}$ 
yields the projective version of $G$ denoted by $PG$. If $G$ is a classical matrix group, 
then $PG$ corresponds exactly to the classical projective version 
\[
    PG := G / (G \cap \{ \lambda I \mid \lambda \in \C\}).
\]
For more information on the projective versions of compact matrix quantum groups, see for example~\cite{banica10b, gromada22a, banica23}.

More generally, we can use any powers of the fundamental representation
$u =: u^\circ$ and a unitary representation $u^\bullet$ to construct new quantum groups.

\begin{definition}
Let $z \in \colors$ with $\abs{z} \geq 1$ and $G$ be a compact matrix quantum group with fundamental representation $u^\circ$ and
unitary representation $u^\bullet$.
Then define $G^{z} := (A, w)$, where 
$A \subseteq C(G)$ is the $C^*$-algebra generated by the matrix coefficients of $w := u^{z}$. 
\end{definition}

Note that if $u^\bullet = \overline{u^\circ}$, then the projective version $PG$ is
exactly given by $G^{\circ\bullet}$. Further, we have shown in the previous sections that 
$\overline{u^\circ}$ is unitary 
for all spatial partition quantum groups and that we can choose $u^\bullet = \overline{u^\circ}$ without loss of generality. 
Thus, the projective version of every spatial partition quantum group $G$ is well-defined and given by $G^{\circ\bullet}$.

In the following, we show that if $G$ is a spatial partition
quantum group with corresponding category $\CC \subseteq \PP^{(m)}$, 
then $G^z$ is again a spatial partition quantum group for all $z \in \colors$ with $\abs{z} \geq 1$. 
Moreover, the corresponding category of $G^z$ is given by $\Flat^{-1}_{m,z}(\CC)$, where $\Flat_{m,z}$
is the functor from \Cref{sec:functor-flat}. Here, the fact that $\Flat^{-1}_{m,z}(\CC)$ is again 
a rigid category of spatial partition relies on our new duality partition, see \Cref{prop:flat-category}.

Since the functor $\Flat^{-1}_{m,z}$ does not preserve colors, 
our first step is to show that 
the representations $w^x$ and $u^{\Flat_{m,z}(x)}$ are equivalent for all $x \in \colors$, 
where $w^\circ := u^z$ and $w^\bullet := \overline{u^z}$.

\begin{lemma}\label{lem:u-flat-equiv}
Let $z \in \colors$ with $d:= \abs{z} \geq 1$ and consider a compact matrix quantum group with fundamental
representation $u^\circ$ on $V$ and unitary representation $u^\bullet$ on $\overline{V}$. 
Define $w^\circ := u^z$, $w^\bullet := \overline{u^z}$ and the unitaries
\begin{align*}
  Q_\circ   \colon {V}^{\otimes z} \to {V}^{\otimes z}, \quad v_1 \otimes \dots \otimes v_d \mapsto v_1 \otimes \dots \otimes v_d \quad \forall v_1, \dots, v_d \in V, \\
  Q_\bullet \colon \overline{V^{\otimes z}} \to {V}^{\otimes \overline{z}}, \quad \overline{v_1 \otimes \dots \otimes v_d} \mapsto \overline{v_d} \otimes \dots \otimes \overline{v_1} \quad \forall v_1, \dots, v_d \in V.
\end{align*}
Then
\[
  Q_x := \bigotimes_{i=1}^{\abs{x}} Q_{x_i} \in \Hom\big(w^x, u^{\Flat_{m,z}(x)}\big)
\]
for all $x \in \colors$ and $m \in \N$.
\end{lemma}
\begin{proof}
First, consider the case $x = \circ$. Then
$w^\circ = u^{z} = u^{\Flat_{m,z}(\circ)}$
and $Q_\circ = \id_{V}^{\otimes z} \in \Hom(u^{z}, u^{z})$.
Next, consider the case $x = \bullet$
and choose an orthonormal basis ${(v_i)}_{i\in I}$
of $V$ inducing canonical bases for $\overline{V^{\otimes z}}$
and ${V}^{\otimes \overline{z}}$ indexed by $i_1, \dots, i_d \in I$.
With respects to these bases, the matrix entries of $w^\bullet = \overline{u^{z}}$ are given by
\[
  {(\overline{u^{z}})}^{i_1, \dots, i_d}_{j_1, \dots, j_d}
  =
  {\big({(u^{z})}^{i_1, \dots, i_d}_{j_1, \dots, j_d}\big)}^*
  = 
  {\big({(u^{z_1})}^{i_1}_{j_1} \cdots {(u^{z_d})}^{i_d}_{j_d}\big)}^*
  =  
  {(u^{\overline{z_d}})}^{i_d}_{j_d} \cdots {(u^{\overline{z_1}})}^{i_1}_{j_1}
\]
for all $i_1, \dots, i_d, j_1, \dots, j_d \in I$.
On the other hand, we have $u^{\Flat_{m,z}(\bullet)} = u^{\overline{z}}$ and
\[
  {(u^{\overline{z}})}^{i_1, \dots, i_d}_{j_1, \dots, j_d}
  =
  {(u^{\overline{z_d}})}^{i_1}_{j_1} \cdots {(u^{\overline{z_d}})}^{i_d}_{j_d}.
\]
Thus, performing a change of basis using $Q_\bullet$ yields 
$\overline{u^{z}} = Q_\bullet^{-1} \cdot u^{\overline{z}} \cdot Q_\bullet$, which is equivalent
to $Q_\bullet \in \Hom(w^\bullet, u^{\Flat_{m,z}(\bullet)})$.

Finally, consider an arbitrary $x \in \colors$. 
The previous computations show that $Q_{x_i} \in \Hom(w^{x_i}, u^{\Flat_{m,z}(x_i)})$
for all $1 \leq i \leq \abs{x}$.
Since intertwiner spaces are closed under tensor products and $\Flat_{m,z}$ is a functor,
this implies
\[
  Q_x := \bigotimes_{i=1}^{\abs{x}} Q_{x_i} \in \Hom\big(\otop_{i=1}^{\abs{x}} w^{x_i}, \otop_{i=1}^{\abs{x}} u^{\Flat_{m,z}(x_i)}\big) = 
  \Hom\big(w^{x}, u^{\Flat_{m,z}(x)}\big).
\]
\end{proof}

\begin{remark}\label{rem:Gz-conj-repr}
As before, let $z \in \colors$ with $\abs{z} \geq 1$ and $G$ be a compact matrix quantum group with fundamental representation $u^\circ$
and unitary representation $u^\bullet$. 
Since both $u^{\overline{z}}$ and $Q_\bullet$ are unitary, the previous lemma also shows that 
\[
  w^\bullet := \overline{u^{z}} = Q_\bullet^{-1} \cdot u^{\overline{z}} \cdot Q_\bullet
\]
is unitary. Therefore, \Cref{prop:u-bullet-impl-conj} implies that $w^\bullet = \overline{u^{z}}$ is conjugate to the fundamental
representation $w^\circ = u^z$ of $G^z$. 
\end{remark}

\begin{remark}\label{rem:equiv-rewrite-hom}
Let $u$ and $v$ be unitary representations of a compact matrix quantum group $G$ and assume there exists 
a unitary $Q \in \Hom(u, v)$. Then $u = Q^{-1} v Q$ and we can write 
\[
  \Hom(u, w) = \Hom(v, w) \cdot Q, 
  \qquad 
  \Hom(w, u) = Q^{-1} \cdot \Hom(w, v)
\]
for any unitary representation $w$ of $G$.
\end{remark}

Next, we show that for any spatial partition $p$
the linear maps $T_{\Flat_{m,z}(p)}^{(\nn)}$ and $T_p^{(\nn\cdots\nn)}$ 
agree up to a change of basis. Here, $\nn\cdots\nn$ denotes the $d$-fold repetition of $\nn$
as in \Cref{prop:flat-category}.

\begin{lemma}\label{lem:flat-inter}
Let $\nn \in \N^m$ and $z \in \colors$ with $d := \abs{z} \geq 1$.
Then 
\[
  Q_{y}^{-1} \cdot T_{\Flat_{m,z}(p)}^{(\nn)} \cdot Q_{x} = {(S^{-1})}^{\otimes y} \cdot T_p^{(\nn\cdots\nn)} \cdot S^{\otimes x}
\]
for all $p \in \PP^{(m \cdot d)}(x, y)$, where $Q_x, Q_y$ are defined in \Cref{lem:u-flat-equiv} for $V = \C^n$ and
\[
  S \colon {(\C^{\nn})}^{\otimes z} \to \C^{\nn \cdots \nn},
  \quad 
  e_{\ii_1}^{z_1} \otimes \dots \otimes e_{\jj_d}^{z_d} \mapsto e_{\ii_1} \otimes \dots \otimes e_{\ii_d}
  \quad 
  \forall \ii_1, \dots, \ii_d \in [\nn].
\]
\end{lemma}
\begin{proof}
Let $p \in \PP^{(m \cdot d)}(x, y)$. Define $k := \abs{x}$, $\ell := \abs{y}$ and let $\ii_1, \dots, \ii_{\ell}, \jj_1, \dots, \jj_{k} \in [\nn \cdots \nn]$ be of the form 
\[\begin{aligned}
  \ii_{\alpha} &:= (i_{\alpha, 1, 1}, \dots, i_{\alpha, 1, m}, \dots, i_{\alpha, d, 1}, \dots, i_{\alpha, d, m}), \\
  \jj_{\alpha} &:= (j_{\alpha, 1, 1}, \dots, j_{\alpha, 1, m}, \dots, j_{\alpha, d, 1}, \dots, j_{\alpha, d, m}).
\end{aligned}\]
For $1 \leq \beta \leq d$, define the slices
\[  
  \ii^{(\beta)}_{\alpha} := \begin{cases}
      (i_{\alpha, \beta, 1}, \ldots, i_{\alpha, \beta, m}) & \text{if $x_\alpha = \circ$}, \\
      (i_{\alpha, \beta, m}, \ldots, i_{\alpha, \beta, 1}) & \text{if $x_\alpha = \bullet$},
  \end{cases}
  \quad
  \jj^{(\beta)}_{\alpha} := \begin{cases}
    (j_{\alpha, \beta, 1}, \ldots, j_{\alpha, \beta, m}) & \text{if $y_\alpha = \circ$}, \\
    (j_{\alpha, \beta, m}, \ldots, j_{\alpha, \beta, 1}) & \text{if $y_\alpha = \bullet$}.
\end{cases}
\]
Then by the definition of $S$, we have 
\[
  {\big({(S^{-1})}^{\otimes y} \cdot T_p^{(\nn\cdots\nn)} \cdot S^{\otimes x}\big)}_{\jj_1, \dots, \jj_k}^{\ii_1, \dots, \ii_\ell}
  =
  {\big(T_p^{(\nn\cdots\nn)}\big)}_{\jj_1, \dots, \jj_k}^{\ii_1, \dots, \ii_\ell}
  =
  {(\delta_p)}_{\ii_1, \dots, \ii_\ell}^{\jj_1, \dots, \jj_k}.
\]
On the other hand, the definition of $Q_x$ and $Q_y$ yields that
\[
  {\big(Q_{y}^{-1} \cdot T_{\Flat_{m,z}(p)}^{(\nn)} \cdot Q_{x}\big)}_{\jj_1, \dots, \jj_k}^{\ii_1, \dots, \ii_\ell}
  =
  {\big(T_{\Flat_{m,z}(p)}^{(\nn)}\big)}_{\jj_1^{(1)}, \dots, \jj_1^{(d)}, \dots, \jj_k^{(1)}, \dots, \jj_k^{(d)}}^{\ii_1^{(1)}, \dots, \ii_1^{(d)}, \dots, \ii_\ell^{(1)}, \dots, \ii_\ell^{(d)}}
\]
which is given by
\[
  {\big(T_{\Flat_{m,z}(p)}^{(\nn)}\big)}_{\jj_1^{(1)}, \dots, \jj_1^{(d)}, \dots, \jj_k^{(1)}, \dots, \jj_k^{(d)}}^{\ii_1^{(1)}, \dots, \ii_1^{(d)}, \dots, \ii_\ell^{(1)}, \dots, \ii_\ell^{(d)}}
    =
    {(\delta_{\Flat_{m,z}(p)})}_{\ii_1^{(1)}, \dots, \ii_1^{(d)}, \dots, \ii_\ell^{(1)}, \dots, \ii_\ell^{(d)}}^{\jj_1^{(1)}, \dots, \jj_1^{(d)}, \dots, \jj_k^{(1)}, \dots, \jj_k^{(d)}}.
\]
The statement of the lemma now follows, since
\[
  {(\delta_p)}_{\ii_1, \dots, \ii_\ell}^{\jj_1, \dots, \jj_k}
  =
  {(\delta_{\Flat_{m,z}(p)})}_{\ii_1^{(1)}, \dots, \ii_1^{(d)}, \dots, \ii_\ell^{(1)}, \dots, \ii_\ell^{(d)}}^{\jj_1^{(1)}, \dots, \jj_1^{(d)}, \dots, \jj_k^{(1)}, \dots, \jj_k^{(d)}}
\]
by the definition of $\Flat_{m,z}$.
\end{proof}

Using the previous two lemmas, we can now prove our main theorem 
and show that quantum groups of the form $G_n(\CC)^z$
are again spatial partition quantum groups.

\begin{theorem}\label{thm:proj-version}
Let $\CC \subseteq \PP^{(m)}$ be a $\nn$-graded rigid category of spatial partitions and $z \in \colors$ with $\abs{z} \geq 1$. 
Then ${G_{\nn}(\CC)}^{z}$ is equivalent to $G_{\nn \cdots \nn}(\DD)$ with $\DD := \Flat_{m,z}^{-1}(\CC)$.
\end{theorem}
\begin{proof}
Denote with $u^\circ$, $u^\bullet$ and $\hat{w}^\circ$, $\hat{w}^\bullet$ the 
fundamental representation and conjugate representation of ${G_{\nn}(\CC)}$ and $G_{\nn \cdots \nn}(\DD)$ respectively.
Further, denote with $w^\circ := {(u^\circ)}^z$ the fundamental representation
of $G^z$ and with $w^\bullet := \overline{{(u^\circ)}^z}$ its conjugate representation, see \Cref{rem:Gz-conj-repr}.
Then \Cref{lem:u-flat-equiv} in combination with \Cref{rem:equiv-rewrite-hom} and fact that 
$\Flat_{m,z}$ is fully faithful imply that
\begin{align*}
  \Hom({w}^x, {w}^y)
  &= Q_y^{-1} \cdot \Hom(u^{\Flat_{m,z}(x)}, u^{\Flat_{m,z}(y)}) \cdot Q_x \\
  &= \Span\big\{ Q_y^{-1} \cdot T^{(\nn)}_{p} \cdot Q_x \mid p \in \CC(\Flat_{m,z}(x), \Flat_{m,z}(y)) \big\} \\
  &= \Span\big\{ Q_y^{-1} \cdot T^{(\nn)}_{\Flat_{m,z}(p)} \cdot Q_x \mid p \in \DD \big\}
\end{align*}
for all $x, y \in \colors$. By \Cref{lem:flat-inter}, we have 
\[
  Q_{y}^{-1} \cdot T_{\Flat_{m,z}(p)}^{(\nn)} \cdot Q_{x} = {(S^{-1})}^{\otimes y} \cdot T_p^{(\nn\cdots\nn)} \cdot S^{\otimes x}
\]
for a unitary $S \colon {(\C^{\nn})}^{\otimes z} \to \C^{\nn \dots \nn}$, which yields
\begin{align*}
  \Hom({w}^x, {w}^y)
  &= \Span\big\{ {(S^{-1})}^{\otimes y} \cdot T_p^{(\nn\cdots\nn)} \cdot S^{\otimes x} \mid p \in \DD(x, y) \big\} \\
  &= {(S^{-1})}^{\otimes y} \cdot \Hom({\hat{w}}^x, {\hat{w}}^y) \cdot S^{\otimes x}.
\end{align*} 
Therefore, ${G_{\nn}(\CC)}^{z}$ and $G_{\nn \dots \nn}(\DD)$ are equivalent.
\end{proof}

Since every spatial partition quantum group is equivalent to 
a quantum group of the form $G_{\nn}(\CC)$, we can use the following proposition to transfer the previous theorem to all spatial partition quantum groups.

\begin{proposition}
Let $G$ and $H$ be compact matrix quantum groups with fundamental representations
$u^\circ$, $\hat{u}^\circ$ and conjugate representations $u^\bullet$, $\hat{u}^\bullet$ respectively. 
If $G$ and $H$ are equivalent, then 
$G^z$ and $H^z$ are equivalent for all $z \in \colors$ with $\abs{z} \geq 1$.
\end{proposition}
\begin{proof}
Since $G$ and $H$ are equivalent, there exists a unitary $S$ such that
\[
  \Hom(u^x, u^y) = S^{\otimes y} \cdot \Hom(\hat{u}^x, \hat{u}^y) \cdot {(S^{-1})}^{\otimes x}
  \qquad \forall x,y \in \colors.
\]
Denote with $w^\circ := {(u^\circ)}^z$ and $\hat{w}^\circ := {(\hat{u}^\circ)}^z$ the fundamental
representations of $G^z$ and $H^z$. Then \Cref{rem:Gz-conj-repr}
shows that $w^\bullet := \overline{{(u^\circ)}^z}$ and $\hat{w}^\bullet := \overline{{(\hat{u}^\circ)}^z}$
are the corresponding conjugate representations.
Let $x, y \in \colors$ and define $\widetilde{x} := \Flat_{1,z}(x)$ and
$\widetilde{y} := \Flat_{1,z}(y)$. Then \Cref{lem:u-flat-equiv} and \Cref{rem:equiv-rewrite-hom} show that 
\begin{align*}
  \Hom(w^x, w^y)  
  & = Q_y^{-1} \cdot \Hom(u^{\widetilde{x}}, u^{\widetilde{y}}) \cdot Q_x \\
  & = Q_y^{-1} \cdot S^{\otimes \widetilde{y}} \cdot \Hom(\hat{u}^{\widetilde{x}}, \hat{u}^{\widetilde{y}}) \cdot {(S^{-1})}^{\otimes \widetilde{x}} \cdot Q_x \\
  & = Q_y^{-1} \cdot S^{\otimes \widetilde{y}} \cdot Q_{\widetilde{y}} \cdot \Hom(\hat{w}^x, \hat{w}^y) \cdot Q_{\widetilde{x}}^{-1} \cdot {(S^{-1})}^{\otimes \widetilde{x}} \cdot Q_x. 
\end{align*}
Thus, $G^z$ and $H^z$ are also equivalent.
\end{proof}

\begin{corollary}\label{corr:tensor-power-spatial}
Let $G$ be a spatial partition quantum group and $z \in \colors$ with $\abs{z} \geq 1$. Then 
$G^{z}$ is a spatial partition quantum group.
\end{corollary}
\begin{proof}
Since $G$ is a spatial partition quantum group, it is equivalent
to $G_\nn(\CC)$ for a $\nn$-graded rigid category of spatial partitions $\CC \subseteq \PP^{(m)}$.
Then the previous proposition shows that $G^z$ is equivalent to $G_\nn(\CC)^z$, which is a spatial partition quantum group by \Cref{thm:proj-version}.
\end{proof}

As a special case, it follows that 
the class of spatial partition quantum groups is closed under taking projective versions.
    
\begin{corollary}\label{corr:proj-spatial-closed}
Let $G$ be a spatial partition quantum group. Then $PG$ is a spatial partition quantum group.
\end{corollary}
\begin{proof}
By \Cref{corr:assume-base-partitions} and the discussion at the end of \Cref{sec:permuting-levels}, 
we can assume that the conjugate representation of $G$ is given by $u^\bullet = \overline{u^\circ}$.
In this case, $PG = G^{\circ\bullet}$, which a spatial partition quantum group by \Cref{corr:tensor-power-spatial}. 
\end{proof}

\subsection{Categories of all spatial pair partitions}\label{sec:P2}
As a first application of the previous theorem, we consider the quantum groups 
associated with the categories $\PP_2^{(m)}$ of all spatial pair partitions on $m$ levels, i.e.\ spatial partition 
on $m$ levels with every block of size two.

It is shown in \cite{banica09} that the
category $\PP_2^{(1)}$ of pair partitions on one level corresponds to the classical orthogonal group $O_n$. Further, the 
author shows in~\cite{faross22} that the category $\PP_2^{(2)}$ of spatial pair partition on two levels corresponds to the classical
projective orthogonal group $PO_n$. In the following, we
determine quantum groups in the remaining cases $m \geq 3$ and show that $G_{(n,\dots,n)}(P_2^{(m)})$ 
is equivalent to $O_n^{\circ \cdots \circ}$.

Note that in contrast to our definition, the categories in~\cite{banica09,faross22} are defined in terms of colorless partitions. 
However, in our setting, $\PP_2^{(m)}$ contains the spatial partition $\partIdBW^{(m)}$ that corresponds to the 
$C^*$-algebraic relations $u^{i}_j = {(u^i_j)}^*$ making the generators self-adjoint.
Thus, our quantum groups $G_{(n,\dots,n)}(\PP^{(m)}_2)$ agree with the corresponding orthogonal quantum groups in~\cite{banica09,faross22}. 
See also~\cite{tarrago17} for more on the relation between orthogonal and unitary easy quantum groups.

\begin{proposition}
  Let $m \geq 1$ and $n \in \N$. Then $G_{(n,\dots,n)}(\PP^{(m)}_2)$ is equivalent to $O_{n}^{z}$ with $z := \circ^m$.
\end{proposition}
\begin{proof}
Since the functor $\Flat_{1,z}$ does not change the size of blocks, pair partitions
are mapped to pair partitions and we have
\[
  \Flat_{1,z}^{-1}(\PP_2^{(1)}) = \PP_2^{(m)}.
\]
Since $G_n(\PP_2^{(1)}) = O_n$, \Cref{thm:proj-version} implies that 
$G_{(n,\dots,n)}(\PP_2^{(m)})$ is equivalent to $O_n^{z}$.
\end{proof}

By relaxing our notion of isomorphism of compact matrix quantum groups,
we can give a more explicit description of the resulting quantum groups $O_n^{\circ\cdots\circ}$. 

\begin{proposition}\label{prop:spatial-pairs-qgs}
  Let $n \in \N$ and $z := \circ^m$ with $m \geq 1$. Then 
  \[
    O_n^{z} = \begin{cases}
      O_n & \text{if $m$ is odd}, \\
      PO_n  & \text{if $m$ is even},
    \end{cases}
  \]
as compact quantum groups.
\end{proposition}
\begin{proof}
Denote with $u$ the fundamental representation of $O_n$.
First, we assume that $m$ is odd, i.e.\ $m = 2k + 1$. Then, consider the
inclusion $C(O_n^{z}) \hookrightarrow C(O_n)$, 
which is injective and respects the comultiplication.
Since $u$ is orthogonal, we compute
\begin{align*}
  \sum_{j_1, \dots, j_k \in [n]} u^{i_0}_{j_0} u^1_{j_1} u^1_{j_1} \dots u^1_{j_k} u^1_{j_k}
  &= u^{i_0}_{j_0} \underbrace{\bigg(\sum_{j_1 = 1}^n  u^1_{j_1} u^1_{j_1}\bigg)}_{=1} \dots \underbrace{\bigg(\sum_{j_k = 1}^n u^1_{j_k} u^1_{j_k} \bigg)}_{=1}
  &= u^{i_0}_{j_0} \in C(O_n^{z}) 
\end{align*}
for all $i_0, j_0 \in [n]$. Thus, the inclusion is also surjective and defines an isomorphism of compact quantum groups.

Next, we assume that $m$ is even, i.e.\ $m = 2k$. Then, we have an inclusion 
$C(O_n^{z}) \hookrightarrow C(PO_n)$ and compute
\[
  \sum_{j_2, \dots, j_k \in [n]} u^{i_0}_{j_0} u^{i_1}_{j_1} u^1_{j_2} u^1_{j_2} \dots u^1_{j_k} u^1_{j_k}
  = u^{i_0}_{j_0} u^{i_1}_{j_1} \underbrace{\bigg(\sum_{j_2 = 1}^n  u^1_{j_2} u^1_{j_2}\bigg)}_{=1} \dots \underbrace{\bigg(\sum_{j_k = 1}^n u^1_{j_k} u^1_{j_k} \bigg)}_{=1}
  = u^{i_0}_{j_0} u^{i_1}_{j_1}
\]
for all $i_0, j_0,i_1,j_1 \in [n]$. Thus, $u^{i_0}_{j_0} u^{i_1}_{j_1} \in C(O_n^{z})$, which shows that
the inclusion is surjective and defines an isomorphism of compact quantum groups.
\end{proof}

\subsection{Projective easy quantum groups}\label{sec:proj-easy-qg}

Consider an easy quantum group $G$. Then its projective version $PG$ might not be an easy quantum group in general.
However, since easy quantum groups are a subclass of spatial partition quantum groups, it follows
from \Cref{thm:proj-version} that projective versions of easy quantum groups are 
again spatial partition quantum groups. Thus, it is always possible to describe projective versions 
of easy quantum groups in terms of spatial partitions.

If $G$ corresponds to a category of partitions $\CC$, then the proof of \Cref{thm:proj-version} shows the category of its projective version 
$PG$ is given by $\Flat_{1,\circ\bullet}^{-1}(\CC)$.
This is particularly useful if the category $\CC$ can be characterized via 
an abstract property as in the case of spatial pair partitions in the previous section.
On the other hand, if the category $\CC$ is only given by a set of generators, then 
it might be difficult in general to obtain a generating set for the category of its projective version.

However, in the special case when $G$ is an orthogonal quantum group with degree of reflection two, Gromada~\cite{gromada22a} provides a result that allows 
us to convert generators of $\CC$ to generators of its projective version.

\begin{proposition}[\cite{gromada22a}]\label{prop:gromada}
Consider an orthogonal compact matrix quantum group $G \subseteq O^+(F)$ 
with fundamental representation on $\C^n$ and degree of reflection two. Consider an admissible generating set $S$ of
$\Rep(G)$ containing the duality morphisms. Then $\Rep(PG)$ is generated by 
$S$ and $\id_{\C^n} \otimes S \otimes \id_{\C^n}$.
\end{proposition}

Here, an orthogonal 
compact matrix group with fundamental representation $u$ has \textit{degree of reflection two}, if 
\[
  \Hom(u^{\otop k}, u^{\otop \ell}) = \{ 0 \} \quad \forall k, \ell \in \N,\, \text{$k$ + $\ell$ even}.
\]  
Further, a linear map $T \in S$ is called \textit{admissible}, if $T$ is of the form 
$T \colon {(\C^n)}^{\otimes k} \to {(\C^n)}^{\otimes \ell}$ with $k$ and $\ell$ even. 
Note that we can always use the duality morphisms to bring any intertwiner of a degree of reflection two 
quantum group into admissible form. See~\cite{gromada22a} or~\cite{neshveyev13} for further details.

Next, we reformulate \Cref{prop:gromada} in the setting of spatial partition quantum groups, before 
we apply it to concrete examples of easy quantum groups.

\begin{proposition}
Let $n \in \N$ and $\CC \subseteq \PP^{(1)}$ be a category of spatial partitions generated
by ${\partIdBW}$ and a set $\CC_0 \subseteq \PP^{(1)}$ such that
\begin{enumerate}
\item $\CC_0$ contains the partition $\partPair$,
\item every $p \in \CC_0$ has only white points,
\item if $p \in \CC_0$ has upper colors $x$ and lower colors $y$, then $\abs{x}$ and $\abs{y}$ are even.
\end{enumerate}
Then $PG_n(\CC)$ is equivalent to $G_{(n,n)}(\DD)$, where $\DD \subseteq \PP^{(2)}$
is generated by ${\partIdBW}^{(2)}$, $\Flat_{1,\circ\bullet}^{-1}(\CC_0)$ and $\Flat_{1,\circ\bullet}^{-1}(\id_{\circ} \otimes \CC_0 \otimes \id_{\circ})$.
\end{proposition}
\begin{proof}
Since $\CC$ contains the partitions ${\partIdBW}$ and $\partPair$, it follows that 
$\CC$ also contains the partitions $\partPairWB$ and $\partPairBW$. Thus, \Cref{prop:universal-are-spatial} implies that 
$G_n(\CC) \subseteq O_n^+$ is an orthogonal quantum group 
and we have $\overline{u} = u$ after identify $\C^n$ with $\overline{\C^n}$. 
Moreover, it follows from \Cref{def:generator-presentation} that the intertwiners
$S:= \{ T^{(n)}_p \mid p \in \CC_0 \}$ generate the category $\Rep(G)$ in the sense of~\cite{gromada22a}.
Therefore, we can apply \Cref{prop:gromada} and obtain that the representation category 
$\Rep(PG)$ is generated $S$ and $\id_{\C^n} \otimes S \otimes \id_{\C^n}$ 
in the sense of~\cite{gromada22a}. 
The proof \Cref{thm:proj-version} and in particular \Cref{lem:flat-inter} now show 
that we can use the functor $\Flat_{1,\circ\bullet}$ translate these new generators back to 
spatial partitions. Further, we have to add again the spatial partition ${\partIdBW}^{(2)}$
to obtain an orthogonal quantum group. Thus, $PG_n(\CC)$ is equivalent to $G_{(n,n)}(\DD)$ 
where $\DD$ is generated by ${\partIdBW}^{(2)}$, $\Flat^{-1}_{1,\circ\bullet}(\CC_0)$
and $\Flat_{1,\circ\bullet}^{-1}(\id_\circ \otimes \CC_0 \otimes \id_\circ)$.
\end{proof}

Now, we can apply the previous proposition to easy quantum group with degree of reflection two.
\Cref{fig:table-gens} presents the result for easy quantum groups described in~\cite{raum16},
see also the appendix of~\cite{volz23}. It contains the generators for the corresponding categories of spatial partitions 
as well as their projective versions. Note that the spatial partitions ${\protect\partIdBW}$ and ${\protect\partIdBW}^{(2)}$ are included implicitly.

\begin{figure}
{\def\arraystretch{1.5}
\begin{center}
\begin{tabular}{c|c|c}
  easy quantum group & generators $\CC_0$ & projective generators $\DD_0$ \\ 
  \hline
  $O_n$        & \partPair, \partCross                          & \partPairUnflat, \partCrossUnflat, \partPairUnflatIds, \partCrossUnflatIds \\
  $O_n^*$      & \partPair, \partTrippleCross                   & \partPairUnflat, \partTrippleCrossUnflat, \partPairUnflatIds, \partTrippleCrossUnflatIds \\
  $O_n^+$      & \partPair                                      & \partPairUnflat, \partPairUnflatIds \\
  $H_n$        & \partPair, \partFour, \partCross               & \partPairUnflat, \partFourUnflat, \partCrossUnflat, \partPairUnflatIds, \partFourUnflatIds, \partCrossUnflatIds \\
  $H_n^*$      & \partPair, \partFour, \partTrippleCross        & \partPairUnflat, \partFourUnflat, \partTrippleCrossUnflat, \partPairUnflatIds, \partFourUnflatIds, \partTrippleCrossUnflatIds \\
  $H_n^+$      & \partPair, \partFour                           & \partPairUnflat, \partFourUnflat, \partPairUnflatIds, \partFourUnflatIds \\
  $S_n'$       & \partPair, \partSingles, \partFour, \partCross & \partPairUnflat, \partSinglesUnflat, \partFourUnflat, \partCrossUnflat, \partPairUnflatIds, \partSinglesUnflatIds, \partFourUnflatIds, \partCrossUnflatIds \\
  $S_n'^+$     & \partPair, \partSingles, \partFour             & \partPairUnflat, \partSinglesUnflat, \partFourUnflat, \partPairUnflatIds, \partSinglesUnflatIds, \partFourUnflatIds \\
  $B_n'$       & \partPair, \partSingles, \partCross            & \partPairUnflat, \partSinglesUnflat, \partCrossUnflat, \partPairUnflatIds, \partSinglesUnflatIds, \partCrossUnflatIds \\
  $B_n'^+$     & \partPair, \partCrossSingles                   & \partPairUnflat, \partCrossSinglesUnflat, \partPairUnflatIds, \partCrossSinglesUnflatIds \\
  $B_n^{\# *}$ & \partPair, \partSingles, \partTrippleCross     & \partPairUnflat, \partSinglesUnflat, \partTrippleCrossUnflat, \partPairUnflatIds, \partSinglesUnflatIds, \partTrippleCrossUnflatIds \\
  $B_n^{\# +}$ & \partPair, \partSingles                        & \partPairUnflat, \partSinglesUnflat, \partPairUnflatIds, \partSinglesUnflatIds \\
\end{tabular}
\end{center}}
\caption{Generators for categories of easy quantum groups and their projective version. The spatial partitions ${\protect\partIdBW}$ and ${\protect\partIdBW}^{(2)}$ are omitted.}
\label{fig:table-gens}
\end{figure}

\Cref{fig:table-gens} does not 
include the quantum groups $S_n$, $B_n$, $S_n^+$ and $B_n^+$ since these 
do not have a degree of reflection two. However, their projective versions depend only on the 
even part of their respective categories of spatial partitions, which are exactly the categories
of $S_n'$, $B_n'$, $S_n'^+$ and $B_n'^+$.

We can now use \Cref{fig:table-gens} and \Cref{def:generator-presentation} to 
describe the $C^*$-algebras 
$A := C(G_{(n,n)}(\DD))$
of the previous projective quantum groups 
as universal unital $C^*$-algebras generated 
by a finite set of relations.
Since the fundamental representation of a projective easy quantum group is defined on $\C^n \otimes \overline{\C^n}$,
it follows that $A$ is generated by the 
coefficients of the matrix $u:= {\big(u^{i_1 i_2}_{j_1 j_2}\big)}$ with $i_1,i_2,j_1,j_2 \in [n]$.
Further, one checks that the category $\DD$ always contains
${\partition[3d]{
    \line{0}{0}{0}{0}{0.5}{0}
    \line{0}{0}{1}{0}{0.5}{1}
    \line{1}{0}{0}{1}{0.5}{0}
    \line{1}{0}{1}{1}{0.5}{1}
    \line{0}{0.5}{0}{1}{0.5}{1}
    \line{0}{0.5}{1}{1}{0.5}{0}
    \point{0}{0}{0}{white}
    \point{0}{0}{1}{white}
    \point{1}{0}{0}{white}
    \point{1}{0}{1}{white}
  }} = (12)_{\circ\circ}$
such that the duality partitions of $\DD$ are given by $(12)_{\circ\bullet}$, $(12)_{\bullet\circ}$
Thus, we have 
$u^\bullet := \overline{F}_{(12)} \overline{u} \overline{F}_{(12)}^{-1}$,
which can be written using matrix coefficients as ${(u^\bullet)}^{i_1 i_2}_{j_1 j_2} = {(u^{i_2 i_1}_{j_2 j_1})}^*$.
The $C^*$-algebra $A$ then satisfies the relations 
\[
    u u^* = u^* u = 1, \qquad u^\bullet {(u^\bullet)}^* = {(u^\bullet)}^* u^\bullet = 1
\] 
making both $u$ and $u^\bullet$ unitary.
Further, the intertwiner $T^{(n,n)}_{{\partIdBW}^{(2)}}$ yields the relation $T^{(n,n)}_{{\partIdBW}^{(2)}} u^\bullet = u T^{(n,n)}_{{\partIdBW}^{(2)}}$
or equivalently 
${(u^{i_2 i_1}_{j_2 j_1})}^* = u^{i_1 i_2}_{j_1 j_2}$.
Finally, the $C^*$-algebra $A$ satisfies the relations 
$T_p^{(n,n)} u^{\otop k} = u^{\otop \ell} T_p^{(n,n)}$ for all $p \in \DD_0$.
For the partitions given in \Cref{fig:table-gens}, these relations can be explicitly
written as follows, where again all indices are quantified over $[n]$.

\begingroup
\allowdisplaybreaks
\begin{alignat*}{8}
  \partSinglesUnflat\colon&\ & 1 &= \sum_{\ell_1, \ell_2} u^{i_1 i_2}_{\ell_1 \ell_2}
  & \qquad \quad
  \partSinglesUnflatIds\colon&\ & u^{i_1 i_4}_{j_1 j_2} &= \sum_{\ell_1, \ell_2} u^{i_1 i_2}_{j_1 \ell_1} u^{i_3 i_4}_{\ell_2 j_2}
  \\
  \partPairUnflat\colon&\ & \delta_{i_1 i_2} &= \sum_{\ell} u^{i_1 i_2}_{\ell \ell} 
  & \qquad \quad 
  \partPairUnflatIds\colon&\ & \delta_{i_2 i_3} u^{i_1 i_4}_{j_1 j_2} &= \sum_{\ell} u^{i_1 i_2}_{j_1 \ell} u^{i_3 i_4}_{\ell j_2}
  \\
  \partCrossUnflat\colon&\ & u^{i_2 i_1}_{j_1 j_2} &= u^{i_1 i_2}_{j_2 j_1}
  & \qquad \quad
  \partCrossUnflatIds\colon&\ & u^{i_1 i_3}_{j_1 j_2} u^{i_2 i_4}_{j_3 j_4} &= u^{i_3 i_4}_{j_1 j_3} u^{i_1 i_2}_{j_2 j_4}
  \\
  \partFourUnflat\colon&\ & \delta_{i_1 i_2} u^{i_1 i_1}_{j_1 j_2} &= \delta_{j_1 j_2} u^{i_1 i_2}_{j_1 j_1}
  & \qquad \quad
  \partFourUnflatIds\colon&\ & \delta_{i_2 i_3} u^{i_1 i_2}_{j_1 j_2} u^{i_2 i_4}_{j_3 j_4} &= \delta_{j_2 j_3} u^{i_1 i_2}_{j_1 j_2} u^{i_3 i_4}_{j_2 j_4}
  \\
  \partCrossSinglesUnflat\colon&\ & \sum_{k} u^{i_2 k}_{j_1 j_2} &= \sum_{\ell} u^{i_1 i_2}_{\ell j_1}
  & \qquad \quad
  \partCrossSinglesUnflatIds\colon&\ & \sum_{k} u^{i_1 i_3}_{j_1 j_2} u^{k i_4}_{j_3 j_4} &= \sum_{\ell} u^{i_1 i_2}_{j_1 \ell} u^{i_3 i_4}_{j_2 j_4}
  \\
  \partTrippleCrossUnflat\colon&\ & \delta_{i_1 i_4} u^{i_3 i_2}_{j_1 j_2}& = \sum_{\ell} u^{i_1 i_2}_{\ell j_2} u^{i_3 i_4}_{j_1 \ell}
  & \qquad \quad
  \partTrippleCrossUnflatIds\colon&\ & \delta_{i_2 i_5} u^{i_1 i_4}_{j_1 j_2} u^{i_3 i_6}_{j_3 j_4} &= \sum_{\ell} u^{i_1 i_2}_{j_1 \ell} u^{i_3 i_4}_{j_3 j_2} u^{i_5 i_6}_{\ell j_4}
\end{alignat*}
\endgroup

\bibliographystyle{amsplain}
\bibliography{paper}

\end{document}